\documentclass{amsart}

\usepackage[utf8]{inputenc}
\usepackage{fullpage}
\usepackage{stmaryrd}
\usepackage{xcolor}
\usepackage{tikz}
\usepackage{amsmath}
\usepackage{amsthm}
\usepackage{amssymb}
\usepackage{multicol}
\usepackage{hyperref}

\DeclareFontFamily{U}{mathb}{\hyphenchar\font45}
\DeclareFontShape{U}{mathb}{m}{n}{
<-6> mathb5 <6-7> mathb6 <7-8> mathb7
<8-9> mathb8 <9-10> mathb9
<10-12> mathb10 <12-> mathb12
}{}
\DeclareSymbolFont{mathb}{U}{mathb}{m}{n}
\DeclareMathSymbol{\llcurly}{\mathrel}{mathb}{"CE}
\DeclareMathSymbol{\ggcurly}{\mathrel}{mathb}{"CF}

\numberwithin{equation}{section}
\newtheorem{theorem}{Theorem}[section]
\newtheorem{corollary}[theorem]{Corollary}
\newtheorem{example}{Example}[section]
\newtheorem{lemma}[theorem]{Lemma}
\newtheorem{proposition}[theorem]{Proposition}

\theoremstyle{definition}
\newtheorem{definition}[theorem]{Definition}
\newtheorem{remark}[theorem]{Remark}

\newtheorem{algorithm}[theorem]{Algorithm}

\newcommand{\precc}[1]{\prec \hspace{-1mm}\prec #1}
\newcommand{\comments}[1]{}

\title{Map-compatible Decomposition of transport paths.}
\author{Qinglan Xia, Haotian Sun}
\address{
Department of Mathematics\\
University of California
at Davis\\
Davis, CA, 95616, USA }
\email{qlxia@math.ucdavis.edu, \ hatsun@ucdavis.edu}
\date{}


\begin{document}
\maketitle


\begin{abstract}
In the Monge-Kantorovich transport problem, the transport cost is expressed in terms of transport maps or transport plans, which play crucial roles there. 
A variant of the Monge-Kantorovich problem is the ramified (branching) transport problem that models branching transport systems via transport paths.
In this article, we showed that any cycle-free transport path between two atomic measures can be decomposed  into the sum of a 
map-compatible path and a 
plan-compatible path.
Moreover, we showed that each stair-shaped transport path can be decomposed into the difference of two map-compatible transport paths.
\end{abstract}

\section{Introduction}

The aim of this article is to provide a map-compatible decomposition of transport paths. 
In the well-known Monge-Kantorovich transport problem (see \cite{villani, Luigi, santambrogio} and references therein), the transport cost is expressed in terms of transport maps or transport plans. 
The existence of optimal transport maps, especially the Brenier map in the case of quadratic cost,
leads to numerous applications of optimal transportation theory in PDEs, Probability theory, Machine learning, etc. 
A variant of the Monge-Kantorovich transport problem is ramified (also called branched) optimal transportation (see \cite{xia1, book, xia2015motivations} and references therein). 
Through the lens of economy of scales, ramified optimal transportation 
aims at studying the branching structures that appeared in many living or non-living transport systems. 
In contrast to the classical Monge-Kantorovich transport problems, where the transport cost relies on transport maps and plans, the transport cost in the ramified transport problem is assessed across the entire branching transport system, referred to as transport paths. 

Since transport maps/plans only utilize information from the initial/target measures, knowing only transport maps/plans is insufficient for describing the transport cost that appears in ramified optimal transportation problem. 
In general, two transport paths (e.g. a ``Y-shaped" and a ``V-shaped" path) may have different transportation costs while sharing the same transport map/plan. 
Nevertheless, motivated by the significance of transport maps in the context of the Monge-Kantorovich problem, when a transport path is given, one may wonder if there exists a hidden transport map or plan that is compatible with this specific transport path.
This compatible transport map/plan tells one how the initial measure is distributed to the target measure via the given transport path.
For simplicity, this article only considers the case of atomic measures, deferring the exploration of other scenarios for future endeavors.
We want to provide a decomposition of transport paths such that each component in the decomposition is compatible with some transport map or transport plan. 

Roughly speaking, our main results are :
\begin{itemize}
\item {\bf Theorem \ref{thm: decomposition}:}
Every cycle-free
\footnote{A transport path $T$ is called cycle-free if there are no nonzero cycles on $T$. See Definition \ref{def: cycle-free current}.}
transport path $T$ can be decomposed as a sum of subcurrents $T=T_0+T_1+\cdots+T_N$ such that each $T_1,T_2,\cdots, T_N$ has a single target and $T_0$ has at most $\binom{N}{2}$ sources\footnote{Here, $N$ is the number of targets in the target measure $\mu^+$.}.

\item \textbf{Theorem \ref{thm: compatability}:} Every cycle-free transport path $T$ can be decomposed as a sum of subcurrents $T=T_\varphi+T_\pi$ such that $T_\varphi$ is compatible with some transport map $\varphi$ and $T_\pi$ is compatible with some transport plan $\pi$.

\item \textbf{Theorem \ref{thm: stair shaped induced transport maps}:} 
Every stair-shaped transport path $T$ can be decomposed as a sum of subcurrents $T=T_1+T_2$ such that both $T_1$ and $-T_2$ are compatible with some transport maps.
\end{itemize}
This article is organized as follows: We first recall in $\S 2$ some related concepts in geometric measure theory, the classical Monge-Kantorovich transport problem, and the ramified optimal transport problem. In particular, the {\it good decomposition} (i.e., Smirnov decomposition) of acyclic normal $1$-currents.

In general, the family of atoms (i.e.,
supporting curves) of a good decomposition is not necessarily linearly independent. 
This fact brings a non-unique representation of vanishing currents and causes a technical obstacle for the proof of Theorem \ref{thm: decomposition}. To overcome this, we generalize the notion of ``good decomposition" to ``better decomposition" (Definition \ref{def: better_decom}) of transport paths in $\S 3$. A better decomposition $\eta$ of a transport path $T$ prohibits combinations of any four supporting curves of $\eta$ to form a non-trivial cycle on the support of $T$.
We showed in Theorem \ref{thm: GoodS_ij} that any good decomposition of a transport path has a better decomposition that is absolutely continuous with respect to the original good decomposition.

In $\S 4$, we introduce the concept of cycle-free transport paths, which are transport paths with no non-trivial cycles on\footnote{The concept cycle-free is different to the concept ``acyclic" defined using subcurrents. As in Definition \ref{def: S_on_T}, a current $S$ is ``on" another current $T$ does not mean that $S$ is a subcurrent of $T$. When $S$ is on $T$, unlike being a subcurrent, it is possible that $S$ has a reverse orientation with $T$ on their intersections.} them.
Then, we use the ``better decomposition" achieved in Theorem \ref{thm: GoodS_ij} to give a decomposition of cycle-free transport paths, described in Theorem \ref{thm: decomposition}.

In $\S 5$, we consider the concept of ``compatibility" between transport paths and transport plans/maps.  
This concept was first introduced in \cite[Definition 7.1]{xia1} for cycle-free transport paths to describe whether a given transport plan is practically possible for transportation along the given transport path. 
We first generalize this concept, in a more general setting, to the compatibility between transport paths and transport plans/maps.
Then, using Theorem \ref{thm: decomposition}, we decompose a cycle-free transport path into the sum of a map-compatible path and a plan-compatible path, which gives Theorem \ref{thm: compatability}.

In $\S 6$, we proceed to study stair-shaped transport paths. 
We first show in Theorem \ref{thm: stairshaped-matrix} that each matrix\footnote{The size of this matrix may be countably infinite.} with non-negative entries can be transformed into a stair-shaped matrix, and in Algorithm \ref{rem: Stair shaped Algorithm}, we provide an algorithm for calculating the stair-shaped matrix. 
A transport path is called stair-shaped if it has a good decomposition that is represented by a stair-shaped matrix. 
A stair-shaped transport path is not necessarily cycle-free, but it still has a better decomposition.
Our main result for the section is Theorem \ref{thm: stair shaped induced transport maps}, which says that any stair-shaped transport path can be decomposed into the difference of two map-compatible transport paths. Note that some cycle-free transport paths are also stair-shaped. They can be decomposed not only as the sum of a map-compatible path and a plan-compatible path by Theorem \ref{thm: compatability}, but also as the sum of two map-compatible transport paths by Theorem \ref{thm: stair shaped induced transport maps}. 
We further investigate some sufficient conditions under which cycle-free transport paths are stair-shaped. An illustrating example is provided at the end.

\section{Preliminaries}
\subsection{Basic concepts in geometric measure theory} \ 

We first recall some related terminologies from geometric measure theory \cite{Simon, Lin}.
Suppose $U$ is an open set in $\mathbb{R}^{m}$ and $k \le m$, the set of all $C^\infty$ $k$-forms with compact support in $U$ is denoted by $\mathcal{D}^k(U).$ The dual space of $\mathcal{D}^k(U)$, $\mathcal{D}_k(U)$, is called the space of \textit{$k$-currents}. The mass of $T \in \mathcal{D}_k(U)$ is defined by
$$
\mathbf{M}(T)= \sup \{ T(\omega) \,:\,  \|\omega\| \le 1, \omega \in \mathcal{D}^k(U)\}.
$$
The \textit{boundary} of a current $T\in \mathcal{D}_k(U), \partial T \in \mathcal{D}_{k-1}(U)$ is defined by 
$$\partial T(\omega) = T(d\omega)\ \mathrm{for}\  \omega \in \mathcal{D}^{k-1}(U). $$

A set $M \subset \mathbb{R}^{m}$ is said to be countably \textit{$k$-rectifiable} if 
$$
M \subset M_0 \cup \left( \bigcup_{j=1}^\infty F_j(\mathbb{R}^k)\right),
$$
where $\mathcal{H}^k(M_0)=0$  under the $k$-dimensional Hausdorff measure $\mathcal{H}^k$, and $F_j: \mathbb{R}^k \to \mathbb{R}^{m}$ are Lipschitz functions for $j=1,2,\cdots.$ 
For any $T\in \mathcal{D}_k(U)$, we say that $T$ is a \textit{rectifiable $k$-current} if for each $\omega \in \mathcal{D}^k(U)$,
$$T(\omega) = \int_M \langle \omega(x) , \xi(x) \rangle \theta(x) \,d\mathcal{H}^k(x),$$
where $M$ is an $\mathcal{H}^k$-measurable countably $k$-rectifiable subset of $U$, $\theta(x)$ is a locally $\mathcal{H}^k$-integrable positive function, 
and $\xi: M \to \Lambda_k(\mathbb{R}^{m})$ is a $\mathcal{H}^k$-measurable function such that for $\mathcal{H}^k$-a.e. $x\in M$, 
$\xi(x) = \tau_1 \wedge \ldots \wedge \tau_k$, where $\tau_1 \ldots \tau_k$ is an orthonormal basis for the approximate tangent space $T_xM$. 
We will denote $T$ by $\underline{\underline{\tau}}(M,\theta,\xi)$.
When $T$ is a rectifiable $k$-current, its mass
$$
\mathbf{M}(T) = \int_M \theta(x) \,d\mathcal{H}^k(x).
$$

A current $T\in \mathcal{D}_k(U)$ is said to be \textit{normal} if 
$\mathbf{M}(T) + \mathbf{M}(\partial T) < \infty$.
In \cite{Paolini}, Paolini and Stepanov introduced the concept of subcurrents: For any $T, S\in \mathcal{D}_k(U)$, $S$ is called a \textit{subcurrent} of $T$ if
$$ \mathbf{M}(T-S) + \mathbf{M}(S) = \mathbf{M}(T).$$ A normal current $T\in \mathcal{D}_k(\mathbb{R}^m)$ is \textit{acyclic} if there is no non-trivial subcurrent $S$ of $T$ such that $\partial S =0$. 

In \cite{smirnov}, Smirnov showed that every acyclic normal $1$-current can be written as the weighted average of simple Lipschitz curves in the following sense.
Let $\Gamma$ be the space of $1$-Lipschitz curves $\gamma: [0,\infty) \to \mathbb{R}^m$, which are eventually constant. 
For $\gamma \in \Gamma$, we denote 
$$t_0(\gamma):=\sup \{t: \gamma \mbox{ is constant on } [0,t]  \}, \ 
t_\infty(\gamma):=\inf \{t: \gamma \mbox{ is constant on } [t,\infty)  \} ,$$
and $p_0(\gamma):= \gamma(0)$, $p_\infty(\gamma):= \gamma(\infty)=\lim_{t\to \infty} \gamma(t)$. A curve $\gamma \in \Gamma$ is simple if $\gamma(s)\not=\gamma(t)$ for every $t_0(\gamma) \le s < t \le t_\infty(\gamma)$.
For each simple curve $\gamma \in \Gamma$, we may associate it with the following rectifiable $1$-current, 
\begin{equation}
\label{eqn: I_gamma}
I_\gamma:= \underline{\underline{\tau}}\left( \mathrm{Im}(\gamma), \frac{\gamma'}{|\gamma'|}, 1  \right) ,
\end{equation}
where $Im(\gamma)$ denotes the image of $\gamma$ in $\mathbb{R}^m$.
\begin{definition}
\label{defn: good decomposition}
Let $T$ be a normal 1-current in $\mathbb{R}^m$ and let
$\eta$ be a finite positive measure on $\Gamma$
such that
\begin{equation}
\label{eqn: T_eta}
    T=\int_\Gamma I_\gamma \,d\eta(\gamma)
\end{equation}
in the sense that for every smooth compactly supported 1-form $\omega\in \mathcal{D}^1(\mathbb{R}^m)$, it holds that
\begin{equation}
    T(\omega)=\int_\Gamma I_\gamma(\omega) \,d\eta(\gamma).
\end{equation}
We say that $\eta$ is a \textit{good decomposition} of $T$ (see \cite{colombo}, \cite{colombo2020}, \cite{smirnov}) if $\eta$ is supported on non-constant, simple curves and
satisfies the following equalities:
\begin{itemize}
    \item[(a)] $\mathbf{M}(T)=\int_\Gamma \mathbf{M}(I_\gamma) d\eta(\gamma)=\int_\Gamma \mathcal{H}^1(Im(\gamma)) d\eta(\gamma)$;
    \item[(b)] $\mathbf{M}(\partial T)=\int_\Gamma \mathbf{M}(\partial I_\gamma) d\eta(\gamma)=2 \eta(\Gamma)$.
\end{itemize}
\end{definition}

Moreover, if $\eta$ is a good decomposition of $T$, the following statements hold \cite[Proposition 3.6]{colombo} :
\begin{itemize}
    \item 
    \begin{equation} \label{eqn: startingending measure}
    \mu^- = \int_\Gamma \delta_{\gamma (0)} \, d\eta(\gamma),\ 
    \mu^+ = \int_\Gamma \delta_{\gamma (\infty)} \, d\eta(\gamma).
    \end{equation}
    
    \item 
    If $T = \underline{\underline{\tau}}(M,\theta,\xi)$ is rectifiable, then
    \begin{equation}
    \label{eqn: theta(x)}
    \theta(x) = \eta (\{\gamma \in \Gamma : x \in \mathrm{Im}(\gamma) \} )
    \end{equation}
     for $\mathcal{H}^1$-a.e. $x \in M.$

    \item For every $\tilde{\eta} \leq  \eta$, the representation
    \[\tilde{T}=\int_\Gamma I_\gamma d\tilde{\eta}(\gamma)\]
    is a good decomposition of $\tilde{T}$. Moreover, if $T=\underline
    {\underline {\tau} }\left(M, \theta, \xi\right)$ is rectifiable, then $\tilde{T}$ can be written as $\tilde{T}=\underline
    {\underline {\tau} }(M, \tilde{\theta}, \xi )$ with
    \begin{equation}\label{eqn: tilde_theta_x}
         \tilde{\theta}(x)\le \min\{\theta(x), \tilde{\eta}(\Gamma)\} 
    \end{equation}
    for $\mathcal{H}^1$-a.e. $x\in M$.
     
\end{itemize}

In the following contexts, we adopt the notations: for any points $x,y \in \mathbb{R}^m$ and subset $A\subseteq \mathbb{R}^m$, denote
\begin{align}
    \Gamma_x & = \{\gamma \in \Gamma : x \in \mathrm{Im}(\gamma) \}, \label{gamma_{x}}\\
    \Gamma_{x,y}&=\{\gamma\in \Gamma: p_0(\gamma)=x, \ p_\infty(\gamma)=y\}, \label{gamma_{x,y}}
    \\
    \Gamma_{A,y}&=\{\gamma\in \Gamma: p_0(\gamma)\in A, \ p_\infty(\gamma)=y\}. \label{gamma_{A,y}}
\end{align}

\subsection{Basic concepts in optimal transportation theory}

\ 

We now recall some basic concepts in optimal transportation theory that are related to this article. Suppose $X$ is a convex compact subset of $\mathbb{R}^m$, the source $\mu^-$ and the target $\mu^+$ are two measures supported on $X$ of equal mass. 

\begin{itemize}
\item    
A map $\varphi: X\rightarrow X$
is called a transport map from 
$\mu^-$ to $\mu^+$ if the push-forward measure $\varphi_\# \mu^-=\mu^+$.
Let $Map(\mu^-, \mu^+)$ be the set of all transport maps from $\mu^-$ to $\mu^+$.

\item 
A Borel measure $\pi$ on $X\times X$ is called a transport plan from 
$\mu^-$ to $\mu^+$ if 
$(p_1)_\# \pi= \mu^-$ and $(p_2)_\#\pi=\mu^+$, where $p_1,p_2$ are respectively the first and the second orthogonal projection maps from $X\times X$ to $X$. Let $Plan(\mu^-, \mu^+)$ be the set of all transport plans from $\mu^-$ to $\mu^+$.

\item 
A rectifiable 1-current $T$ is called a transport path from $\mu^-$ to $\mu^+$ if its boundary $\partial T=\mu^+-\mu^-$. Let $Path(\mu^-, \mu^+)$ be the set of all transport paths from $\mu^-$ to $\mu^+$.
\end{itemize}

Let $C(x,y)$ be a non-negative Borel function, called the cost function, on $X\times X$. For any transport map $\varphi\in Map(\mu^-, \mu^+)$, the transport $C-$cost of $\varphi$ is
\[I_C(\varphi):=\int_X C(x, \varphi(x))d\mu^-(x).\]
Similarly, for any transport plan $\pi\in Plan(\mu^-, \mu^+)$,  the transport $C-$cost of $\pi$ is
\[J_C(\pi):=\int_{X\times X} C(x,y) d\pi(x,y).\]
For any transport path $T= \underline{\underline{\tau}}(M,\theta,\xi)\in Path(\mu^-, \mu^+)$, 
and any $0 \le \alpha < 1$, the transport $\mathbf{M}_\alpha$-cost of $T$ is
$$\mathbf{M}_\alpha(T) := \int_{M} \theta(x)^\alpha \,d\mathcal{H}^1.$$ 
The corresponding optimal transport problems are:
\begin{itemize}
\item {\bf Monge:} Minimize $I_C(\varphi)$ among all transport maps $\varphi\in Map(\mu^-,\mu^+)$;

\item {\bf Kantorovich:} Minimize $J_C(\pi)$ among all transport maps $\pi\in Plan(\mu^-,\mu^+)$;

\item {\bf Ramified/Branched:} Minimize $\mathbf{M}_\alpha(T)$ among all transport paths $T\in Path(\mu^-,\mu^+)$.
    
\end{itemize}
 For theoretic results such as existence/regularity and their applications, we refer to \cite{villani, Luigi, santambrogio}  for Monge-Kantorovich transport theory and \cite{xia1, book, xia2015motivations} for ramified/branched transportation.

In this article, we mainly focus on transportation between atomic measures.
Let 
\begin{equation}
\label{eqn: measures}
\mu^- = \sum_{i=1}^M m'_i \delta_{x_i} \text{ and }\mu^+ = \sum_{j=1}^N m_j \delta_{y_j} \text{ with } \sum_{i=1}^M m'_i = \sum_{j=1}^N m_j < \infty
\end{equation}
be two finite atomic measures on $X$ of equal mass with $M,N \in \mathbb{N} \cup \{\infty\}$. 
In this case, the above concepts have simplified forms:
\begin{itemize}
\item A transport map $\varphi\in Map(\mu^-,\mu^+)$ corresponds to a map $\varphi: \{1,2,\cdots, M\}\rightarrow \{1,2,\cdots, N\}$ such that for each $j=1,2,\cdots, N$,
    \[m_j=\sum_{i\in \varphi^{-1}(\{j\})}m_i'.\]
    The corresponding transport cost is
    \[I_C(\varphi)=\sum_{i=1}^M C(x_i, y_{\varphi(i)})m_i'.\]
\item A transport plan $\pi\in Map(\mu^-,\mu^+)$ corresponds to an $M\times N$ matrix $\pi=[\pi_{ij}]$ such that for each $i, j$, it holds that
    \[\sum_i \pi_{ij}=m_j \text{ and } \sum_j \pi_{ij}=m_i'.\]
    The corresponding transport cost is
    \[J_C(\pi)=\sum_{i=1}^M\sum_{j=1}^N c_{ij}\pi_{ij}\]
    where $c_{ij}=C(x_i,y_j)$.
\item A transport path $T\in Path(\mu^-,\mu^+)$ corresponds to a weighted directed graph $T$ consisting of a vertex set $V$, a directed edge set $E$ and a weight function $w: E \rightarrow (0, +\infty)$ such that $\{x_1,x_2,\ldots,x_M \}\cup\{y_1,y_2,\ldots,y_N \} \subseteq V $ and for any vertex $v\in V$, there is a balance equation: 
\[
\sum_{e\in E, e^-=v} w(e) \ = \sum_{e\in E, e^+=v} w(e) \  +\  \left\{
    \begin{array}{ll}
        \ \ m_i  &\mbox{if } v=x_i \mbox{ for some } i=1,\ldots,M  \\
        -n_j     &\mbox{if }v=y_j \mbox{ for some } j=1,\ldots,N \\
        \ \ 0    &\mbox{otherwise,}
    \end{array}
    \right.
\]
where $e^-$ and $e^+$ denote the starting and ending point of the edge $e\in E$.  
The corresponding transport $\mathbf{M}_{\alpha}$-cost of $T$ is
\[\mathbf{M}_{\alpha}(T)=\sum_{e\in E}w(e)^\alpha 
 length(e)\]
where the length $length(e)$ of the edge $e$ equals to $\mathcal{H}^1(e)$.
\end{itemize}

\section{Better decomposition of acyclic transport paths  }

Let $\mu^-$ and  $\mu^+$ be two atomic measures as given in (\ref{eqn: measures}), $T$ be an acyclic transport path from $\mu^-$ to $\mu^+$, and let $\eta$ be a good decomposition (i.e., Smirnov decomposition) of $T$. Observe that
as shown in the following example, with respect to the good decomposition $\eta$, it is possible that the family
\[\{I_\gamma:\; \eta(\{\gamma\})>0\}\]
is linearly dependent.

\begin{example}
\label{example: example 1}
Let $T$ be a transport path from $\mu^- = 4\delta_{x_1} + 2\delta_{x_2}$ to $\mu^+ = 3\delta_{y_1} + 3\delta_{y_2}$, as shown in the following figure

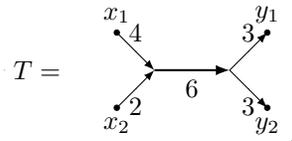
\begin{center}

\begin{tikzpicture} [>=latex]
\filldraw[black] (-1,1) circle (1pt) node[anchor=south]{$x_1$};
\filldraw[black] (-1,0) circle (1pt) node[anchor=north]{$x_2$};

\filldraw[black] (1,1) circle (1pt) node[anchor=south]{$y_1$};
\filldraw[black] (1,0) circle (1pt) node[anchor=north] {$y_2$};

\filldraw[black] (-0.75,0.75) circle (0pt) node[anchor=south] {$4$};

\filldraw[black] (-0.75,0.25) circle (0pt) node[anchor=north] {$2$};

\filldraw[black] (0.75,0.75) circle (0pt) node[anchor=south] {$3$};

\filldraw[black] (0.75,0.25) circle (0pt) node[anchor=north] {$3$};

\filldraw[black] (0,0.5) circle (0pt) node[anchor=north] {$6$};

\draw[->] (-1,1) -- (-0.5,0.5);
\draw[->] (-1,0) --  (-0.5,0.5);
\draw[thick,->] (-0.5,0.5) --  (0.5,0.5);
\draw[->]  (0.5,0.5) -- (1,1);
\draw[->]  (0.5,0.5) -- (1,0);    

\filldraw[black] (-2.5,0.5) circle (0pt) node[anchor=west]{$T=$};
\end{tikzpicture}.
\end{center}
 For each $(i,j)$, let 
 $\gamma_{x_i, y_j}$ be the corresponding curve from $x_i$ to $y_j$ on $T$:
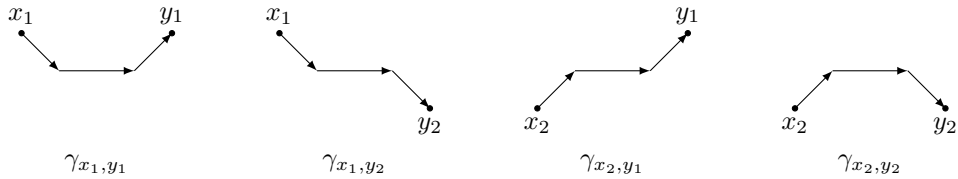
\begin{center}
\begin{tikzpicture} [>=latex]
\filldraw[black] (-1,1) circle (1pt) node[anchor=south]{$x_1$};

\filldraw[black] (1,1) circle (1pt) node[anchor=south]{$y_1$};

\filldraw[black] (0,-1) circle (0pt) node[anchor=south]{$\gamma_{x_1,y_1}$};

\draw[->] (-1,1) -- (-0.5,0.5);

\draw[->] (-0.5,0.5) --  (0.5,0.5);
\draw[->]  (0.5,0.5) -- (1,1);
  
\end{tikzpicture}
\qquad 
\begin{tikzpicture} [>=latex]
\filldraw[black] (-1,1) circle (1pt) node[anchor=south]{$x_1$};

\filldraw[black] (1,0) circle (1pt) node[anchor=north] {$y_2$};

\filldraw[black] (0,-1) circle (0pt) node[anchor=south]{$\gamma_{x_1,y_2}$};

\draw[->] (-1,1) -- (-0.5,0.5);

\draw[->] (-0.5,0.5) --  (0.5,0.5);

\draw[->]  (0.5,0.5) -- (1,0);    
\end{tikzpicture}
\qquad 
\begin{tikzpicture} [>=latex]

\filldraw[black] (-1,0) circle (1pt) node[anchor=north]{$x_2$};

\filldraw[black] (1,1) circle (1pt) node[anchor=south]{$y_1$};

\filldraw[black] (0,-1) circle (0pt) node[anchor=south]{$\gamma_{x_2,y_1}$};

\draw[->] (-1,0) --  (-0.5,0.5);
\draw[->] (-0.5,0.5) --  (0.5,0.5);
\draw[->]  (0.5,0.5) -- (1,1);

\end{tikzpicture}
\qquad 
\begin{tikzpicture} [>=latex]

\filldraw[black] (-1,0) circle (1pt) node[anchor=north]{$x_2$};

\filldraw[black] (1,0) circle (1pt) node[anchor=north] {$y_2$};

\filldraw[black] (0,-1) circle (0pt) node[anchor=south]{$\gamma_{x_2,y_2}$};

\draw[->] (-1,0) --  (-0.5,0.5);
\draw[->] (-0.5,0.5) --  (0.5,0.5);

\draw[->]  (0.5,0.5) -- (1,0);    
\end{tikzpicture}

\end{center}
Then
\[
\eta=2\delta_{\gamma_{x_1,y_1}}+2\delta_{\gamma_{x_1,y_2}}+\delta_{\gamma_{x_2,y_1}}+\delta_{\gamma_{x_2,y_2}}
\]
is a good decomposition of $T$. But
$$I_{\gamma_{x_1,y_1}}-I_{\gamma_{x_1,y_2}}-I_{\gamma_{x_2,y_1}}+I_{\gamma_{x_2,y_2}}$$
is the zero 1-current.

\end{example}
The linear dependence of the family $\{I_\gamma: \eta(\{\gamma\})>0\}$ brings a non-unique representation of vanishing currents and causes an obstacle later for the proof of Theorem \ref{thm: decomposition}. To overcome this, we introduce the concept of ``better decomposition" of $T$ as follows.

For each $i=1,2,\cdots, M$, $j=1,2,\cdots, N$, as given in (\ref{gamma_{x,y}}), let $\Gamma_{x_i, y_j}$ denote all $1$-Lipschitz curves in $\Gamma$ from $x_i$ to $y_j$. 
Also, for any finite positive measure $\eta$ on $\Gamma$, denote
\begin{equation}
\label{eqn: S_ij}
S_{i,j}(\eta):=
  \begin{cases}
       \frac{1}{\eta(\Gamma_{x_i,y_j})} \int_{\Gamma_{x_i,y_j} }  I_\gamma d\eta, & \text{ if }\eta(\Gamma_{x_i,y_j})>0\\
       0, &\text{ if }\eta(\Gamma_{x_i,y_j})=0.
  \end{cases}
\end{equation}

\begin{definition} \label{def: better_decom}
Let $T$ be a transport path from $\mu^-$ to $\mu^+$ where $\mu^-$ and $\mu^+$ are given in (\ref{eqn: measures}).
Suppose $\eta$ is a good decomposition of $T$. We say that $\eta$ is a {\it better decomposition}  
of $T$ if for any pairs $1\le i_1< i_2\le M$ and $1\le j_1<j_2\leq N$, 
$$S_{i_1,j_1} (\eta) - S_{i_1,j_2}  (\eta)- S_{i_2,j_1} (\eta) + S_{i_2,j_2} (\eta)=0$$ implies that 
$$\eta(\Gamma_{ x_{i_1},y_{j_1} }) =\eta(\Gamma_{x_{i_1},y_{j_2}})=\eta(\Gamma_{x_{i_2},y_{j_1}}) = \eta(\Gamma_{x_{i_2},y_{j_2}})=0.$$
\end{definition}

\begin{example}
In Example \ref{example: example 1},
\[
\eta=2\delta_{\gamma_{x_1,y_1}}+2\delta_{\gamma_{x_1,y_2}}+\delta_{\gamma_{x_2,y_1}}+\delta_{\gamma_{x_2,y_2}}
\]
is a good but not better decomposition of $T$. 
Indeed,
$$S_{1,1} (\eta) - S_{1,2}  (\eta)- S_{2,1} (\eta) + S_{2,2} (\eta)= 
I_{\gamma_{x_1,y_1}}-I_{\gamma_{x_1,y_2}}-I_{\gamma_{x_2,y_1}}+ I_{\gamma_{x_2,y_2}} =0,$$
but $$\eta (\Gamma_{x_1,y_1}) = 2, \eta (\Gamma_{x_1,y_2}) = 2,\eta (\Gamma_{x_2,y_1}) = 1, \text{ and }\eta (\Gamma_{x_2,y_2}) = 1.$$
To realize $T$ using $\eta$, all four transportation need to be used. 

On the other hand, 
\[ \tilde{\eta}= 3\delta_{\gamma_{x_1,y_1}} +\delta_{\gamma_{x_1,y_2}}+2\delta_{\gamma_{x_2,y_2}}\]
is a better decomposition of $T$. 
In this case,
$$S_{1,1} (\tilde{\eta}) - S_{1,2}  (\tilde{\eta})- S_{2,1} (\tilde{\eta}) + S_{2,2} (\tilde{\eta}) = 
I_{\gamma_{x_1,y_1}}-I_{\gamma_{x_1,y_2}} + I_{\gamma_{x_2,y_2}} \not=0$$
despite that
$$\tilde{\eta} (\Gamma_{x_1, y_1}) = 3, 
\tilde{\eta} (\Gamma_{x_1, y_2}) = 1, 
\tilde{\eta} (\Gamma_{x_2, y_1}) = 0,
\tilde{\eta} (\Gamma_{x_2, y_2}) = 2.$$
Using this new decomposition, to realize the same $T$, one only needs to arrange three transportation. 
\end{example}

\begin{definition} 
For any two finite measures $\eta$ and $\tilde{\eta}$ on $\Gamma$, we say $\tilde{\eta} \precc \eta$ if for each pair $(i,j)$,
\begin{equation}
\label{eqn: partial_order_eta}
    \int_{\Gamma_{x_i,y_j} }  I_\gamma d\tilde{\eta} =a_{i,j} \int_{\Gamma_{x_i,y_j} }  I_\gamma d\eta
\end{equation}
for some $a_{i,j}\ge 0$.
\end{definition}
Our main result for this section is the following theorem:

\begin{theorem} \label{thm: GoodS_ij}
Let 
$T$ be a transport path from $\mu^-$ to $\mu^+$ where $\mu^-$ and $\mu^+$ are given in (\ref{eqn: measures}). For any good decomposition $\eta$ of $T$, there exists a better decomposition $\eta_\infty$ of $T$ such that $\eta_\infty \precc  \eta$.
\end{theorem}
We first give an equivalent definition of $\tilde{\eta} \precc \eta$ as follows.
\begin{lemma} \label{lem: equivalent_definition_precc}
For any two finite measures $\eta$ and $\tilde{\eta}$ on $\Gamma$,
$\tilde{\eta}\precc \eta$ if and only if they satisfy the condition
\begin{equation}
\label{eqn: equivalent_def_precc}
\text{ if } \tilde{\eta} (\Gamma_{x_i,y_j})>0 \text{ for some } (i,j), \text{ then } \eta (\Gamma_{x_i,y_j})>0  \text{ and } 
S_{i,j}(\tilde{\eta})=S_{i,j}(\eta).
\end{equation}
\end{lemma}
\begin{remark}
By Lemma \ref{lem: equivalent_definition_precc}, it follows that  $\tilde{\eta} (\Gamma_{x_i,y_j})=0 $ whenever $\eta (\Gamma_{x_i,y_j})=0$. We use the notation $\tilde{\eta} \precc \eta$ to mimic the absolute continuity notation $\ll$ of measures. 
\end{remark}
\begin{proof}
Suppose $\tilde{\eta} \precc \eta$. By taking the boundary operator on both sides of (\ref{eqn: partial_order_eta}), it follows that
\[
\int_{\Gamma_{x_i,y_j} } (\delta_{y_j} -\delta_{x_i})d\tilde{\eta} =a_{i,j} \int_{\Gamma_{x_i,y_j} }  (\delta_{y_j} -\delta_{x_i})  d\eta.
\]
That is, 
$$ \tilde{\eta} (\Gamma_{x_i,y_j} ) (\delta_{y_j} -\delta_{x_i}) =a_{i,j}   {\eta} (\Gamma_{x_i,y_j} ) (\delta_{y_j} -\delta_{x_i}),$$
which implies that  $ \tilde{\eta} (\Gamma_{x_i,y_j} ) 
 = a_{i,j}  {\eta} (\Gamma_{x_i,y_j} )$.
Thus, 
$\tilde{\eta} (\Gamma_{x_i,y_j} ) >0  $ implies  $a_{ij} >0$ and ${\eta} (\Gamma_{x_i,y_j} ) >0.$
Moreover, 
$$ S_{i,j}(\tilde{\eta})=  
\frac{1}{\tilde{\eta} (\Gamma_{x_i,y_j} )}\int_{\Gamma_{x_i,y_j} }  I_\gamma d\tilde{\eta} =
\frac{1}{a_{i,j}  {\eta} (\Gamma_{x_i,y_j} )} \cdot a_{i,j} \int_{\Gamma_{x_i,y_j} }  I_\gamma d\eta  =  S_{i,j}(\eta).$$

On the other hand, suppose (\ref{eqn: equivalent_def_precc}) holds. 
If $\tilde{\eta} (\Gamma_{x_i,y_j} ) =0  $, then $a_{i,j}=0$ will give (\ref{eqn: partial_order_eta}). 
If $\tilde{\eta} (\Gamma_{x_i,y_j} ) >0  $,  then (\ref{eqn: equivalent_def_precc})  implies 
$\eta (\Gamma_{x_i,y_j})>0  \text{ and } 
S_{i,j}(\tilde{\eta})=S_{i,j}(\eta).$
By setting
$$a_{i,j} = \frac{\tilde{\eta} (\Gamma_{x_i,y_j} )}{ 
 {\eta} (\Gamma_{x_i,y_j} )  },$$
 equation
(\ref{eqn: S_ij}) gives that
\[ \int_{\Gamma_{x_i,y_j} }  I_\gamma d\tilde{\eta} =\tilde{\eta} (\Gamma_{x_i,y_j} ) S_{i,j}(\tilde{\eta})
=(a_{i,j}\eta (\Gamma_{x_i,y_j} ) )S_{i,j}(\eta)
=a_{i,j} \int_{\Gamma_{x_i,y_j} }  I_\gamma d\eta .\]
\end{proof}
Note that, by using the sign function
\begin{equation}
sgn(x)=
\left\{
\begin{array}{ll}
    1, & \text{ if }x>0 \\
    0, & \text{ if }x=0\\
     -1, & \text{ if }x<0, \\
\end{array}
\right.
\end{equation} 
equation (\ref{eqn: S_ij}) gives

\begin{equation} \label{eqn: partial_S_ij}
\partial S_{i,j}(\eta) = \left\{ 
\begin{array}{ll}
    \delta_{y_j}-\delta_{x_i}, & \text{ if }\eta(\Gamma_{x_i,y_j})>0, \\
    0, & \text{ if }\eta(\Gamma_{x_i,y_j})=0
\end{array}
\right.
=sgn(\eta(\Gamma_{x_i,y_j}))(\delta_{y_j}-\delta_{x_i}).
\end{equation}

For any pairs $1\le i_1< i_2\le M$ and $1\le j_1<j_2\leq N$, define
\begin{equation}
\label{eqn: C_def}
    C[(i_1,j_1), (i_2, j_2), \eta ] : = S_{i_1,j_1} (\eta) - S_{i_1,j_2}  (\eta)- S_{i_2,j_1} (\eta) + S_{i_2,j_2} (\eta).
\end{equation}
Direct calculation gives
$$
\partial C[(i_1,j_1), (i_2, j_2), \eta ]
=
\left(
sgn(\eta(\Gamma_{x_{i_1},y_{j_2}})  - sgn(\eta(\Gamma_{x_{i_1},y_{j_1}})\right) \delta_{x_{i_1}} + 
\left(sgn(\eta(\Gamma_{x_{i_2},y_{j_1}})  - sgn(\eta(\Gamma_{x_{i_2},y_{j_2}})\right)\delta_{x_{i_2}} 
$$
$$
\hphantom{\partial C[(i_1,j_1), (i_2, j_2), \eta ]\,} +
\left(sgn(\eta(\Gamma_{x_{i_1},y_{j_1}})  - sgn(\eta(\Gamma_{x_{i_2},y_{j_1}})\right)\delta_{y_{j_1}} +
\left(sgn(\eta(\Gamma_{x_{i_2},y_{j_2}})  - sgn(\eta(\Gamma_{x_{i_1},y_{j_2}})\right)\delta_{y_{j_2}}.
$$
Hence, it follows that $\partial C[(i_1,j_1), (i_2, j_2), \eta ] =0$ if and only if
\begin{equation}
\label{eqn: equalsgn}
sgn(\eta(\Gamma_{x_{i_1},y_{j_1}}))=sgn(\eta(\Gamma_{x_{i_1},y_{j_2}}))=sgn(\eta(\Gamma_{x_{i_2},y_{j_1}}))=sgn(\eta(\Gamma_{x_{i_2},y_{j_2}}))=c,
\end{equation} 
where $c=0 $ or $1$. We denote this common value, $c$, by $s[(i_1,j_1),(i_2,j_2), \eta]$.

\begin{definition}
    For any finite positive measure $\eta$ on $\Gamma$, define
\[
\mathcal{A}_\eta(i^*,j^*)=\{(i,j): i^*< i\le M, j^*<j\leq N, \   C[(i^*,j^*), (i, j), \eta ] =0 \text{ and } s[(i^*,j^*),(i,j), \eta]=1 
 \}. \]
\end{definition}

Using this definition, saying a good decomposition $\eta$ of $T$ is a better decomposition of $T$ is equivalent to $\mathcal{A}_\eta(i,j)=\emptyset$ for all pairs $(i,j)$.

We now consider the graded lexicographical order on $\mathbb{N}^2$, namely
\[(a,b)<(c,d) \text{ if } a+b<c+d \text{ or } a=c \text{ but } b<d.\]
Under this order, $\mathbb{N}^2$ is listed in the order of
\begin{equation}
\label{eqn: ordering}
    \{(i_n,j_n)\}_{n=1}^\infty=\{(1,1), (1,2), (2,1), (1,3),\ldots, (i_n,j_n), (i_{n+1},j_{n+1}),\ldots \}.
\end{equation}

\begin{lemma} \label{lem: S_ij(1,1)}
For any good decomposition $\eta$ of $T$, there exists a good decomposition $\tilde{\eta}$ of $T$  such that $\tilde{\eta}\precc \eta$ and  $\mathcal{A}_{\tilde{\eta}}(1,1)=\emptyset$. 
\end{lemma}

\begin{proof}
When $\eta(\Gamma_{x_1,y_1})=0$,  by (\ref{eqn: equalsgn}), the condition $C[(1,1), (i, j), \eta ] =0$ implies $s[(1,1),(i,j), \eta]=0$, and hence $\mathcal{A}_{\eta}(1,1)=\emptyset$.  Setting $\tilde{\eta}:= \eta$ gives us the desired results.

When $\eta(\Gamma_{x_1,y_1})\not=0$, we inductively define a sequence of good decomposition  $\{\eta_n\}$ of $T$ with $\eta_n(\Gamma_{x_1,y_1})>0 $, and whose limit is our desired measure $\tilde{\eta}$.
Set $\eta_1=\eta$.

If $\mathcal{A}_{\eta_n} (1,1)= \emptyset$ for some $n\ge 1$,  set  $\eta_m=\eta_n$ for all $m\ge n$ and set $\tilde{\eta}=\eta_n$ as well.

If $\mathcal{A}_{\eta_n}(1,1)$ is non-empty for all $n\ge 1$, we construct $\tilde{\eta}$ from $\{\eta_n\}$ via the following steps. 

{\bf Step 1: Construct a sequence of good decomposition $\{\eta_n\}$ of $T$.}

For each $n\ge 1$, assume that $\eta_n$ is a good decomposition of $T$ with $\eta_n(\Gamma_{x_1,y_1})>0 $.
Let $(i_n,j_n)$ be the minimum element in $ \mathcal{A}_{\eta_n}(1,1)$ which is a subset of $\mathbb{N}^2$ with the graded lexicographical order. Define
\[\eta_{n+1}:=\eta_{n}+ \min\{\eta_n(\Gamma_{x_1,y_{j_n}}), \eta_n(\Gamma_{x_{i_n},y_1})\}\left(
\frac{ \eta_n\lfloor_{\Gamma_{x_1,y_1}} }{\eta_n(\Gamma_{x_1,y_1})}  - 
\frac{\eta_n\lfloor_{\Gamma_{x_1,y_{j_n}}} }{\eta_n(\Gamma_{x_1,y_{j_n}})}  - 
\frac{\eta_n\lfloor_{\Gamma_{x_{i_n},y_1}} }{\eta_n(\Gamma_{x_{i_n},y_1})}  + 
\frac{\eta_n\lfloor_{\Gamma_{x_{i_n},y_{j_n}}} }{\eta_n(\Gamma_{x_{i_n},y_{j_n}})}  \right).
\]
Here, the denominators in the above equation are positive because  
$s[(1,1),(i_n,j_n), \eta_n]=1$.
Without loss of generality, we may assume that
\[0<\eta_n(\Gamma_{x_1,y_{j_n}})\leq\eta_n(\Gamma_{x_{i_n},y_1}).\]
Under this construction, we have for each $i, j$,
\begin{equation}
\label{eqn: eta_propotional}
\eta_{n+1}\lfloor_{\Gamma_{x_i,y_j}}=(1+\lambda_{n,i,j})\eta_{n}\lfloor_{\Gamma_{x_i,y_j}}
\end{equation}
for some real number $ \lambda_{n,i,j} \ge -1$.  In particular, it follows that
\begin{equation}\label{eqn: n+1 inequatliy 1}
\eta_{n+1}(\Gamma_{x_1,y_1}) >  \eta_{n}(\Gamma_{x_1,y_1})>0,  \
\eta_{n}(\Gamma_{x_1,y_{j_n} })> \eta_{n+1}(\Gamma_{x_1,y_{j_n}})=0,
\end{equation}
\begin{equation}\label{eqn: n+1 inequatliy 2}
\eta_{n} (\Gamma_{x_{i_n},y_1})> \eta_{n+1}(\Gamma_{x_{i_n},y_1})\geq 0, \
\eta_{n+1}(\Gamma_{x_{i_n},y_{j_n} }) > \eta_{n}(\Gamma_{x_{i_n},y_{j_n} })>0,
\end{equation}
and 
\begin{equation} \label{eqn: n+1 inequatliy 3}
\eta_{n+1}(\Gamma_{x_i,y_j})= \eta_{n}(\Gamma_{x_i,y_j})    
\text{ for all other } i,j.
\end{equation}

Since $\eta_n$ is a good decomposition of $T$, we have
$$T =\int_{\Gamma}I_{\gamma}d\eta_n, \; \mathbf{M}(T)=\int_\Gamma \mathbf{M}(I_\gamma) d\eta_{n}(\gamma)\; \text{ and }\mathbf{M}(\partial T)=\int_\Gamma \mathbf{M}(\partial I_\gamma) d\eta_n(\gamma). $$
In particular, $\mathbf{M}(T)=\int_\Gamma \mathbf{M}(I_\gamma) d\eta_{n}(\gamma)$ implies that
$$\mathbf{M}( S_{1,1}(\eta_{n}) + S_{i_n,j_n}(\eta_{n}) ) = \mathbf{M}( S_{1,1}(\eta_{n}) ) + \mathbf{M}( S_{i_n,j_n}(\eta_{n}) ),$$ 
and
$$\mathbf{M}( S_{1,j_n} (\eta_{n})+ S_{i_n,1}(\eta_{n}) ) = \mathbf{M}( S_{1,j_n}(\eta_{n}) ) + \mathbf{M}( S_{i_n,1}(\eta_{n}) ).$$
By assumption,
$$C[(1,1), (i_n, j_n), \eta_n] = S_{1,1}(\eta_{n})-  S_{1,j_n}(\eta_{n}) - S_{i_n,1}(\eta_{n}) + S_{i_n,j_n}(\eta_{n}) =0,$$ 
i.e.,
$S_{1,1}(\eta_{n})+S_{i_n,j_n} (\eta_{n})=S_{1,j_n}(\eta_{n}) + S_{i_n,1}(\eta_{n}).$
Thus,
\begin{eqnarray*}
&& \mathbf{M}( S_{1,1}(\eta_{n}) ) + \mathbf{M}( S_{i_n,j_n}(\eta_{n}) )
=
\mathbf{M}( S_{1,1}(\eta_{n}) + S_{i_n,j_n}(\eta_{n}) ) \\
&=& 
\mathbf{M}( S_{1,j_n}(\eta_{n}) + S_{i_n,1}(\eta_{n}) ) =
\mathbf{M}( S_{1,j_n} (\eta_{n})) + \mathbf{M}( S_{i_n,1}(\eta_{n}) ).
\end{eqnarray*}
Now, by the construction of $\eta_{n+1}$,
\[\int_{\Gamma}I_{\gamma}d \eta_{n+1}-\int_{\Gamma}I_{\gamma}d \eta_n=
\min\{\eta_n(\Gamma_{x_1,y_{j_n}}), \eta_n(\Gamma_{x_{i_n},y_1})\} \cdot 
 C[(1,1), (i_n, j_n), \eta_n]  =0 ,
\]
and
\begin{eqnarray*}
& &\int_\Gamma \mathbf{M}(I_\gamma) d\eta_{n+1}(\gamma)- \int_\Gamma \mathbf{M}(I_\gamma) d\eta_{n}(\gamma) \\
&=&\min\{\eta_n(\Gamma_{x_1,y_{j_n}}), \eta_n(\Gamma_{x_{i_n},y_1})\} \left( \mathbf{M}( S_{1,1} ) - \mathbf{M} ( S_{1,j_n} ) - \mathbf{M} ( S_{i_n,1}) + \mathbf{M}( S_{i_n,j_n} ) \right)=0.
\end{eqnarray*}
Moreover,
\begin{eqnarray*}
& &\int_\Gamma \mathbf{M}(\partial I_\gamma) d\eta_{n+1}(\gamma)-\int_\Gamma \mathbf{M}(\partial I_\gamma) d\eta_{n}(\gamma)\\
&=&
\min\{\eta_n(\Gamma_{x_1,y_{j_n}}), \eta_n(\Gamma_{x_{i_n},y_1})\} 
\left( \mathbf{M}(\partial S_{1,1} ) - \mathbf{M} ( \partial S_{1,j_n} ) - \mathbf{M} (\partial S_{i_n,1}) + \mathbf{M}( \partial S_{i_n,j_n} ) \right) \\
&=&
\min\{\eta_n(\Gamma_{x_1,y_{j_n}}), \eta_n(\Gamma_{x_{i_n},y_1})\} 
\left(2 - 2 - 2 +2 \right) =0.
\end{eqnarray*}
As a result, since $\eta_n$ is a good decomposition of $T$, $\eta_{n+1}$ is a good decomposition of $T$ as well.

{\bf Step 2: Show that the sequence $\{\eta_n\}$ converges to a good decomposition $\tilde{\eta}$ of $T$.}

Note that for each $1\le i\le M$ and $1\le j\le N$, the sequence $\{\eta_n\lfloor_{\Gamma_{x_i,y_j}}\}_{n=1}^\infty$ is a  monotonic sequence of measures with bounded mass. 
Indeed, by the construction above and by equations (\ref{eqn: n+1 inequatliy 1}), (\ref{eqn: n+1 inequatliy 2}) and (\ref{eqn: n+1 inequatliy 3}),
\begin{itemize}
\item if $i=1,j=1$, then 
$\{\eta_n\lfloor_{\Gamma_{x_i,y_j}}\}_{n=1}^\infty$ is monotone increasing;

\item if $i=1,j>1$, then 
$\{\eta_n\lfloor_{\Gamma_{x_i,y_j}}\}_{n=1}^\infty$ is monotone decreasing;

\item if $i>1,j=1$, then 
$\{\eta_n\lfloor_{\Gamma_{x_i,y_j}}\}_{n=1}^\infty$ is monotone decreasing;

\item if $i>1,j>1$, then 
$\{\eta_n\lfloor_{\Gamma_{x_i,y_j}}\}_{n=1}^\infty$ is monotone increasing, and eventually constant.

\end{itemize}
As a result, the sequence, 
$\{\eta_n\lfloor_{\Gamma_{x_i,y_j}}\}_{n=1}^\infty$, converges to some measure $\eta_{ij}$ for each $(i,j)$.
Define
\[\tilde{\eta}:=\sum_{i=1}^M\sum_{j=1}^N \eta_{ij}. \]
Hence, as $n\rightarrow \infty$,
\[
\eta_n = \sum_{i=1}^M \sum_{j=1}^N \eta_n \lfloor_{\Gamma_{x_i,y_j}} 
\longrightarrow \tilde{\eta}=
\sum_{i=1}^M\sum_{j=1}^N \eta_{ij}.
\]
Since each $\eta_n$ is a good decomposition of $T$, it follows that
\begin{eqnarray*}
&&\int_{\Gamma} I_{\gamma} d\tilde{\eta} =
\lim_{n\to \infty}\int_{\Gamma} I_{\gamma} d\eta_{n}= T,\\
&&\int_{\Gamma} \mathbf{M}(I_\gamma) d\tilde{\eta} =
\lim_{n\to \infty}\int_{\Gamma} \mathbf{M}(I_\gamma) d\eta_{n} = \mathbf{M}(T),\\
&&\int_{\Gamma} \mathbf{M}(\partial I_\gamma) d\tilde{\eta} =
\lim_{n\to \infty}\int_{\Gamma} \mathbf{M}(\partial I_\gamma) d\eta_{n} = \mathbf{M}(\partial T).
\end{eqnarray*}
As a result, $\tilde{\eta}$ is also a good decomposition of $T$.

{\bf Step 3: Show that $\tilde{\eta} \precc \eta$ .}

Suppose $\tilde{\eta}(\Gamma_{x_i,y_j}) >0$ for some pair $(i,j)$. 
Then, ${\eta}_n(\Gamma_{x_i,y_j}) >0$
when $n$ is large enough. 
By (\ref{eqn: eta_propotional}),
$${\eta}_n \lfloor_{\Gamma_{x_i,y_j}} = \prod_{k=1}^{n-1}(1+\lambda_{k,i,j}){\eta} \lfloor_{\Gamma_{x_i,y_j}}, \text{ for some } \lambda_{k,i,j} \ge -1 \text{ for each } k.$$
That is,
\[{\eta}_n=\left(\prod_{k=1}^{n-1}(1+\lambda_{k,i,j})\right){\eta} \text{ on } \Gamma_{x_i,y_j}.\]
As a result, ${\eta}_n(\Gamma_{x_i,y_j}) >0$ implies ${\eta}(\Gamma_{x_i,y_j}) >0$ and $S_{i,j}(\eta_n)=S_{i,j}(\eta)$. Since $\tilde{\eta}$ is the limit of $\eta_n$, $$S_{i,j}(\tilde{\eta})=\lim_{n\rightarrow \infty}S_{i,j}(\eta_n)=S_{i,j}(\eta).$$
This proves $\tilde{\eta}\precc \eta$.

{\bf Step 4: Show that $\mathcal{A}_{\eta_{n+1}} (1,1) \subsetneqq \mathcal{A}_{\eta_n} (1,1) $ for each $n$.}

Note that $(i_n,j_n)\in \mathcal{A}_{\eta_n} (1,1) \setminus \mathcal{A}_{\eta_{n+1}} (1,1) $. 
Indeed, if $(i_n,j_n)\in \mathcal{A}_{\eta_{n+1}} (1,1)$, then $$C[(1,1), (i_n, j_n), \eta_{n+1} ] =0 \text{ and } s[(1,1),(i_n,j_n), \eta_{n+1}]=1 .$$
This implies $sgn(\eta_{n+1}(\Gamma_{x_1,y_{j_n}}))=1$, which contradicts with $\eta_{n+1}(\Gamma_{x_1,y_{j_n}})=0$ as given in (\ref{eqn: n+1 inequatliy 1}). 

We now show that $\mathcal{A}_{\eta_{n+1}} (1,1) \subseteq  \mathcal{A}_{\eta_n} (1,1)$. 
For any $(i_0,j_0) \in \mathcal{A}_{\eta_{n+1}} (1,1)$, by definition, 
$$
C[(1,1), (i_0, j_0), \eta_{n+1} ] =0 \text{ and } 
s[(1,1), (i_0, j_0), \eta_{n+1} ] =1. 
$$ 
The condition
$s[(1,1), (i_0, j_0), \eta_{n+1} ] =1$ indicates that
$$\eta_{n+1}(\Gamma_{x_1,y_1})>0,\eta_{n+1}(\Gamma_{x_1,y_{j_0}})>0,\eta_{n+1}(\Gamma_{x_{i_0},y_1})>0,\eta_{n+1}(\Gamma_{x_{i_0},y_{j_0}}) >0.$$ 
By equations  (\ref{eqn: n+1 inequatliy 1})--(\ref{eqn: n+1 inequatliy 3}), and  $(i_0, j_0)\neq (i_n, j_n)$, 
\[
\eta_{n}(\Gamma_{x_1,y_1}) >0,  \ 
\eta_{n}(\Gamma_{x_1,y_{j_0} })\geq \eta_{n+1}(\Gamma_{x_1,y_{j_0}})>0,
\]
\[
\eta_{n} (\Gamma_{x_{i_0},y_1})\geq \eta_{n+1}(\Gamma_{x_{i_0},y_1})>0,  \ 
\eta_{n}(\Gamma_{x_{i_0},y_{j_0} })= \eta_{n+1}(\Gamma_{x_{i_0},y_{j_0} })>0.
\]
By (\ref{eqn: eta_propotional}), for each $i, j$,
when both $\eta_{n}(\Gamma_{x_i,y_j})>0$ and
$\eta_{n+1}(\Gamma_{x_i,y_j})>0$, then
\[S_{i,j}(\eta_n)=S_{i,j}(\eta_{n+1}).\]
As a result,
\[
C[(1,1), (i_0, j_0), \eta_{n} ] =
C[(1,1), (i_0, j_0), \eta_{n+1} ] = 0.
\]
Therefore, $(i_0,j_0)\in \mathcal{A}_{\eta_n}(1,1)$ and hence $\mathcal{A}_{\eta_{n+1}} (1,1) \subseteq  \mathcal{A}_{\eta_n} (1,1)$.

{\bf Step 5: Show that $\mathcal{A}_{\tilde{\eta}}(1,1)=\emptyset$.}

Assume that there exists $(i',j')\in \mathcal{A}_{\tilde{\eta}} (1,1)$, i.e. 
$C[(1,1), (i', j'), \tilde{\eta} ] =0$ and 
$s[(1,1), (i', j'), \tilde{\eta} ] =1$.
For any $(i,j) \in \{ (1,1),(1,j'),(i',1),(i',j')\}$,
since
$s[(1,1), (i', j'), \tilde{\eta} ] =1$, it follows that 
\[\lim_{n\rightarrow \infty} \eta_{n}(\Gamma_{x_{i},y_{j} }) = \tilde{\eta}(\Gamma_{x_{i},y_{j}})>0.
\]
Thus, there exists an $N_0\in \mathbb{N}$ such that $\eta_n(\Gamma_{x_i,y_j})>0$ for all $n\ge N_0$. By (\ref{eqn: eta_propotional}), this implies that the normalized current
 $S_{i,j}(\eta_n)$ is independent of $n$, and hence
$S_{i,j}(\eta_n)=S_{i,j}(\tilde{\eta})$ for all $n\ge N_0$. As a result, for each $n\ge N_0$,
\[
C[(1,1), (i', j'), \eta_{n} ]=C[(1,1), (i', j'), \tilde{\eta} ] =0  \text{ and } 
s[(1,1), (i', j'), \eta_{n} ] = s[(1,1), (i', j'), \tilde{\eta} ] =1.\]
This shows that $(i',j')\in\mathcal{A}_{\eta_n} (1,1)$. On the other hand, since $\{ \mathcal{A}_{\eta_n}  (1,1)\}$ is a sequence of nested subsets in $\mathbb{N}^2$ with $\mathcal{A}_{\eta_{n+1}} (1,1) \subsetneqq \mathcal{A}_{\eta_n} (1,1)$ for each $n$. When $n$ is larger than the order of the fixed element $(i',j')$, it is not possible for $(i',j')\in\mathcal{A}_{\eta_n} (1,1)$. 
A contradiction.

\end{proof}

We now extend Lemma \ref{lem: S_ij(1,1)} to a more general case:

\begin{lemma} \label{lem: ikjk}
For any good decomposition $\eta$ of $T$, there exists a sequence of good decomposition $\{\eta_n\}_{n=0}^\infty$ of $T$ with $\eta_0=\eta$ such that for each $n\ge 1$,
$\eta_n\precc \eta_{n-1}$   and   $\mathcal{A}_{\eta_{n}}(i_k, j_k)=\emptyset$ for all $1\leq k\leq n$, where $\{(i_k, j_k)\}$ is given in (\ref{eqn: ordering}).
\end{lemma}

\begin{proof}
We will prove these results by induction. Lemma \ref{lem: S_ij(1,1)} provides the base case when $n=1$. For each $n\ge 2$, assume that there exists a good decomposition $\eta_{n-1}$ of $T$ such that $\eta_{n-1}\precc \eta_{n-2}$ and   $\mathcal{A}_{\eta_{n-1}}(i_k, j_k)=\emptyset$ for all $1\leq k\leq n-1$. Using $\eta_{n-1}$, we construct $\eta_n$ as follows.

Denote
\[\tilde{\Gamma}_n=\bigcup_{i_n \le i, j_n \le j} \Gamma_{x_i, y_j}.\]
Let $\tilde{\eta}_n$ be the measure $\tilde{\eta}$ achieved in Lemma \ref{lem: S_ij(1,1)} with $\eta$ being replaced by 
$\eta_{n-1}\lfloor_{\tilde{\Gamma}_n}$ and $T$ being replaced by $\tilde{T}:=\int_{\tilde{\Gamma}_n} I_\gamma d\eta_{n-1}$. Define
\[
\eta_n:=
\eta_{n-1} \lfloor_{\Gamma\setminus\tilde{\Gamma}_n} + \tilde{\eta}_n.
\]
We first claim that $\eta_n$ is a good decomposition of $T$. Indeed, since both $\tilde{\eta}_n$ and $\eta_{n-1}\lfloor_{\tilde{\Gamma}_n}$ are good decompositions of $\tilde{T}$,
\[\int_\Gamma I_\gamma d\eta_n-\int_\Gamma I_\gamma d\eta_{n-1}=\int_\Gamma I_\gamma d\tilde{\eta}_n-\int_{\tilde{\Gamma}_n} I_\gamma d\eta_{n-1}=0,\]
\begin{eqnarray*}
\int_\Gamma \mathbf{M}(I_\gamma) d\eta_{n}(\gamma)- \int_\Gamma \mathbf{M}(I_\gamma) d\eta_{n-1}(\gamma) =\int_\Gamma\mathbf{M}(I_\gamma) d\tilde{\eta}_n-\int_{\tilde{\Gamma}_n} \mathbf{M}(I_\gamma) d\eta_{n-1}=0,
\end{eqnarray*}
and
\begin{eqnarray*}
\int_\Gamma \mathbf{M}(\partial I_\gamma) d\eta_{n}(\gamma)- \int_\Gamma \mathbf{M}(\partial I_\gamma) d\eta_{n-1}(\gamma) =\int_\Gamma\mathbf{M}(\partial I_\gamma) d\tilde{\eta}_n-\int_{\tilde{\Gamma}_n} \mathbf{M}(\partial I_\gamma) d\eta_{n-1}=0.
\end{eqnarray*}
As a result, since $\eta_{n-1}$ is a good decomposition of $T$, $\eta_{n}$ is also a good decomposition of $T$.

We now show that $ \eta_n  \precc \eta_{n-1} $.
Suppose $\eta_n (\Gamma_{x_i,y_j}) >0$ for some $1\le i\le M, 1\leq j\leq N$.
\begin{itemize}

\item 
When $i<i_n$ or $j<j_n$, definition of $\eta_n$ gives $\eta_n \lfloor_{\Gamma_{x_i,y_j}}=\eta_{n-1}\lfloor_{ \Gamma_{x_i,y_j}}$. 
Therefore, 
$$\eta_{n-1} (\Gamma_{x_i,y_j})=\eta_{n} (\Gamma_{x_i,y_j})>0  \text{ and }  S_{i,j}(\eta_{n-1})=S_{i,j}(\eta_{n}).$$

\item 
When $i\geq i_n$ and $j\geq j_n$, definition of $\eta_n$ gives $\eta_n \lfloor_{\Gamma_{x_i,y_j}}=\tilde{\eta}_{n}\lfloor_{ \Gamma_{x_i,y_j}}$, so that $$\tilde{\eta}_{n} (\Gamma_{x_i,y_j})=\eta_{n} (\Gamma_{x_i,y_j})>0. $$
Since $ \tilde{\eta}_n \precc \eta_{n-1}\lfloor_{\tilde{\Gamma}_n} $ by Lemma \ref{lem: S_ij(1,1)}, it follows that
\[\eta_{n-1} (\Gamma_{x_i,y_j})>0 \text{ and } S_{i,j}(\eta_{n-1})=S_{i,j}(\tilde{\eta}_n)=S_{i,j}(\eta_{n}).\]
\end{itemize}
In both cases, $\eta_{n-1} (\Gamma_{x_i,y_j})>0$ and $S_{i,j}(\eta_{n-1})=S_{i,j}(\eta_{n})$. That is, $ \eta_{n} \precc \eta_{n-1}$.

We now show that $\mathcal{A}_{\eta_{n}}(i_k, j_k)=\emptyset$ for all $1\leq k\leq n$.
When $k=n$, $\mathcal{A}_{\eta_{n}}(i_n, j_n) =\emptyset$ by Lemma \ref{lem: S_ij(1,1)}.
Suppose $k<n$, and for contradiction, we assume $\mathcal{A}_{\eta_{n}}(i_k, j_k) \neq\emptyset$. Thus, there exists $(i^*,j^*)\in \mathcal{A}_{\eta_{n}}(i_k, j_k)$, i.e.,
$$C[(i_k,j_k), (i^*, j^*), \eta_n ] =0 \text{ and } s[(i_k,j_k),(i^*,j^*), \eta_n]=1. $$
Now, for any $(i,j)\in \{(i_k, j_k), (i_k, j^*), (i^*,j_k), (i^*,j^*)\}$, since 
$s[(i_k,j_k),(i^*,j^*), \eta_n]=1$, it follows that
$\eta_n(\Gamma_{x_{i}, y_{j}})>0$. By the definition of $\eta_n$, when $i<i_n$ or $j<j_n$,  $\eta_n=\eta_{n-1}$ on $\Gamma_{x_i, y_j}$. Thus,
\begin{equation}
\label{eqn: equal_eta_S}
\eta_{n-1}(\Gamma_{x_{i}, y_{j}})=\eta_{n}(\Gamma_{x_{i}, y_{j}})>0 \text{ and }
S_{i,j}(\eta_n)=S_{i,j}(\eta_{n-1}).
\end{equation}
When $i\ge i_n$ and $j\ge j_n$, 
$$\tilde{\eta}_{n}(\Gamma_{x_{i}, y_{j}})=\eta_{n}(\Gamma_{x_{i}, y_{j}})>0.$$
Since $ \tilde{\eta}_n \precc \eta_{n-1}\lfloor_{\tilde{\Gamma}_n} $, then equations in (\ref{eqn: equal_eta_S}) still hold. As a result, $$C[(i_k,j_k), (i^*, j^*), \eta_{n-1} ]=C[(i_k,j_k), (i^*, j^*), \eta_{n} ] =0 \text{ and } s[(i_k,j_k),(i^*,j^*), \eta_{n-1}]=1.$$ 
Therefore,
 $(i^*, j^*)\in \mathcal{A}_{\eta_{n-1}}(i_k, j_k)$, which contradicts with  $\mathcal{A}_{\eta_{n-1}}(i_k, j_k)=\emptyset$ whenever $k\le n-1$.
\end{proof}

We now give the proof of Theorem~\ref{thm: GoodS_ij} by showing that for any good decomposition $\eta$ of $T$, there exists a good decomposition $\eta_\infty$ of $T$ such that $\eta_\infty \precc  \eta$ and $\mathcal{A}_{\eta_\infty}(i, j)=\emptyset$ for all $ 1\le i\le M, 1\le j\le N$.

\begin{proof}[Proof of Theorem~\ref{thm: GoodS_ij}]
Let $\{\eta_n\}$ be the sequence of good decomposition of $T$ constructed in the proof of Lemma \ref{lem: ikjk}. Observe that by the construction of the sequence $\{\eta_n\}$, it follows that for any $k\in \mathbb{N}$, \begin{equation}
    \label{eqn: eta_n_constant}
\eta_n\lfloor_{\Gamma_{x_{i_k},y_{j_k}} } =\eta_k\lfloor_{\Gamma_{x_{i_k},y_{j_k}} }
\end{equation}
for all $n\geq k$. Define $\eta_\infty: \Gamma \rightarrow \mathbb{R}$ by setting
\begin{equation}
\label{eqn: definition_eta_infinty}
\eta_\infty:=\eta_{k} \quad \text{ on } \Gamma_{x_{i_k},y_{j_k}}, \forall k\in \mathbb{N}.
\end{equation}

We first show that $\{\eta_n\}$ converges to $\eta_\infty$ with respect to the total variation distance $\|\cdot\|$. Indeed, by (\ref{eqn: eta_n_constant}),
\begin{eqnarray*}
\|  \eta_n  -  \eta_\infty \| 
&=& 
 \|  
\sum_{k\ge 1} (\eta_n- \eta_k) \lfloor_{\Gamma_{x_{i_k}, y_{j_k}}}  \| 
=
 \|  
\sum_{k\ge n+1} (\eta_n- \eta_k) \lfloor_{\Gamma_{x_{i_k}, y_{j_k}}} \|    \\
&\le &
\sum_{k \ge n+1}\eta_n(\Gamma_{x_{i_k}, y_{j_k}}) 
+
\sum_{k \ge n+1}  \eta_k(\Gamma_{x_{i_k}, y_{j_k}})\\
&\le &
\sum_{i_k +j_k \ge i_n + j_n }\eta_n(\Gamma_{x_{i_k}, y_{j_k}}) 
+
\sum_{k \ge n+1 } \eta_k(\Gamma_{x_{i_k}, y_{j_k}}) \\
&\leq &
\sum_{i_k \ge \sqrt{i_n j_n}} \sum_{j_k =1 }^N \eta_n(\Gamma_{x_{i_k}, y_{j_k}}) 
+ 
\sum_{j_k \ge \sqrt{i_n j_n}} \sum_{i_k =1 }^M \eta_n(\Gamma_{x_{i_k}, y_{j_k}})      
+
\sum_{k \ge n+1 } \eta_k(\Gamma_{x_{i_k}, y_{j_k}}) 
\\
&=&
\sum_{i_k \ge \sqrt{i_n j_n}}   m'_{i_k}
+ 
\sum_{j_k \ge \sqrt{i_n j_n}}  m_{j_k}
+
\sum_{k \ge n+1 } \eta_k(\Gamma_{x_{i_k}, y_{j_k}}),
\end{eqnarray*}
and
$$ \eta_\infty (\Gamma)  = \sum_{k=1}^\infty \eta_{k}(\Gamma_{x_{i_k}, y_{j_k}})=\lim_{n\rightarrow \infty}\sum_{k=1}^n \eta_{k}(\Gamma_{x_{i_k}, y_{j_k}}) 
=\lim_{n\rightarrow \infty}\sum_{k=1}^n \eta_{n}(\Gamma_{x_{i_k}, y_{j_k}})\le \lim_{n\rightarrow \infty}\eta_n(\Gamma)=\eta(\Gamma)<\infty.$$
Thus, since $\lim_{n\rightarrow \infty}i_nj_n=\infty$ and $\sum_{i=1}^M m'_i = \sum_{j=1}^N m_j < \infty$, it follows that
$\lim_{n\rightarrow \infty}\|  \eta_n  -  \eta_\infty \|=0$. 
Since $\eta_n$ is a good decomposition for each $n$, it follows that its limit $\eta_\infty$ is also a good decomposition of $T$.

Moreover, if $\eta_{\infty}(\Gamma_{x_{i_k},y_{j_k}})>0$ for some $k$, then $\eta_{k}(\Gamma_{x_{i_k},y_{j_k}})>0$ by (\ref{eqn: definition_eta_infinty}). 
Thus, by Lemma  \ref{lem: ikjk} and transitivity of ``$\precc$", we have
$\eta_k \precc \eta$, which implies $$\eta (\Gamma_{x_{i_k},y_{j_k}})>0  \text{ and } 
S_{i_k,j_k}(\eta_\infty)=S_{i_k,j_k}(\eta_k)=S_{i_k,j_k}(\eta).$$
Therefore, $ \eta_\infty \precc \eta$.

We now show that
$\mathcal{A}_{\eta_\infty}(i_k, j_k)=\emptyset$ for each $k$. Assume that for some $k$, $\mathcal{A}_{\eta_\infty}(i_k, j_k)$ contains an element $(i_n,j_n)$. 
Then the definition of $\mathcal{A}_{\eta_\infty}(i_k, j_k)$ implies $n> k$ and  
\[C[(i_k,j_k), (i_n, j_n), \eta_\infty ] =0 \text{ and } s[(i_k,j_k),(i_n,j_n), \eta_\infty]=1.\]
By (\ref{eqn: eta_n_constant}) and (\ref{eqn: definition_eta_infinty}), since $(i_n, j_n)$ has the largest order among the elements 
$$\{(i_k,j_k),(i_k,j_n),(i_n,j_k),(i_n,j_n)\},$$ it follows that 
$\eta_\infty=\eta_n$ on $\Gamma_{x_i, y_j}$ for each $(i,j)$ of these four elements. Thus,
\[C[(i_k,j_k), (i_n, j_n), \eta_n ] =0 \text{ and } s[(i_k,j_k),(i_n,j_n), \eta_n]=1.\]
This shows
$(i_n,j_n) \in \mathcal{A}_{\eta_n}(i_k, j_k)$, a contradiction with $\mathcal{A}_{\eta_n}(i_k, j_k)=\emptyset$ due to Lemma \ref{lem: ikjk}.
\end{proof}

\section{Decomposition of cycle-free transport paths }

In this section, we will prove the decomposition theorem in Theorem \ref{thm: decomposition} using the better decomposition $\eta_\infty$ achieved from Theorem \ref{thm: GoodS_ij}. 

We first recall a concept that was introduced in \cite[Definition 4.6]{boundary_payoff}.
\begin{definition}
\label{def: S_on_T}
Let $T=\underline{\underline{\tau}}(M,\theta,\xi)$ and $S=\underline{\underline{\tau}}(N,\phi,\zeta)$ be two real rectifiable $k$-currents. We say $S$ is
on $T$ if $\mathcal{H}^k(N\setminus M)=0$, and $\phi(x)\le \theta(x)$ for $\mathcal{H}^k$ almost all $ x\in N$. 
\end{definition}

Note that when $S=\underline{\underline{\tau}}(N,\phi,\zeta)$ is
on $T=\underline{\underline{\tau}}(M,\theta,\xi)$, then $\xi(x)=\pm \zeta(x)$ for $\mathcal{H}^k$
almost all $ x\in N$, since two rectifiable sets have the same tangent almost everywhere on their intersection. 
Using it, we now introduce the concept of ``cycle-free" currents as follows:
\begin{definition} 
\label{def: cycle-free current}
Let $T$ and $S$ be two real rectifiable $k$-currents. $S$ is called a cycle on $T$ if $S$ is on $T$ and $\partial S=0$. Also, $T$ is called {\it cycle-free} if except for the zero current, there is no other cycle on $T$.
\end{definition}
The zero current is called the trivial cycle on $T$.

\begin{remark}
\label{rmk: cycle-free}
The concept of ``cycle-free" is different from ``acyclic".  A cycle-free current is automatically acyclic, but not vice versa. For instance, let $T$ be a transport path (which is a 1-current) from $\mu^- = \delta_{x_1} + \delta_{x_2}$ to $\mu^+ = \delta_{y_1} + \delta_{y_2}$ as shown below. 
\end{remark}

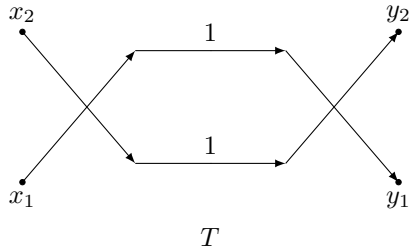
\begin{center}

\begin{tikzpicture}[>=latex]
\filldraw[black] (0,0) circle (1pt) node[anchor=north]{$x_1$};
\filldraw[black] (5,0) circle (1pt) node[anchor=north] {$y_1$};
\filldraw[black] (0,2) circle (1pt) node[anchor=south]{$x_2$};
\filldraw[black] (5,2) circle (1pt) node[anchor=south]{$y_2$};

\draw[->] (0,0)--(1.5,1.75);
\draw[->]        (1.5,1.75)--(3.5,1.75);
\draw[->]                    (3.5,1.75)--(5,0);

\draw[->] (0,2)--(1.5,0.25);
\draw[->]        (1.5,0.25)--(3.5,0.25);
\draw[->]                    (3.5,0.25)--(5,2);

\filldraw[black] (2.5,1.75) circle (0pt) node[anchor=south]{$1$};
\filldraw[black] (2.5,0.25) circle (0pt) node[anchor=south]{$1$};
   
\end{tikzpicture}    

\begin{tikzpicture}
\filldraw[black] (0,0) circle (0pt) node[anchor=south]{$T$};
\end{tikzpicture}
\end{center}
Then $T$ is acyclic but not cycle-free.

As an example, we first show that each optimal transport path is cycle-free. To do so, we start with 
an analogous result to \cite[Theorem 4.7]{boundary_payoff} as follows.

\begin{proposition}
    \label{eq:thm_perturbation}
Let $T\in Path(\mu^-, \mu^+)$ with $\mathbf{M}_{\alpha}(T)<\infty$ for some $0<\alpha<1$. 
Suppose there exists a rectifiable 1-current $S$ such that $S$ is on $T$ and $\partial S=0$, then for any $\epsilon \in [-1,1]$, $T + \epsilon S\in  Path(\mu^-, \mu^+)$ and
$$\min \left\{
\mathbf{M}_\alpha(T + S),
\mathbf{M}_\alpha(T - S) \right\}
\le \mathbf{M}_\alpha(T)$$
with the equality holds only when $S=0$. 
\end{proposition}

\begin{proof}
The statements clearly hold if $S=0$. Thus, in the following, we may assume that $S$ is non-zero.
Since $T\in Path(\mu^-, \mu^+)$ and $\partial S=0$, it holds that $\partial (T + \epsilon S) = \partial T + \epsilon\partial S = \partial T=\mu^+-\mu^-$. That is, $T+\epsilon S \in  Path(\mu^-,\mu^+)$.

Let $T=\underline{\underline{\tau}}(M,\theta,\xi)$ and $S=\underline{\underline{\tau}}(N,\phi,\zeta)$. Since $S$ is
on $T$, we have $\mathcal{H}^1(N\setminus M)=0$, and $\phi(x)\le \theta(x)$ for $\mathcal{H}^1$ almost all $ x\in N$. One may assume that $N=M$ by extending $\phi (x)=0$ and $\zeta (x)=\xi (x)$
for $x\in M\setminus N$. 

For $\epsilon \in [-1,1]$, we now consider the function
\[g(\epsilon)=\mathbf{M}_\alpha (T+\epsilon S) =\int_{M}\left( \theta (x)+\epsilon\phi (x)\langle \xi
(x),\zeta (x)\rangle \right) ^{\alpha }d\mathcal{H}^{1}(x).\]
Here, the value of the inner product is $\langle \xi (x),\zeta (x)\rangle =\pm 1$ for $\mathcal{H}^1-a.e.\  x\in M$.
Since $\mathbf{M}_{\alpha}(T)=\int_M \theta^\alpha d\mathcal{H}^1<\infty$ and $\phi(x)\le \theta(x)$ for $\mathcal{H}^1$ almost all $ x\in M$, we have for any $\epsilon\in (-1,1)$,
\[g'(\epsilon)=\alpha\int_{M}\left( \theta (x)+\epsilon\phi (x)\langle \xi
(x),\zeta (x)\rangle \right) ^{\alpha -1}\phi (x)\langle \xi
(x),\zeta (x)\rangle d\mathcal{H}^{1}(x)\]
and
\[g''(\epsilon)=\alpha(\alpha-1)\int_{M}\left( \theta (x)+\epsilon\phi (x)\langle \xi
(x),\zeta (x)\rangle \right) ^{\alpha -2}\phi (x)^2 d\mathcal{H}^{1}(x)<0,\]
because $0<\alpha<1$ and $S$ is non-zero. This shows that $g(\epsilon)$ is a strictly concave function on $(-1,1)$. 
By the lower semi-continuity of $\mathbf{M}_{\alpha}$, $g(\epsilon)$ is lower semi-continuous at $\epsilon=\pm 1$. 
Thus, $\min\{g(-1),g(1)\}<g(0)$. That is, $\min  \{ \mathbf{M}_\alpha({T}+{S}), \mathbf{M}_\alpha({T}-{S})  \} 
< \mathbf{M}_\alpha({T})$ whenever  $S$ is on $T$, nonzero and $\partial S=0$. 
\end{proof}

\begin{corollary}
Suppose $T$ is an $\alpha$-optimal transport path from $\mu^-$ to $\mu^+$ for $0< \alpha<1$. Then $T$ is cycle-free.
\end{corollary}
\begin{proof}
Since $T$ is $\alpha$-optimal, it is acyclic and hence it has a good decomposition.
Suppose $S$ is on $T$ and $\partial S=0$. Assume $S$ is non-zero, then $\min  \{ \mathbf{M}_\alpha({T}+{S}), \mathbf{M}_\alpha({T}-{S})  \} 
< \mathbf{M}_\alpha({T})$, which contradicts with the $\mathbf{M}_{\alpha}$ optimality of $T$. Therefore, $S$ must be zero. Hence, $T$ is cycle-free.
\end{proof}

To characterize cycle-free transport paths, we consider their better decomposition.

\begin{proposition}
    Each cycle-free transport path $T\in Path(\mu^-, \mu^+)$ has at least a better decomposition.
\end{proposition}
\begin{proof}
 By definition, each cycle-free transport path is acyclic and hence has a good decomposition. By Theorem \ref{thm: GoodS_ij}, it has a better decomposition.
\end{proof}

\begin{proposition}
Let $T\in Path(\mu^-, \mu^+)$ be a cycle-free transport path, and let $\eta$ be a better decomposition of $T$. For each $y_j \in \{y_1,y_2,\ldots,y_N\}$, denote
\begin{equation}
 \label{eqn: X_j}   
X_j(\eta) := \{x_i\in X : \eta(\Gamma_{x_i,y_j})>0\}.
\end{equation}
Then for each pair $1\le j_1<j_2\le N$,
\begin{equation}
\label{eqn: X_j_intersection}
    |X_{j_1}(\eta)\cap X_{j_2}(\eta)|\le 1,
\end{equation}
i.e., the intersection $X_{j_1}(\eta)\cap X_{j_2}(\eta)$ is either empty or a single point.
\end{proposition}

\begin{proof}

Assume $|X_{j_1}(\eta)\cap X_{j_2}(\eta)| > 1$. Then there exist two distinct points $x_{i_1},x_{i_2} \in X_{j_1}(\eta)\cap X_{j_2}(\eta)$ with $i_1<i_2$. Thus, 
\begin{equation}
\label{eqn: eta_positive}
    \eta (\Gamma_{x_{i_1},y_{j_1}}) >0,\   
\eta (\Gamma_{x_{i_1},y_{j_2}}) >0,\ 
\eta (\Gamma_{x_{i_2},y_{j_1}}) >0,\ 
\text{ and }\eta (\Gamma_{x_{i_2},y_{j_2}}) >0. 
\end{equation}
By (\ref{eqn: equalsgn}), this implies that $C[(i_1,j_1),(i_2,j_2), \eta]$ defined in (\ref{eqn: C_def}) is a cycle. 
Since $\eta$ is a better decomposition of $T$, by (\ref{eqn: eta_positive}), it follows that $C[(i_1,j_1),(i_2,j_2), \eta]$ is non-vanishing. Pick 
$$0 < \epsilon_0 \le \frac{1}{4} 
\min\{
\eta(\Gamma_{x_{i_1}, y_{j_1}}), \eta(\Gamma_{x_{i_1}, y_{j_2} }), \eta(\Gamma_{x_{i_2}, y_{j_1}}), \eta(\Gamma_{x_{2}, y_{j_2}}) \},$$
and observe that 
$$S = \epsilon_0 \cdot C[(i_1,j_1),(i_2,j_2), \eta]  $$
is a non-vanishing cycle on $T$. Indeed, assume $T = \underline{\underline{\tau}}(M,\theta,\xi)$ and $S=\underline{\underline{\tau}}(N,\phi,\zeta)$, then $N \subseteq M$ and for $\mathcal{H}^1$-a.e. $x$,
\begin{eqnarray*}
\phi(x) &\le &
\epsilon_0\left( 
\frac{\eta\lfloor_{\Gamma_{x_{i_1}, y_{j_1}}}}{\eta(\Gamma_{x_{i_1}, y_{j_1}})}  + 
\frac{\eta\lfloor_{\Gamma_{x_{i_1}, y_{j_2}}}  }{\eta(\Gamma_{x_{i_1}, y_{j_2}})} + 
\frac{\eta\lfloor_{\Gamma_{x_{i_2}, y_{j_1}}} }{\eta(\Gamma_{x_{i_2}, y_{j_1}})}  + 
\frac{\eta\lfloor_{\Gamma_{x_{i_2}, y_{j_2}}}  }{\eta(\Gamma_{x_{i_2}, y_{j_2}})} 
\right) \left(\{\gamma \in \Gamma : x \in \mathrm{Im}(\gamma) \}  \right) \\
& \le &
\epsilon_0 \left( 
\frac{1}{\eta(\Gamma_{x_{i_1}, y_{j_1}})}+ 
\frac{1}{\eta(\Gamma_{x_{i_1}, y_{j_2}})}+ 
\frac{1}{\eta(\Gamma_{x_{i_2}, y_{j_1}})}+ 
\frac{1}{\eta(\Gamma_{x_{i_2}, y_{j_2}})}
\right) \eta \left(\{\gamma \in \Gamma : x \in \mathrm{Im}(\gamma) \}   \right) \\
& \le &
\eta\left(\{\gamma \in \Gamma : x \in \mathrm{Im}(\gamma) \}  \right) = \theta(x),
\end{eqnarray*}
by equation (\ref{eqn: theta(x)}). 
This shows that $S$ is a non-vanishing cycle on $T$. A contradiction with $T$ is cycle-free.
\end{proof}

\begin{theorem}
\label{thm: decomposition} 
Let $T$ be a cycle-free transport path from $\mu^-$ to $\mu^+$, where $\mu^-$ and $\mu^+$ are given in (\ref{eqn: measures}). Then there exists a decomposition
\begin{equation}
\label{eqn: decomposition_T}
T=\sum_{j=0}^NT_j
\end{equation}
such that
\begin{itemize}
\item[(a)] The set $\{x_1,x_2,\cdots, x_M\}$ can be expressed as the disjoint union of its subsets 
        $\{B_j\}_{j=0}^N$ with the cardinality $|B_0|\le \binom{N}{2} $;
\item[(b)] For each $j=1,2,\cdots, N$, $T_j$ is a single-target transport path from 
        \[\mu_j^-:=\mu^-\lfloor_{B_j} \text{ to } \mu_j^+=\tilde{m}_j\delta_{y_j}\]
        for some $0\le \tilde{m}_j:=\mu^-(B_j) \le m_j$. Each $T_j$ is a subcurrent of $T$.
\item[(c)] $T_0$ is a transport path from 
         \[\mu_0^-:=\mu^-\lfloor_{B_0} \text{ to } \mu_0^+=\sum_{j=1}^N (m_j-\tilde{m}_j)\delta_{y_j}.\]
         $T_0$ is also a subcurrent of $T$.
\end{itemize}

\end{theorem}

Note that, by Theorem \ref{thm: decomposition} , it follows that
\begin{equation}
    \label{eqn: decompositions}
    \mu^-=\sum_{j=0}^N\mu_{j}^- \  \text{ and }\mu^+=\sum_{j=0}^N\mu_{j}^+.
\end{equation}

\begin{proof}
Let $\eta$ be a better decomposition of $T$, and $X_j(\eta)$ be the set as defined in (\ref{eqn: X_j}).
Denote
\begin{equation}
    \label{eqn: B_0}
    B_0:=\bigcup_{1\le j_1<j_2\le N}\left(X_{j_1}(\eta)\cap X_{j_2}(\eta)\right)
\end{equation}
and for each $1\le j\le N$, denote
\[B_j:=X_j(\eta)\setminus B_0.\]
Then $\{B_j\}_{j=0}^N$ are pairwise disjoint. Moreover, by (\ref{eqn: X_j_intersection}), $|B_0|\le \binom{N}{2}$.

Define
\[T_0:=\sum_{j=1}^N  \sum_{x_i \in B_0} \int_{\Gamma_{x_i,y_j} }  I_\gamma \, d\eta, \]
and for each $1\le j\le N$, denote
\[T_j:=  \sum_{x_i \in B_j} \int_{\Gamma_{x_i,y_j} }  I_\gamma \, d\eta. \]
Then each $T_j$ is a subcurrent of $T$ for $0\le j \le N$ and
\begin{eqnarray*}
T &=& \sum_{j=1}^N \sum_{i=1}^M \int_{\Gamma_{x_i,y_j} }  I_\gamma \, d\eta
=
\sum_{j=1}^N \left(  \sum_{x_i \in B_j} \int_{\Gamma_{x_i,y_j} }  I_\gamma \, d\eta + 
\sum_{x_i \in B_0}  \int_{\Gamma_{x_i,y_j} }  I_\gamma \, d\eta \right) \\
&=&
\sum_{j=1}^N  \sum_{x_i \in B_j} \int_{\Gamma_{x_i,y_j} }  I_\gamma \, d\eta 
+
\sum_{j=1}^N  \sum_{x_i \in B_0} \int_{\Gamma_{x_i,y_j} }  I_\gamma \, d\eta \\
&=&
\sum_{j=1}^N T_j    \ + T_0=\sum_{j=0}^N T_j.
\end{eqnarray*}
For each $1\le j\le N$, $T_j$ is a single-target transport path with
\[\partial T_j=\sum_{x_i \in B_j} \int_{\Gamma_{x_i,y_j} }  (\delta_{y_j}-\delta_{x_i}) \, d\eta 
=
\left(\sum_{x_i \in B_j}\eta(\Gamma_{x_i,y_j}) \right) \delta_{y_j}- \sum_{x_i \in B_j} \eta(\Gamma_{x_i,y_j}) \delta_{x_i}.\]
Note that when $x_i \in B_j$, since $\{B_k\}$'s are pairwise disjoint, it follows that $\eta (\Gamma_{x_i,y_k})=0$ for all $k\ne j$. So,
\[\sum_{x_i \in B_j} \eta(\Gamma_{x_i,y_j}) \delta_{x_i}=\sum_{x_i \in B_j} \left(\sum_{k=1}^N\eta(\Gamma_{x_i,y_k})\right)\delta_{x_i}=\sum_{x_i \in B_j} \mu^-(\{x_i\})\delta_{x_i}=\mu^-\lfloor_{B_j}=\mu_j^-,\]
and 
\[\left(\sum_{x_i \in B_j}\eta(\Gamma_{x_i,y_j}) \right) \delta_{y_j}=\mu^-(B_j)\delta_{y_j}=\mu_j^+.\]
As a result, $\partial T_j=\mu_j^+-\mu_j^-$.

Moreover, we have the result,
\begin{eqnarray}
\label{eqn: boundary_T_0}
\partial T_0
&=& 
\sum_{j=1}^N  \sum_{x_i \in B_0} \int_{\Gamma_{x_i,y_j} }  (\delta_{y_j}-\delta_{x_i}) \, d\eta    \\
&=& \nonumber
\sum_{j=1}^N  \left(\sum_{x_i \in B_0} \eta(\Gamma_{x_i,y_j})\right) \delta_{y_j}- \sum_{x_i \in B_0}\left(\sum_{j=1}^N \eta(\Gamma_{x_i,y_j})\right) \delta_{x_i}\\
&=& \nonumber
\sum_{j=1}^N  \left(\sum_{x_i \in B_0\cap X_j(\eta)} \eta(\Gamma_{x_i,y_j})\right) \delta_{y_j}- \sum_{x_i \in B_0} \mu^-(\{x_i\})\delta_{x_i}\\
&=& \nonumber
\sum_{j=1}^N  \left(\sum_{x_i \in X_j(\eta)} \eta(\Gamma_{x_i,y_j})-\sum_{x_i \in B_j} \eta(\Gamma_{x_i,y_j})\right) \delta_{y_j}- \mu^-\lfloor_{B_0}
 \\
&=&\nonumber
\sum_{j=1}^N  \left(\sum_{i=1}^M \eta(\Gamma_{x_i,y_j})-\sum_{x_i \in B_j} \eta(\Gamma_{x_i,y_j})\right) \delta_{y_j}- \mu^-\lfloor_{B_0}
 \\
&=& \nonumber
\sum_{j=1}^N  \left(m_j- \mu^-(B_j)\right) \delta_{y_j}- \mu^-\lfloor_{B_0} \\
&=& \nonumber
\mu_0^+-\mu_0^-.
\end{eqnarray}

\end{proof}

\section{Transport Paths induced Transport Maps and Transport Plans} 
In this section, we will decompose a cycle-free transport path into the sum of two transport paths, the first one is induced by a compatible transport map, while the second one is induced by a compatible transport plan.
We first recall the concept of compatibility introduced in \cite[Definition 7.1]{xia1}, and rewrite it in terms of our current contexts.

Suppose $\mu^-$ and $\mu^+$
are two atomic measures of equal finite mass as given in (\ref{eqn: measures}). Let $Path_0(\mu^-, \mu^+)$ denote the family of all cycle-free transport paths from $\mu^-$ to $ \mu^+$.

\begin{remark}
    In \cite[Definition 7.1]{xia1}, we used $Path_0(\mu^-, \mu^+)$ to denote the family of all ``acyclic" transport paths from $\mu^-$ to $\mu^+$. In \cite{xia1}, a transport path $G$  is called ``acyclic" if it satisfies the following condition: {\it for any polyhedral 1-chain $\tilde{G}$ with the support of $\tilde{G}$ contained in the support of $G$, if $\partial \tilde{G}=0$ then $\tilde{G}=0$.} 
In the current context, $G$ is an ``acyclic" transport path simply means that it is cycle-free. 
To avoid confusion between the term ``acyclic" used in \cite{xia1} and the acyclic concept defined using subcurrents in \cite{Paolini}, we opt for the term ``cycle-free" to name the term ``acyclic" used in \cite{xia1}.
\end{remark}

Observe that for any $G\in Path_0(\mu^-, \mu^+)$ and for each $x_i$ and $y_j$, there exists at most one directed polyhedral curve $g_{ij}$  from $x_i$ to $y_j$, supported on the support of $G$.  
Thus, we associate each $G\in Path_0(\mu^-, \mu^+)$
with a $M\times N$ polyhedral 1-chain valued matrix 
$g=
\begin{bmatrix}
I_{g_{ij}}
\end{bmatrix}$, such that $I_{g_{ij}}=0$ when $g_{ij}$ does not exist.

\begin{definition} (\cite[Definition 7.1]{xia1})
\label{def: atomic_compatibility}
Let $G\in Path_0(\mu^-, \mu^+)$ and $q \in Plan(\mu^-, \mu^+)$ with associated matrices $
\begin{bmatrix}
I_{g_{ij}}
\end{bmatrix}$ and $
\begin{bmatrix}
q_{ij}
\end{bmatrix}$ respectively. The pair $(G,q)$ is called {\it compatible} if 
$q_{ij} = 0$ whenever $I_{g_{ij}} = 0$ and
\begin{equation}  
\label{eqn: comptaible_currents}
G = \sum_{i=1}^M \sum_{j=1}^N q_{ij} I_{g_{ij}} \text{ and }
q=\sum_{i=1}^M \sum_{j=1}^N q_{ij}  \delta_{(x_i, y_j)}
\end{equation}
as polyhedral 1-chains.
\end{definition}

\begin{example}

For instance, 
let $$\mu^- = \frac{1}{4} \delta_{x_1} + \frac{3}{4} \delta_{x_2},\ \mu^+ = \frac{5}{8} \delta_{y_1} + \frac{3}{8}\delta_{y_2} ,$$
and consider the following transport plan,
$$q= \frac{1}{8} \delta_{(x_1, y_1)} + \frac{1}{8} \delta_{(x_1, y_2)} +\frac{1}{2} \delta_{(x_2, y_1)} +\frac{1}{4} \delta_{(x_2, y_2)} \in Plan (\mu^-, \mu^+)  .$$
Let $G_1$ and $G_2$ be two transport paths as illustrated in the following figure.

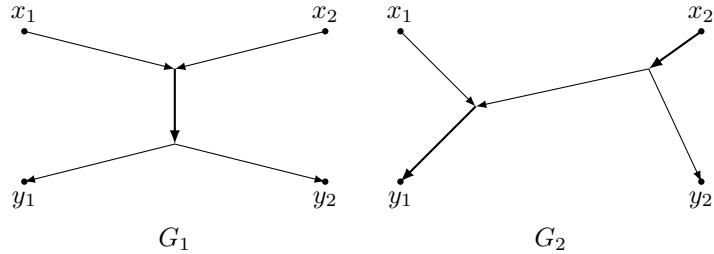
\begin{center}

\begin{tikzpicture} [>=latex]
\filldraw[black] (-2,0) circle (1pt) node[anchor=north]{$y_1$};
\filldraw[black] (2,0) circle (1pt) node[anchor=north] {$y_2$};
\filldraw[black] (-2,2) circle (1pt) node[anchor=south]{$x_1$};
\filldraw[black] (2,2) circle (1pt) node[anchor=south]{$x_2$};

\draw[->] (-2,2) -- (0,1.5);
\draw[->] (2,2) --  (0,1.5);
\draw[thick,->] (0,1.5) --  (0,0.5);
\draw[->] (0,0.5) -- (-2,0);
\draw[->] (0,0.5) --  (2,0);

\filldraw[black] (3,0) circle (1pt) node[anchor=north]{$y_1$};
\filldraw[black] (7,0) circle (1pt) node[anchor=north] {$y_2$};
\filldraw[black] (3,2) circle (1pt) node[anchor=south]{$x_1$};
\filldraw[black] (7,2) circle (1pt) node[anchor=south]{$x_2$};

\draw[->] (3,2) -- (4,1);
\draw[thick,->] (7,2) --  (6.3,1.5);
\draw[->] (6.3,1.5) --  (4,1);
\draw[thick,->] (4,1) -- (3,0);
\draw[->] (6.3,1.5) --  (7,0);
\end{tikzpicture}

\begin{tikzpicture}
\filldraw[black] (0,0) circle (0pt) node[anchor=south]{$G_1$};
\filldraw[black] (5,0) circle (0pt) node[anchor=south]{$G_2$};
\end{tikzpicture}

\end{center}
Then $(G_1, q)$ is compatible but $(G_2, q)$ is not, since $q_{12}=\frac{1}{8}\neq 0$ and  there is no directed curve $g_{12}$  from $x_1$ to $y_2$ on the support of $G_2$.

\end{example}

Now, we generalize the compatibility of atomic measures $\mu^-, \mu^+$ stated above to those of general measures. 
\begin{definition}
\label{def: Radon_compatibility}
Let $\mu$ and $\nu$ be two Radon measures on $X$ of equal total mass. 
Given $T \in Path (\mu,\nu)$, and $\pi \in Plan (\mu,\nu)$, we say the pair $(T, \pi) $ is compatible if there exists a finite Borel measure $\eta$ on $\Gamma$ such that 
$$T = \int_\Gamma  I_\gamma d\eta, \text{ and } \pi = \int_\Gamma \delta_{ ( p_0(\gamma), p_\infty(\gamma) ) } d\eta .$$
Moreover,
given $T \in Path (\mu,\nu)$ and $\varphi \in Map (\mu,\nu)$, we say the pair $(T,\varphi)$ is compatible if $(T, \pi_\varphi)$ is compatible, where 
$\pi_\varphi = (id \times \varphi)_{\#}\mu$.
\end{definition}

The following Proposition says that Definition \ref{def: Radon_compatibility} is a generalization of  Definition \ref{def: atomic_compatibility}.
\begin{proposition}
Let $\mu^-$ and $\mu^+$ be two atomic measures of equal mass as given in (\ref{eqn: measures}).
Let $G\in Path_0 (\mu^-, \mu^+)$ and $q \in Plan(\mu^-, \mu^+)$. Then $(G,q)$ is compatible in the sense of Definition \ref{def: atomic_compatibility} if and only if $(G,q)$ is compatible in the sense of Definition \ref{def: Radon_compatibility}.
    
\end{proposition}

\begin{proof}
Suppose $(G, q)$ is compatible in the sense of Definition \ref{def: atomic_compatibility}.  
By setting
$$\eta = \sum_{i=1}^M \sum_{j=1}^N q_{ij}  \delta_{g_{_{ij}}}$$
over all $\{1\le i\le M, 1\le j\le N\}$ with $g_{ij}$ exists,
equation (\ref{eqn: comptaible_currents}) gives that
$$G = \int_\Gamma  I_\gamma d\eta \text{ and } q = \int_\Gamma \delta_{ ( p_0(\gamma), p_\infty(\gamma) ) } d\eta .$$
Therefore, $(G,q)$ is also compatible in the sense of Definition \ref{def: Radon_compatibility}.

On the other hand, suppose $(G,q)$ is compatible in the sense of Definition \ref{def: Radon_compatibility}, then there exists a Borel measure $\eta$ on $\Gamma $ such that 
$$G = \int_\Gamma  I_\gamma d\eta \text{ and } q = \int_\Gamma \delta_{ ( p_0(\gamma), p_\infty(\gamma) ) } d\eta.  $$
Since $q\in Plan (\mathbf{a}, \mathbf{b})$, we may write 
\[
q= \sum_{i=1}^M \sum_{j=1}^N q_{ij}  \delta_{(x_i, y_j)}\]
for some $q_{ij}\ge 0$.
Denote
\[J_q:=\{(i,j): 1\le i\le M, 1\le j \le N, \text{ with } q_{ij}>0 \}.\]
and
\[\tilde{\Gamma}:=\bigcup_{(i,j)\in J_q} \Gamma_{x_i, y_j}.\]
Since
\[\int_{\Gamma\setminus \tilde{\Gamma}}\delta_{ ( p_0(\gamma), p_\infty(\gamma) ) } d\eta+\int_{ \tilde{\Gamma}} \delta_{ ( p_0(\gamma), p_\infty(\gamma) ) } d\eta =
\int_\Gamma \delta_{ ( p_0(\gamma), p_\infty(\gamma) ) } d\eta= 
q= \sum_{i=1}^M \sum_{j=1}^N q_{ij}  \delta_{(x_i, y_j)}=\sum_{(i,j)\in J_q} q_{ij}  \delta_{(x_i, y_j)}
,\]
it follows that
\[\int_{\Gamma\setminus \tilde{\Gamma}}\delta_{ ( p_0(\gamma), p_\infty(\gamma) ) } d\eta=0 
\text{ and }
\int_{ \tilde{\Gamma}} \delta_{ ( p_0(\gamma), p_\infty(\gamma) ) } d\eta=
\sum_{(i,j)\in J_q} q_{ij}  \delta_{(x_i, y_j)}.
\]
Thus,
$\eta (\Gamma\setminus \tilde{\Gamma} )=0$
and
\[q=\sum_{(i,j)\in J_q}
\int_{ \Gamma_{x_i, y_j}} \delta_{ ( p_0(\gamma), p_\infty(\gamma) ) } d\eta=
\sum_{(i,j)\in J_q} 
q_{ij}  \delta_{(x_i, y_j)}.\]
Hence for each $1\le i\le M, 1\le j \le N$,
\[ \eta(\Gamma_{x_i, y_j})=q_{ij} \text{ if } (i,j)\in J_q \text{ and } \eta(\Gamma_{x_i, y_j})=0 \text{ if not.}
\]
Now, for each $(i,j)\in J_q$, since $\eta(\Gamma_{x_i, y_j})=q_{ij}>0$
and
\[G= \int_\Gamma  I_\gamma d\eta=\sum_{(i,j)\in J_q}\int_{\Gamma_{x_i, y_j}} I_\gamma d\eta,\]
it follows that there exists a polyhedral 1-curve $g_{ij}$ supported on the support of $G$.
Let
\[\tilde{G} = \sum_{(i,j)\in J_q} q_{ij}  I_{g_{_{ij}}},\]
then
\[\partial(G-\tilde{G})=\partial \left(\sum_{(i,j)\in J_q}\int_{\Gamma_{x_i, y_j}} I_\gamma d\eta-\sum_{(i,j)\in J_q} q_{ij}  I_{g_{_{ij}}} \right) =\sum_{(i,j)\in J_q} \left( \eta \left(\Gamma_{x_i, y_j}\right) - q_{ij}  \right) \left( \delta_{y_j}-\delta_{x_i} \right)=0, \] 
so that $G-\tilde{G}$ is a cycle supported on the support of $G$. Since $G\in Path_0(\mu^-, \mu^+)$, we have $G-\tilde{G}=0$. Therefore, 
\[G=\tilde{G}= \sum_{(i,j)\in J_q} q_{ij}  I_{g_{_{ij}}}.\] Note also that whenever $I_{g_{ij}}=0$, it follows that $(i,j)\not\in J_q$, and thus $q_{ij}=0$. As a result, $(G, q)$ is compatible in the sense of Definition \ref{def: atomic_compatibility}.
\end{proof}

By Theorem \ref{thm: decomposition}, we now have the following theorem:
\begin{theorem} 
\label{thm: compatability}
Let $T\in Path(\mu^-,\mu^+)$ be a cycle-free transport path, where $\mu^-$ and $\mu^+$ are given in (\ref{eqn: measures}).
Then there exist 
\begin{itemize}
\item[(a)] decomposition
\[\mu^- = \mu_\pi^- + \mu_\varphi^-, \  \mu^+ = \mu_\pi^+ + \mu_\varphi^+,\text{ with } \mu_\pi^-(X)=\mu_\pi^+(X),\ \mu_\varphi^-(X)=\mu_\varphi^+(X)\  \] 
where $\mu_\pi^-$ and $\mu_\varphi^-$ have disjoint supports and $|spt(\mu_\pi^-)|\le \binom{N}{2}$ with $|A|$ denoting the cardinality of the set $A$; 

\item[(b)] $T=T_{\pi}+T_{\varphi}$ for some $T_\pi \in Path\left(\mu_\pi^-, \mu_\pi^+\right)$ and 
$T_\varphi \in Path\left(\mu_\varphi^-, \mu_\varphi^+ \right)$. Both $T_\pi$ and $T_\varphi$ are subcurrents of $T$;
\item[(c)]   
a transport map $\varphi \in Map\left(\mu_\varphi^{-}, \mu_\varphi^+\right)$ such that $(T_\varphi, \varphi)$ is compatible;   

\item[(d)] 
a transport plan $\pi\in Plan\left(\mu_\pi^-, \mu_\pi^+\right)$ such that $(T_\pi, \pi)$ is compatible;

\item[(e)]
For each $x_i$ with $\mu_\pi^{-} (\{x_i\}) >0$, 
there are at least two $y_{j_1},y_{j_2}$, 
such that 
$$\pi(\{x_i\} \times \{y_{j_1}\})>0, \pi(\{x_i\} \times \{y_{j_2}\})>0.$$
\end{itemize}

\end{theorem}

\begin{proof}
We continue with the same notations used in Theorem \ref{thm: decomposition}. 
Part (a),(b) follows from (\ref{eqn: decomposition_T}) and (\ref{eqn: decompositions}) by setting
$$
\mu_\pi^- := \mu_0^-, \  \mu_\varphi^- := \sum_{j=1}^N \mu_j^-,\ \mu_\pi^+ := \mu_0^+, \  \mu_\varphi^+ := \sum_{j=1}^N \mu_j^+,\ T_\pi := T_0,\ T_\varphi := \sum_{j=1}^N T_j .
$$

For part (c), we define 
$$\varphi := \sum_{j=1}^N y_j \chi_{_{B_j}},$$
where $B_j$'s are subsets of $\{x_1,x_2,\cdots, x_M\}$ given in Theorem \ref{thm: decomposition}.
Since $\mu_j^- = \mu^- \lfloor_{B_j}, \ \mu_j^+ = \tilde{m}_j \delta_{y_j}$, and $B_j$'s are pairwise disjoint for $j=1,2,\ldots,N$, 
we get 
$$\varphi_{\#} (\mu_\varphi^-) = 
\varphi_{\#} \left(\sum_{j=1}^N \mu_j^- \right)=\varphi_{\#} \left(\sum_{j=1}^N \mu^- \lfloor_{B_j}\right)
= \sum_{j=1}^N 
\mu^- (B_j)\delta_{y_j} = \sum_{j=1}^N 
\tilde{m}_j\delta_{y_j}=\mu_\varphi^+.   $$
Therefore, $\varphi$ is a transport map from $\mu_\varphi^-$ to $\mu_\varphi^+$.

We now show that $(T_\varphi, \varphi)$ is compatible. Since 
$$ T_\varphi := \sum_{j=1}^N T_j = \sum_{j=1}^N\sum_{x_i \in B_j} \int_{\Gamma_{x_i,y_j} }  I_\gamma \, d\eta,$$

it is sufficient to show that
\begin{equation}
\label{eqn: pi_varphi}
 \pi_\varphi := 
\left(id \times \varphi\right)_{\#}\mu^- = \sum_{j=1}^N\sum_{x_i \in B_j} \int_{\Gamma_{x_i,y_j} }  \delta_{ (p_0(\gamma)  , p_\infty(\gamma) )}\, d\eta.
\end{equation}

Indeed, for any measurable rectangle $Q\times R$ in $X\times X$,
\begin{eqnarray*}
\pi_\varphi(Q\times R)
&=&
(id \times \varphi )_{\#} \mu^-(Q\times R)= \mu^-(\{x: x\in Q, \varphi(x)\in R\})\\
&=&
\sum_{j=1}^N \mu^-(\{x: x\in Q, \varphi(x)=y_j, y_j\in R\}) 
  =   \sum_{j=1}^N \chi_R(y_j)\mu^-(\{x: x\in Q, \varphi(x)=y_j\})\\
&=&
\sum_{j=1}^N \chi_R(y_j)\mu^-(\{x: x\in Q, x\in B_j\}) 
  =   \sum_{j=1}^N \chi_R(y_j)\mu^-(Q\cap B_j)\\
&=&
\sum_{j=1}^N \chi_R(y_j)\left((p_0)_\#\eta\right)(Q\cap B_j) 
 =  \sum_{j=1}^N \chi_R(y_j)\eta(  p_0^{-1}(Q\cap B_j)  )\\
&=&
\sum_{j=1}^N \chi_R(y_j)\eta(\{\gamma\in \Gamma,  p_0(\gamma)\in Q\cap B_j\})=
\sum_{j=1}^N \chi_R(y_j) \sum_{x_i \in B_j} \int_{\Gamma_{x_i,y_j} }\chi_Q(p_0(\gamma)) d  \eta\\
&=&\sum_{j=1}^N \sum_{x_i \in B_j} \int_{\Gamma_{x_i,y_j} } \chi_Q(p_0(\gamma)) 
 \cdot \chi_R(y_j)d  \eta
=
\sum_{j=1}^N \sum_{x_i \in B_j} \int_{\Gamma_{x_i,y_j} } \chi_Q(p_0(\gamma)) \cdot \chi_R(p_\infty(\gamma))d  \eta\\
&=&
\sum_{j=1}^N \sum_{x_i \in B_j} \int_{\Gamma_{x_i,y_j}, x_i \in Q, y_j \in R }  
\delta_{ p_0(\gamma)} \cdot \delta_{p_\infty(\gamma)} d  \eta 
=
\sum_{j=1}^N\sum_{x_i \in B_j} \int_{\Gamma_{x_i,y_j} }  \delta_{ (p_0(\gamma)  , p_\infty(\gamma) )}\, d\eta(Q\times R).
\end{eqnarray*}
Therefore, (\ref{eqn: pi_varphi}) holds and hence $(T_\varphi, \varphi)$ is compatible.

For part (d), we define 
$$\pi := \sum_{x_i\in B_0} \sum_{j=1}^N \eta\left( \Gamma_{x_i,y_j}  \right) \delta_{(x_i,y_j)} .$$
As shown in (\ref{eqn: boundary_T_0}),
\[\mu_\pi^+-\mu_\pi^-=\mu_0^+-\mu_0^-=
\sum_{j=1}^N  \left(\sum_{x_i \in B_0} \eta(\Gamma_{x_i,y_j})\right) \delta_{y_j}- \sum_{x_i \in B_0}\left(\sum_{j=1}^N \eta(\Gamma_{x_i,y_j})\right) \delta_{x_i}.
\]

This shows that $\pi$ is a transport plan from $\mu_\pi^-$ to $\mu_\pi^+$.
Note that since
\[T_0=\sum_{j=1}^N  \sum_{x_i \in B_0} \int_{\Gamma_{x_i,y_j} }  I_\gamma \, d\eta\]
and
\[
\pi = \sum_{x_i\in B_0} \sum_{j=1}^N \eta \left( \Gamma_{x_i,y_j}  \right) \delta_{(x_i,y_j)} =
\sum_{j=1}^N  \sum_{x_i \in B_0} \int_{\Gamma_{x_i,y_j} }  \delta_{ ( p_0(\gamma), p_\infty(\gamma) ) } \, d\eta,
\]
we have $(T_{\pi}, \pi)$ is compatible.

For part (e), by definition of $\mu_\pi^-$, $x_i \in B_0$ which is defined in Theorem \ref{thm: decomposition}. The result in (e) then follows from the definition of $B_0$ given in (\ref{eqn: B_0}).
\end{proof}

\section{Stair-shaped matrices and Decomposition of stair-shaped transport paths }
In Theorem \ref{thm: compatability}, we decomposed a cycle-free transport path as the sum of a map-compatible path and a plan-compatible path. In this section, we aim to decompose some transport paths as the difference of two map-compatible paths. The family of transport paths that we are interested in are stair-shaped transport paths. To do this, we start with the study of stair-shaped matrices.

\subsection{ Stair-shaped matrices}

\ 

Given $M, N\in \mathbb{N}\cup \{\infty\}$,
let $\mathcal{A}_{M,N}$ denote the collection of all $M\times N$ matrices with non-negative entries.

\begin{definition}
\label{def: stair-shaped}
A matrix $A
\in \mathcal{A}_{M,N}$ is called stair-shaped if there exists two non-decreasing sequences of natural numbers $\{r_1,r_2,\cdots,r_{M+N-1}\}$ and $\{c_1,c_2,\cdots, c_{M+N-1}\}$ with 
$r_k+c_k=k+1$ 
for each $k=1,2,\cdots, M+N-2$, and entries of $A$ that are not located in the positions $\{(r_k, c_k)\}_{k=1}^{M+N-1}$ equal to zero. 
\end{definition}
Note that when $A
\in \mathcal{A}_{M,N}$ is stair-shaped, then $(r_1,c_1)=(1,1)$ and $(r_{M+N-1},c_{M+N-1})=(M,N)$.

\begin{definition}
For each $k=1,2,\cdots, M+N-1$,
a matrix $A\in \mathcal{A}_{M,N}$ is called $k$-stairable if it is in the form of
$$ A = 
\begin{bmatrix}
a_{11} & \cdots &a_{1,c -1} &a_{1,c} &0& \cdots & 0 & \cdots \\
\vdots &        &\vdots &\vdots& \vdots  &  & \vdots \\
a_{r-1,1} & \cdots &a_{r-1,c -1} &a_{r-1,c} &0& \cdots & 0 & \cdots \\
a_{r,1} & \cdots &a_{r, c-1}&a_{r, c} &a_{r, c+1}& \cdots & a_{r, j} & \cdots \\ 
0  & \cdots &0 &a_{r+1, c} &a_{r+1, c+1}& \cdots & a_{r+1, j} & \cdots \\ 
\vdots &  &\vdots &  \vdots   & \vdots &  & \vdots      \\
0 & \cdots &0 &a_{i, c} &a_{i, c+1} &\cdots & a_{i, j} & \cdots \\
\vdots &  &\vdots & \vdots       & \vdots &   & \vdots     \\
\end{bmatrix},
$$
where the leading (i.e., upper left corner) sub-matrix 
$$
\begin{bmatrix}
a_{11} & \cdots &a_{1,c -1} &a_{1,c} \\
\vdots &        &\vdots       &\vdots \\
a_{r-1,1} & \cdots &a_{r_k-1,c -1} &a_{r-1,c} \\
a_{r,1} & \cdots &a_{r,c-1}&a_{r, c} \\ 
\end{bmatrix} $$
is stair-shaped and $k=r+c-1$.
\end{definition}
In particular, each matrix $A\in \mathcal{A}_{M,N}$ is at least $1$-stairable, and 
each stair-shaped matrix $A\in \mathcal{A}_{M,N}$ is  $(M+N-1)$-stairable.

For each $1\le i_1< i_2\le M$ and $1\le j_1< j_2 \le N$, denote
$E [ (i_1,j_1),(i_2,j_2 ) ]$ as the $M \times N$ matrix with $1$ at $(i_1,j_1)$ and $(i_2,j_2)$ entries, with $-1$ at $(i_1,j_2)$ and $(i_2,j_1)$ entries, and $0$ at all other entries.
Each $E [ (i_1,j_1),(i_2,j_2 ) ]$ is called an elementary matrix.

\begin{definition}
For any two matrices $A, B\in \mathcal{A}_{M,N}$, we say $A\cong B$  if there exists a list of real numbers $\{t_k\}_{k=1}^K$ and a list of elementary matrices $\{E_k\}_{k=1}^K$ such that
$B = A + \sum_{k=1}^K t_k E_k$ for some $K\in \mathbb{N} \cup \{ \infty \}$. 
\end{definition}

\begin{theorem}
\label{thm: stairshaped-matrix}
For any matrix $A\in \mathcal{A}_{M,N}$, there exists a stair-shaped matrix $B\in \mathcal{A}_{M,N}$ such that $A \cong B$.
\end{theorem}

\begin{proof}
\textbf{Step 1: }
Let
$$ A = 
\begin{bmatrix}
a_{11} &  a_{12} & \cdots & a_{1j} & \cdots \\
a_{21} &  a_{22} & \cdots & a_{2j} & \cdots \\ 
\vdots &  \vdots &        & \vdots &        \\
a_{i1} &  a_{i2} & \cdots & a_{ij} & \cdots \\
\vdots &  \vdots &        & \vdots &        \\
\end{bmatrix},
$$
and
\[
u_1 = \sum_{i=2}^M a_{i 1}  \text{ and } 
v_1 = \sum_{j=2}^N a_{1 j} .
\]
If $u_1 = 0 $, and since all entries in $A$ are non-negative, then we get 
$$A_1 = A = 
\begin{bmatrix}
a_{11} &  a_{12} & \cdots & a_{1j} & \cdots \\
0      &  a_{22} & \cdots & a_{2j} & \cdots \\ 
\vdots &  \vdots &        & \vdots &        \\
0      &  a_{i2} & \cdots & a_{ij} & \cdots \\
\vdots &  \vdots &        & \vdots &        \\
\end{bmatrix} .$$
If $u_1 \not=0$, and $u_1 \ge v_1$ then we do the following transformation and denote
$$ A_1 = A  +  \sum_{i=2}^\infty \sum_{j=2}^\infty   \frac{a_{i1} a_{1j}}{u_1} 
E [ (1,1),(i,j ) ] .$$
This implies 
\begin{eqnarray*}
A_1 &=& 
\begin{bmatrix}
a_{11} + \sum_{i=2}^\infty \sum_{j=2}^\infty   \frac{a_{i1} a_{1j}}{u_1} &  a_{12} - \sum_{i=2}^\infty   \frac{a_{i1} a_{12}}{u_1}& \cdots & a_{1j} - \sum_{i=2}^\infty   \frac{a_{i1} a_{1j}}{u_1}& \cdots \\
a_{21} -  \sum_{j=2}^\infty   \frac{a_{21} a_{1j}}{u_1}  &  a_{22} + \frac{a_{21} a_{12}}{u_1} & \cdots & a_{2j} +  \frac{a_{21} a_{1j}}{u_1} & \cdots \\ 
\vdots &  \vdots &        & \vdots &        \\
a_{i1} -  \sum_{j=2}^\infty   \frac{a_{i1} a_{1j}}{u_1}  &  a_{i2} +  \frac{a_{i1} a_{12}}{u_1}& \cdots & a_{ij}+ \frac{a_{i1} a_{1j}}{u_1} & \cdots \\
\vdots &  \vdots &        & \vdots &        \\
\end{bmatrix}   \\
&=&
\begin{bmatrix}
a_{11} + v_1 &  0 & \cdots & 0 & \cdots \\
\left(1-\frac{v_1}{u_1}\right) a_{21} &  a_{22} + \frac{a_{21} a_{12}}{u_1} & \cdots & a_{2j} +  \frac{a_{21} a_{1j}}{u_1} & \cdots \\ 
\vdots &  \vdots &        & \vdots &        \\
\left(1-\frac{v_1}{u_1}\right)a_{i1} &  a_{i2} +  \frac{a_{i1} a_{12}}{u_1}& \cdots & a_{ij}+ \frac{a_{i1} a_{1j}}{u_1} & \cdots \\
\vdots &  \vdots &        & \vdots &        \\
\end{bmatrix}   . \\
\end{eqnarray*}
If $u_1 \not= 0$, and $u_1 \le v_1$, we consider the following transformation:
$$A_1 = A + \sum_{i=2}^\infty \sum_{j=2}^\infty   \frac{a_{i1} a_{1j}}{v_1} 
E [ (1,1),(i,j ) ] , $$
and 
\begin{eqnarray*}
A_1 &=&
\begin{bmatrix}
a_{11} + \sum_{i=2}^\infty \sum_{j=2}^\infty   \frac{a_{i1} a_{1j}}{v_1} &  a_{12} - \sum_{i=2}^\infty   \frac{a_{i1} a_{12}}{v_1}& \cdots & a_{1j} - \sum_{i=2}^\infty   \frac{a_{i1} a_{1j}}{v_1}& \cdots \\
a_{21} -  \sum_{j=2}^\infty   \frac{a_{21} a_{1j}}{v_1}  &  a_{22} + \frac{a_{21} a_{12}}{v_1} & \cdots & a_{2j} +  \frac{a_{21} a_{1j}}{v_1} & \cdots \\ 
\vdots &  \vdots &        & \vdots &        \\
a_{i1} -  \sum_{j=2}^\infty   \frac{a_{i1} a_{1j}}{v_1}  &  a_{i2} +  \frac{a_{i1} a_{12}}{v_1}& \cdots & a_{ij}+ \frac{a_{i1} a_{1j}}{v_1} & \cdots \\
\vdots &  \vdots &        & \vdots &        \\
\end{bmatrix}   \\
&=&
\begin{bmatrix}
a_{11} + u_1 &  \left(1-\frac{u_1}{v_1}\right)a_{12} & \cdots & \left(1-\frac{u_1}{v_1}\right)a_{1j} & \cdots \\
0 &  a_{22} + \frac{a_{21} a_{12}}{v_1} & \cdots & a_{2j} +  \frac{a_{21} a_{1j}}{v_1} & \cdots \\ 
\vdots &  \vdots &        & \vdots &        \\
0 &  a_{i2} +  \frac{a_{i1} a_{12}}{v_1}& \cdots & a_{ij}+ \frac{a_{i1} a_{1j}}{v_1} & \cdots \\
\vdots &  \vdots &        & \vdots &        \\
\end{bmatrix} .  \\
\end{eqnarray*}
Hence, $A \cong A_1$ where $A_1$ is of the form:

$$\begin{bmatrix}
a_{11} &  a_{12} & \cdots & a_{1j} & \cdots \\
0      &  a_{22} & \cdots & a_{2j} & \cdots \\ 
\vdots &  \vdots &        & \vdots &        \\
0      &  a_{i2} & \cdots & a_{ij} & \cdots \\
\vdots &  \vdots &        & \vdots &        \\
\end{bmatrix}  
\text{  or }
\begin{bmatrix}
a_{11} &  0      & \cdots & 0      & \cdots \\
a_{21} &  a_{22} & \cdots & a_{2j} & \cdots \\ 
\vdots &  \vdots &        & \vdots &        \\
a_{i1} &  a_{i2} & \cdots & a_{ij} & \cdots \\
\vdots &  \vdots &        & \vdots &        \\
\end{bmatrix} 
$$
and $(r_1,c_1) = (1,1)$. 
Here and in the following steps, for simplicity of notations, we continue using the same notation, $a_{ij}$'s, to denote non-negative entries.

\textbf{Step 2:} 
Set $A_1 = f(A)$, note that $A_1 \cong A$ is $1$-stairable.  
For each $k \in \mathbb{N}$,
if $A_k\cong A$ is $k$-stairable,
we construct a $(k+1)$-stairable matrix $A_{k+1}\cong A$ as follows.
Given
$$ A_{k} = 
\begin{bmatrix}
a_{11} & \cdots &a_{1,c_k -1} &a_{1,c_k} &0& \cdots & 0 & \cdots \\
\vdots &        &\vdots &\vdots& \vdots  &  & \vdots \\
a_{r_k-1,1} & \cdots &a_{r_k-1,c_k -1} &a_{r_k-1,c_k} &0& \cdots & 0 & \cdots \\
a_{r_k1} & \cdots &a_{r_k c_k-1}&a_{r_k c_k} &a_{r_k,c_k+1}& \cdots & a_{r_k,j} & \cdots \\ 
0  & \cdots &0 &a_{r_k+1, c_k} &a_{r_k+1,c_k+1}& \cdots & a_{r_k+1,j} & \cdots \\ 
\vdots &  &\vdots &  \vdots   & \vdots &  & \vdots      \\
0 & \cdots &0 &a_{i,c_k} &a_{i,c_k+1} &\cdots & a_{ij} & \cdots \\
\vdots &  &\vdots & \vdots       & \vdots &   & \vdots     \\
\end{bmatrix},
$$
where the upper left corner sub-matrix 
$$ S = 
\begin{bmatrix}
a_{11} & \cdots &a_{1,c_k -1} &a_{1,c_k} \\
\vdots &        &\vdots       &\vdots \\
a_{r_k-1,1} & \cdots &a_{r_k-1,c_k -1} &a_{r_k-1,c_k} \\
a_{r_k1} & \cdots &a_{r_k c_k}&a_{r_k c_k} \\ 
\end{bmatrix} $$
is stair-shaped (which implies that $r_k+c_k-1=k$), $S \in \mathcal{A}_{r_k,c_k}$,
and let
$$B= f \left(
\begin{bmatrix}
a_{r_k c_k} &a_{r_k,c_k+1}& \cdots & a_{r_k,j} & \cdots \\ 
a_{r_k+1, c_k} &a_{r_k+1,c_k+1}& \cdots & a_{r_k+1,j} & \cdots \\ 
 \vdots   & \vdots &  & \vdots      \\
a_{i,c_k} &a_{i,c_k+1} &\cdots & a_{ij} & \cdots \\
\vdots       & \vdots &   & \vdots     \\
\end{bmatrix} \right) 
=
\begin{bmatrix}
b_{r_k c_k} &b_{r_k,c_k+1}& \cdots & b_{r_k,j} & \cdots \\ 
b_{r_k+1, c_k} &b_{r_k+1,c_k+1}& \cdots & b_{r_k+1,j} & \cdots \\ 
 \vdots   & \vdots &  & \vdots      \\
b_{i,c_k} &b_{i,c_k+1} &\cdots & b_{ij} & \cdots \\
\vdots       & \vdots &   & \vdots     \\
\end{bmatrix} 
.$$
Then we define
$$A_{k+1} = 
\begin{bmatrix}
a_{11} & \cdots &a_{1,c_k -1} &a_{1,c_k} &0& \cdots & 0 & \cdots \\
\vdots &        &\vdots &\vdots& \vdots  &  & \vdots \\
a_{r_k-1,1} & \cdots &a_{r_k-1,c_k -1} &a_{r_k-1,c_k} &0& \cdots & 0 & \cdots \\
a_{r_k1} & \cdots &a_{r_k c_k-1}&b_{r_k c_k} &b_{r_k,c_k+1}& \cdots & b_{r_k,j} & \cdots \\ 
0  & \cdots &0 &b_{r_k+1, c_k} &b_{r_k+1,c_k+1}& \cdots & b_{r_k+1,j} & \cdots \\ 
\vdots &  &\vdots &  \vdots   & \vdots &  & \vdots      \\
0 & \cdots &0 &b_{i,c_k} &b_{i,c_k+1} &\cdots & b_{ij} & \cdots \\
\vdots &  &\vdots & \vdots       & \vdots &   & \vdots     \\
\end{bmatrix}. $$
By definition of $f$, two sequences $(r_k)_{k=1}^\infty$ and $(c_k)_{k=1}^\infty$ can be constructed as follows:
\begin{itemize}
\item [(1)] 
If
$$\begin{bmatrix}
     b_{r_k,c_k+1} & \ldots& b_{r_k,j}&\ldots
    \end{bmatrix} \not=
    \begin{bmatrix}
     0 & \ldots& 0&\ldots
    \end{bmatrix},$$
then $(r_{k+1},c_{k+1})= (r_{k},c_{k} + 1)$;
\item [(2)] 
If 
$$\begin{bmatrix}
     b_{r_k,c_k+1} & \ldots& b_{r_k,j}&\ldots
    \end{bmatrix}=
    \begin{bmatrix}
     0 & \ldots& 0&\ldots
    \end{bmatrix}$$
and 
$$\begin{bmatrix}
b_{r_{k}+1,c_k} & \ldots & b_{i,c_k} & \ldots     
\end{bmatrix}^T \not= 
\begin{bmatrix}
0 & \ldots & 0 & \ldots     
\end{bmatrix}^T   ,$$
then $(r_{k+1},c_{k+1})= (r_{k}+1,c_{k})$;

\item [(3)]
If
$$\begin{bmatrix}
b_{r_k,c_k+1} & \ldots& b_{r_k,j}&\ldots
\end{bmatrix} 
=
\begin{bmatrix}
0 & \ldots& 0&\ldots
\end{bmatrix}$$
and 
$$\begin{bmatrix}
b_{r_{k}+1,c_k} & \ldots & b_{i,c_k} & \ldots     
\end{bmatrix}^T 
= 
\begin{bmatrix}
0 & \ldots & 0 & \ldots     
\end{bmatrix}^T  ,$$
then $(r_{k+1},c_{k+1})= (r_{k},c_{k} + 1)$.
\end{itemize}
This gives $(r_k)_{k=1}^\infty$, $(c_k)_{k=1}^\infty$ are non-decreasing sequences with $r_{k+1} + c_{k+1} = r_k +c_k +1=k+2$.
By doing so, we get a $(k+1)$-stairable matrix $A_{k+1}$ with $A \cong A_{k} \cong A_{k+1}$. 
Note that in this construction we have
\begin{equation}
\label{eqn: matrix_comparsion}  
A_{k+1}(i,j) = A_{k}(i,j), \quad \text{ for } i<r_k \text{ or } j<c_k.
\end{equation}
Moreover,
$$ A_{k+1} = A_{k} + \sum_{l=1}^\infty t_{k,l}  E_{k,l}$$
for some 
$t_{k,l} \in \mathbb{R}$, and $E_{k,l}$'s are elementary matrices.
Set 
$$A_\infty := A_1 + \sum_{k=1}^\infty \sum_{l=1}^\infty t_{k,l} E_{k,l},$$ 
then
$$A_\infty  \cong A_1 \cong A.$$

Note that for each $i=1,\ldots, M$ and $j=1, \dots, N$, by (\ref{eqn: matrix_comparsion}) and $r_k+c_k-1=k$, the sequence $A_k(i,j)$ is eventually constant when $k$ is large enough.  
Thus, $A_\infty (i,j) = \lim_{k\to \infty} A_k(i,j)$ is well defined, stair-shaped, with non-negative entries.
\end{proof}

After knowing the existence of the stair-shaped matrix $B$ using Theorem \ref{thm: stairshaped-matrix}, one may use the following algorithm to recursively find its entries.

\begin{algorithm}
\label{rem: Stair shaped Algorithm}

\ 

{\it Input:} \quad\quad A matrix $A=[a_{ij}]\in \mathcal{A}_{M,N}$.
    
{\it Output:} \quad \  A stair-shaped matrix  $B=[b_{ij}]\in \mathcal{A}_{M,N}$ with $B\cong A$.
    
{\it Algorithm:} One may recursively calculate the entries of $B$ as follows: 
\begin{itemize}
   \item Step 1: Start with $i_0=1, j_0=1$, set
\[R=\sum_{j=1}^N a_{1  j} \text{ and } C=\sum_{i=1}^M a_{i  1}.\]
If $R\leq C$, then $b_{1  1}=R,\ b_{1 j}=0 \text{ for all }j>1$. Otherwise, $b_{1  1}=C$ and  $b_{i 1}=0$ for all $ i>1$.

\ 

\item Step 2:  For each $(i_0,j_0)$ with $b_{i_0, j_0}$ unknown and $b_{ij}$ is known for all $i<i_0$ and $j<j_0$, let 
\[R=\sum_{j=1}^N a_{i_0 , j}-\sum_{j<j_0}b_{i_0 ,j},\; 
C=\sum_{i=1}^M a_{i , j_0}-\sum_{i<i_0}b_{i ,j_0}. \]
If $R\leq C$, set
\[b_{i_0 , j_0}=R,\ b_{i_0 ,j}=0 \text{ for all }j>j_0.\]
Otherwise, when $R>C$, set
\[b_{i_0 , j_0}=C,\ b_{i, j_0}=0 \text{ for all }i>i_0.\]
    
\end{itemize}

\end{algorithm}
Using Step 2 recursively, one can calculate all entries of the stair-shaped matrix $B$.

\subsection{Stair-shaped good decomposition}

\begin{definition}
\label{def: stair shaped matrix corresponds to eta}

Let $\eta$ be a finite measure on $\Gamma$ with $(p_0)_\#\eta=\mu^-$ and $(p_\infty)_\#\eta=\mu^+$.
The representing matrix of $\eta$ is the matrix $A=[a_{ij}]\in \mathcal{A}_{M,N}$ such that $a_{ij}=\eta(\Gamma_{x_i, y_j})$ for each $i, j$. We say that $\eta$ is stair-shaped if its representing matrix $A$ is stair-shaped.
A transport path $T\in Path(\mu^-, \mu^+)$ is called stair-shaped if there exists a good decomposition $\eta$ of $T$ such that $\eta$ is stair-shaped.
\end{definition}

\begin{proposition}
Any stair-shaped good decomposition $\eta$ of $T$ is a better decomposition of $T$.

\end{proposition}

\begin{proof}
By Definition \ref{def: better_decom}, suppose there exist $1\le i_1< i_2\le M$ and $1\le j_1<j_2\leq N$, with 
$$S_{i_1,j_1} (\eta) - S_{i_1,j_2}  (\eta)- S_{i_2,j_1} (\eta) + S_{i_2,j_2} (\eta)=0,$$
then direct calculation from (\ref{eqn: equalsgn}) gives either 
$$\eta(\Gamma_{i_1,j_1}) =\eta(\Gamma_{i_1,j_2})=\eta(\Gamma_{i_2,j_1}) = \eta(\Gamma_{i_2,j_2})=0,$$
or 
$$\eta(\Gamma_{i_1,j_1})> 0, \eta(\Gamma_{i_1,j_2})> 0, \eta(\Gamma_{i_2,j_1})> 0, \eta(\Gamma_{i_2,j_2})> 0.$$
The latter case cannot appear since $\eta$ is stair-shaped and there is no way to align the indexes $$(i_1,j_1), (i_1,j_2),(i_2,j_1),(i_2,j_2),$$
such that both two coordinates are non-decreasing sequences.
As a result, $\eta$ is a better decomposition.

\end{proof}

A stair-shaped path is not necessarily cycle-free. For instance, the transport path $T$ given in Remark \ref{rmk: cycle-free} is stair-shaped because $\eta= \delta_{\gamma_{x_1,y_1}} + \delta_{\gamma_{x_2,y_2}}$ is a stair-shaped good decomposition of $T$. However, $T$ is not cycle-free.

\begin{example}
\label{ex: stair shaped decomposition and measures}
Let $T$ be a transport path from
\[\mu^-=9\delta_{x_1}+9\delta_{x_2}+9\delta_{x_3}+27\delta_{x_4}+27\delta_{x_5} \text{ to } 
\mu^+=36\delta_{y_1} +9\delta_{y_2} + 18\delta_{y_3} +9\delta_{y_4} +9\delta_{y_5} \]
given as shown in the following figure.

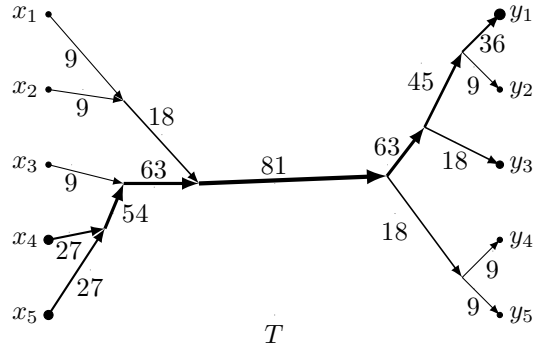
\begin{center}

\begin{tikzpicture} [>=latex]
\filldraw[black] (0,4) circle (1pt) node[anchor=east]{$x_1$};
\filldraw[black] (0,3) circle (1pt) node[anchor=east]{$x_2$};
\filldraw[black] (0,2) circle (1pt) node[anchor=east]{$x_3$};
\filldraw[black] (0,1) circle (1.73pt) node[anchor=east]{$x_4$};
\filldraw[black] (0,0) circle (1.73pt) node[anchor=east]{$x_5$};

\filldraw[black] (6,4) circle (2pt) node[anchor=west]{$y_1$};
\filldraw[black] (6,3) circle (1pt) node[anchor=west]{$y_2$};
\filldraw[black] (6,2) circle (1.4pt) node[anchor=west]{$y_3$};
\filldraw[black] (6,1) circle (1pt) node[anchor=west]{$y_4$};
\filldraw[black] (6,0) circle (1pt) node[anchor=west]{$y_5$};

\filldraw[black] (3,0) circle (0pt) node[anchor=north]{$T$};

\filldraw[black] (0.55,0.6) circle (0pt) node[anchor=north]{$27$};
\filldraw[black] (0.27,1.1) circle (0pt) node[anchor=north]{$27$};
\filldraw[black] (0.3,2) circle (0pt) node[anchor=north]{$9$};
\filldraw[black] (0.45,3) circle (0pt) node[anchor=north]{$9$};
\filldraw[black] (0.3,3.65) circle (0pt) node[anchor=north]{$9$};
\filldraw[black] (1.5,2.4) circle (0pt) node[anchor=south]{$18$};
\filldraw[black] (0.85,1.35) circle (0pt) node[anchor=west]{$54$};
\filldraw[black] (1.4,1.7) circle (0pt) node[anchor=south]{$63$};
\filldraw[black] (3,1.75) circle (0pt) node[anchor=south]{$81$};
\filldraw[black] (4.5,2) circle (0pt) node[anchor=south]{$63$};
\filldraw[black] (4.6,1.35) circle (0pt) node[anchor=north]{$18$};
\filldraw[black] (4.95,3.35) circle (0pt) node[anchor=north]{$45$};
\filldraw[black] (5.9,3.85) circle (0pt) node[anchor=north]{$36$};
\filldraw[black] (5.65,3.3) circle (0pt) node[anchor=north]{$9$};
\filldraw[black] (5.4,2.3) circle (0pt) node[anchor=north]{$18$};
\filldraw[black] (5.9,0.85) circle (0pt) node[anchor=north]{$9$};
\filldraw[black] (5.65,0.35) circle (0pt) node[anchor=north]{$9$};

\filldraw[black] (2,1.75) circle (0pt) ;
\filldraw[black] (4.5,1.85) circle (0pt) ;

\filldraw[black] (5,2.5) circle (0pt) ;
\filldraw[black] (5.5,0.5) circle (0pt) ;

\filldraw[black] (5.5,3.5) circle (0pt) ;

\draw[->] (0,4)--(1,2.85);
\draw[->] (0,3)--(1,2.85);

\draw[->] (0,2)--(1,1.75);

\draw[thick,->] (0,1)--(0.75,1.15);
\draw[thick,->] (0,0)--(0.75,1.15);

\draw[very thick, ->] (0.75,1.15)--(1,1.75);

\draw[semithick, ->] (1,2.85)--(2,1.75);
\draw[very thick, ->] (1,1.75)--(2,1.75);
\draw[ultra thick, ->] (2,1.75)--(4.5,1.85);

\draw[very thick, ->] (4.5,1.85)--(5,2.5);
\draw[semithick, ->] (4.5,1.85)--(5.5,0.5);

\draw[line width=1pt, ->] (5,2.5)--(5.5,3.5);
\draw[semithick, ->] (5,2.5)--(6,2);

\draw[->] (5.5,3.5)--(6,3);
\draw[thick, ->] (5.5,3.5)--(6,4);

\draw[->] (5.5,0.5)--(6,1);
\draw[->] (5.5,0.5)--(6,0);

\end{tikzpicture}

\end{center}
For each $(i,j)$, let $\gamma_{x_i, y_j}\in \Gamma$ be the unique polyhedral curve from $x_i$ to $y_j$ on $T$, and $a_{i,j}$ be the $(i,j)$-entry of the matrix
$$A = 
\begin{bmatrix}
4& 1 & 2 & 1 & 1 \\
4& 1 & 2 & 1 & 1 \\
4& 1 & 2 & 1 & 1 \\
12& 3 & 6 & 3 & 3 \\
12& 3 & 6 & 3 & 3 \\
\end{bmatrix}.
$$
Then
$$\eta_A:=\sum_{ i,j =1}^5 a_{ij}\delta_{\gamma_{x_i, y_j}}$$
is a good but not a better decomposition of $T$. Using Algorithm \ref{rem: Stair shaped Algorithm}, the corresponding stair-shaped matrix of $A$ is given by
$$
B = 
\begin{bmatrix}
9& 0 & 0 & 0 & 0 \\
9& 0 & 0 & 0 & 0 \\
9& 0 & 0 & 0 & 0 \\
9& 9 & 9 & 0 & 0 \\
0& 0 & 9 & 9 & 9
\end{bmatrix}. $$
The corresponding measure 
$$\eta_B:=\sum_{ i,j =1}^5 b_{ij}\delta_{\gamma_{x_i, y_j}}$$
on $\Gamma$ is a stair-shaped good decomposition of $T$, which is automatically a better decomposition of $T$.
\end{example}

The following theorem says that any stair-shaped transport path can be decomposed as the sum of two subcurrents generated by two transport maps.

\begin{theorem}
\label{thm: stair shaped induced transport maps}
Let $T\in Path(\mu^-,\mu^+)$ be a stair-shaped transport path, where $\mu^-$ and $\mu^+$ are given in (\ref{eqn: measures}).
Then there exist decomposition
\[\mu^-=\mu_1^-+\mu_2^-, \mu^+=\mu_1^++\mu_2^+, \text{ and } T=T_1+T_2\]
such that
\begin{itemize}
\item[(a)] for each $i=1,2$, $T_i$ is a subcurrent of $T$ and $T_i\in Path(\mu_i^-, \mu_i^+)$, 
\item[(b)] there exists transport maps $\varphi \in Map(\mu_1^-, \mu_1^+)$ and $\psi\in Map(\mu_2^+, \mu_2^-)$ such that both $( T_1,\varphi)$ and $(-T_2,\psi)$ are compatible.
\end{itemize}
\end{theorem}

\begin{proof}
Since $T$ is stair-shaped, there exists a good decomposition $\eta$ whose representing matrix $A=[a_{ij}]$ is a stair-shaped matrix.
 We now write $A$ as the sum of $B=[b_{ij}]$ and $C=[c_{ij}]$ as follows. For each $i$ and $j$, if $a_{ij}=0$, set $b_{ij}=0$ and $c_{ij}=0$. When $a_{ij}>0$, 
\begin{itemize}
    \item if $a_{ij}$ is the last non-zero entry in the $i$-th row of $A$, (i.e., $a_{ij'}=0$ for all $j' \ge j+1$,) 
    we set $b_{ij}=a_{ij}$ and $c_{ij}=0$; 
    \item if $a_{ij}$ is not the last non-zero entry in the $i$-th row of $A$, since $A$ is stair-shaped, $a_{ij}$ is the last non-zero entry in the $j$-th column of $A$. In this case, we set $b_{ij}=0$ and $c_{ij}=a_{ij}$.
\end{itemize}
By doing so, we write $A=B+C$ such that each row of $B=[b_{ij}]$ and each column of $C=[c_{ij}]$ contain at most one non-zero entry. 
Note that for each $(i,j)$, $a_{ij}=b_{ij}+c_{ij}$ and $a_{ij}>0$ means either $b_{ij}>0$ or $c_{ij}>0$ but not both. Define 
$$\mu_1^- = \sum_{i} \left( \sum_j b_{ij} \right) \delta_{x_i}, \ 
\mu_1^+ = \sum_{j} \left( \sum_i b_{ij} \right) \delta_{y_j}, \ 
\mu_2^- = \sum_{i} \left( \sum_j c_{ij} \right) \delta_{x_i}, \ 
\mu_2^+ = \sum_{j} \left( \sum_i c_{ij} \right) \delta_{y_j}. \ 
$$
Then 
$\mu^- = \mu_1^- + \mu_2^-$ and $ \mu^+ = \mu_1^+ + \mu_2^+$.
Let
$$T_1:=\int_{\{\gamma \in \Gamma_{x_i,y_j}:\; b_{ij}>0\}} I_\gamma \,d\eta,\text{ and }
T_2:=\int_{\{ \gamma \in \Gamma_{x_i,y_j}:\; c_{ij}>0\}} I_\gamma \,d\eta.$$
Both $T_1$ and $T_2$ are subcurrents of $T$, and 
$$\partial T_1 = \int_{\{ \gamma \in \Gamma_{x_i,y_j}:\; b_{ij}>0\}} ( \delta_{y_j}-\delta_{x_i} ) \,d\eta = 
\sum_{i,j} b_{ij} ( \delta_{y_j}-\delta_{x_i} ) = 
\mu_1^+ - \mu_1^- ,
$$
$$\partial T_2 = \int_{\{ \gamma \in \Gamma_{x_i,y_j}:\; c_{ij}>0\}} ( \delta_{y_j}-\delta_{x_i} ) \,d\eta = 
\sum_{i,j} c_{ij} ( \delta_{y_j}-\delta_{x_i} ) = 
\mu_2^+ - \mu_2^- ,
$$
which gives $T_i \in Path(\mu_i^-,\mu_i^+)$ for $i=1,2$.
Then,
$$
T= \int_\Gamma I_\gamma \,d\eta =\int_{\{ \gamma \in \Gamma_{x_i,y_j}:\; a_{ij}>0\}} I_\gamma \,d\eta
=
\int_{\{ \gamma \in \Gamma_{x_i,y_j}:\; b_{ij}>0\}} I_\gamma \,d\eta + 
\int_{\{ \gamma \in \Gamma_{x_i,y_j}:\; c_{ij}>0\}} I_\gamma \,d\eta=T_1 + T_2.    
$$
Denote 
$$X_1 =\{x_i \in X : \mu_1^-(\{x_i\}) > 0\},\ 
Y_1 =\{y_j \in X : \mu_1^+(\{y_j\}) > 0\},$$
$$X_2 =\{x_i \in X : \mu_2^-(\{x_i\}) > 0\},\ 
Y_2 =\{y_j \in Y : \mu_2^+(\{y_j\}) > 0\}.$$
Observe that since $A$ is stair-shaped, by the construction of $b_{ij}$, for each $i$, there exists at most one $j$ (i.e. the largest $j$ with $a_{ij}>0$) such that $b_{ij}>0$. This leads to a map: 
$\varphi: X_1 \to Y_1$ given by
$$\varphi(x_i)=y_j \text{ if } b_{ij}>0.$$
Similarly, for each $j$, there exists at most one $i$ (i.e. the largest $i$ with $a_{ij}>0$) such that $c_{ij}>0$. This leads to a map: 
$\psi: Y_2 \to X_2$ given by
$$\psi (y_j)=x_i \text{ if } c_{ij}>0.$$
By definition of $\varphi$, for each $y_j \in Y_1$,
$$\varphi_{\#}\mu_1^-(\{y_j\}) 
= \mu_1^-(\varphi^{-1}(y_j)) 
= \mu_1^-(\{x_i : b_{ij} >0\})
=\sum_{b_{ij >0}}\mu_1^-(\{x_i\}) 
= \sum_{i} b_{ij}   
= \mu_1^+ (\{y_j\}).$$
Therefore,
$\varphi_{\#}\mu_1^- = \mu_1^+$, and similarly,
$
\mu_2^- = \psi_{\#}\mu_1^+$. 
Also, direct calculation gives
$$\pi_\varphi: = (id \times \varphi)_{\#} \mu_1^- 
= 
\int_{\{ \gamma \in \Gamma_{x_i,y_j}:\; b_{ij}>0\}} \delta_{(x_i,y_j)} \,d\eta, 
\text{ and } 
\pi_\psi: = (id \times \psi)_{\#} \mu_2^+ 
= 
\int_{\{ \gamma \in \Gamma_{x_i,y_j}:\; c_{ij}>0\}} \delta_{(y_j,x_i)} \,d\eta.
$$
Hence, $( T_1,\varphi)$ and $(-T_2,\psi)$ are compatible.
\end{proof}

We now provide an example to illustrate Theorem \ref{thm: stair shaped induced transport maps}.

\begin{example}
\label{ex: decomposing stair-shaped matrix, standard.}
Let $T$, $\mu^-$, $\mu^+$, $A$, $B$, $\eta_A$, $\eta_B$ be the same values as defined in Example \ref{ex: stair shaped decomposition and measures}. 
By Theorem \ref{thm: stair shaped induced transport maps},
we have 
$$
B_1 = 
\begin{bmatrix}
9& 0 & 0 & 0 & 0 \\
9& 0 & 0 & 0 & 0 \\
9& 0 & 0 & 0 & 0 \\
0& 0 & 9 & 0 & 0 \\
0& 0 & 0 & 0 & 9
\end{bmatrix}, \quad
B_2 = 
\begin{bmatrix}
0& 0 & 0 & 0 & 0 \\
0& 0 & 0 & 0 & 0 \\
0& 0 & 0 & 0 & 0 \\
9& 9 & 0 & 0 & 0 \\
0& 0 & 9 & 9 & 0
\end{bmatrix},
$$
so that $B=B_1 + B_2$.
By matrix $B_1$, we get a transport path $T_1$, with  
$$\mu_1^- = 9 \delta_{x_1} + 9 \delta_{x_2} + 9 \delta_{x_3} + 9 \delta_{x_4} + 9 \delta_{x_5}, \ 
\mu_1^+ = 27 \delta_{y_1} + 9 \delta_{y_3} + 9 \delta_{y_5},
$$
and 
$\varphi: \{x_1,x_2,x_3,x_4,x_5\} \to \{y_1,y_3,y_5\},$ such that  
$$\varphi(x_1)=\varphi(x_2)=\varphi(x_3)=y_1,\ \varphi(x_4)=y_3,\ \varphi(x_5)=y_5.$$

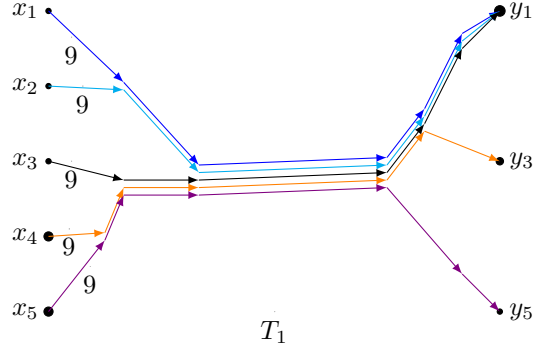
\begin{center}

\begin{tikzpicture} [>=latex]
\filldraw[black] (0,4) circle (1pt) node[anchor=east]{$x_1$};
\filldraw[black] (0,3) circle (1pt) node[anchor=east]{$x_2$};
\filldraw[black] (0,2) circle (1pt) node[anchor=east]{$x_3$};
\filldraw[black] (0,1) circle (1.73pt) node[anchor=east]{$x_4$};
\filldraw[black] (0,0) circle (1.73pt) node[anchor=east]{$x_5$};

\filldraw[black] (6,4) circle (2pt) node[anchor=west]{$y_1$};

\filldraw[black] (6,2) circle (1.4pt) node[anchor=west]{$y_3$};

\filldraw[black] (6,0) circle (1pt) node[anchor=west]{$y_5$};

\filldraw[black] (3,0) circle (0pt) node[anchor=north]{$T_1$};

\filldraw[black] (0.55,0.6) circle (0pt) node[anchor=north]{$9$};
\filldraw[black] (0.27,1.1) circle (0pt) node[anchor=north]{$9$};
\filldraw[black] (0.3,2) circle (0pt) node[anchor=north]{$9$};
\filldraw[black] (0.45,3) circle (0pt) node[anchor=north]{$9$};
\filldraw[black] (0.3,3.65) circle (0pt) node[anchor=north]{$9$};

\draw[blue,->] (0,4)--(1,3.05);
\draw[blue,->] (1,3.05)--(2,1.95);
\draw[blue,->] (2,1.95)--(4.5,2.05);
\draw[blue,->] (4.5,2.05)--(5,2.7);
\draw[blue,->] (5,2.7)--(5.5,3.7);
\draw[blue,->] (5.5,3.7)--(6,4);

\draw[cyan,->] (0,3)--(1,2.95);
\draw[cyan,->] (1,2.95)--(2,1.85);
\draw[cyan,->] (2,1.85)--(4.5,1.95);
\draw[cyan,->] (4.5,1.95)--(5,2.6);
\draw[cyan,->] (5,2.6)--(5.5,3.6);
\draw[cyan,->] (5.5,3.6)--(6,4);

\draw[->] (0,2)--(1,1.75);
\draw[->] (1,1.75)--(2,1.75);
\draw[->] (2,1.75)--(4.5,1.85);
\draw[->] (4.5,1.85)--(5,2.5);
\draw[->] (5,2.5)--(5.5,3.5);
\draw[->] (5.5,3.5)--(6,4);

\draw[orange,->] (0,1)--(0.75,1.05);
\draw[orange, ->] (0.75,1.05)--(1,1.65);
\draw[orange,->] (1,1.65)--(2,1.65);
\draw[orange,->] (2,1.65)--(4.5,1.75);
\draw[orange,->] (4.5,1.75)--(5,2.4);
\draw[orange, ->] (5,2.4)--(6,2);

\draw[violet,->] (0,0)--(0.75,0.95);
\draw[violet,->] (0.75,0.95)--(1,1.55);
\draw[violet,->] (1,1.55)--(2,1.55);
\draw[violet,->] (2,1.55)--(4.5,1.65);
\draw[violet,->] (4.5,1.65)--(5.5,0.5);
\draw[violet,->] (5.5,0.5)--(6,0);

\end{tikzpicture}
    
\end{center}

By matrix $B_2$, we get a transport path $T_2$, with
$$\mu_2^- =  18 \delta_{x_4} + 18 \delta_{x_5}, \ 
\mu_2^+ = 9 \delta_{y_1} + 9 \delta_{y_2} + 9 \delta_{y_3} + 9 \delta_{y_4},
$$
and 
$\psi:\{y_1,y_2,y_3,y_4\} \to \{x_4,x_5\} $, such that 
$$\psi(y_1)=\psi(y_2)=x_4,\ \psi(y_3)=\psi(y_4)=x_5 $$

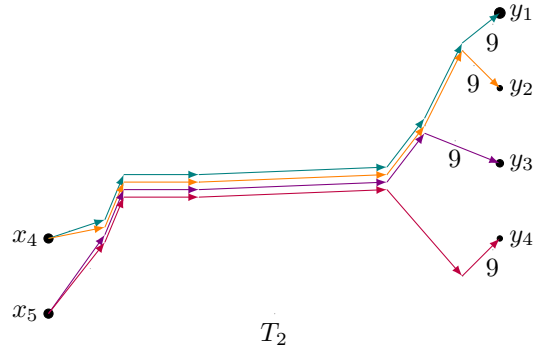
\begin{center}

\begin{tikzpicture} [>=latex]

\filldraw[black] (0,1) circle (1.73pt) node[anchor=east]{$x_4$};
\filldraw[black] (0,0) circle (1.73pt) node[anchor=east]{$x_5$};

\filldraw[black] (6,4) circle (2pt) node[anchor=west]{$y_1$};
\filldraw[black] (6,3) circle (1pt) node[anchor=west]{$y_2$};
\filldraw[black] (6,2) circle (1.4pt) node[anchor=west]{$y_3$};
\filldraw[black] (6,1) circle (1pt) node[anchor=west]{$y_4$};

\filldraw[black] (3,0) circle (0pt) node[anchor=north]{$T_2$};

\filldraw[black] (0.55,0.6) circle (0pt) node[anchor=north]{$ $};
\filldraw[black] (0.27,1.1) circle (0pt) node[anchor=north]{$ $};

\filldraw[black] (5.9,3.85) circle (0pt) node[anchor=north]{$9$};
\filldraw[black] (5.65,3.3) circle (0pt) node[anchor=north]{$9$};
\filldraw[black] (5.4,2.3) circle (0pt) node[anchor=north]{$9$};
\filldraw[black] (5.9,0.85) circle (0pt) node[anchor=north]{$9$};

\filldraw[black] (2,1.75) circle (0pt) ;
\filldraw[black] (4.5,1.85) circle (0pt) ;

\filldraw[black] (5,2.5) circle (0pt) ;
\filldraw[black] (5.5,0.5) circle (0pt) ;

\filldraw[black] (5.5,3.5) circle (0pt) ;

\draw[teal,->] (0,1)--(0.75,1.25);
\draw[teal, ->] (0.75,1.25)--(1,1.85);
\draw[teal,->] (1,1.85)--(2,1.85);
\draw[teal,->] (2,1.85)--(4.5,1.95);
\draw[teal,->] (4.5,1.95)--(5,2.6);
\draw[teal,->] (5,2.6)--(5.5,3.6);
\draw[teal,->] (5.5,3.6)--(6,4);

\draw[orange,->] (0,1)--(0.75,1.15);
\draw[orange, ->] (0.75,1.15)--(1,1.75);
\draw[orange,->] (1,1.75)--(2,1.75);
\draw[orange,->] (2,1.75)--(4.5,1.85);
\draw[orange,->] (4.5,1.85)--(5,2.5);
\draw[orange,->] (5,2.5)--(5.5,3.5);
\draw[orange,->] (5.5,3.5)--(6,3);

\draw[violet,->] (0,0)--(0.75,1.05);
\draw[violet,->] (0.75,1.05)--(1,1.65);
\draw[violet,->] (1,1.65)--(2,1.65);
\draw[violet,->] (2,1.65)--(4.5,1.75);
\draw[violet,->] (4.5,1.75)--(5,2.4);
\draw[violet, ->] (5,2.4)--(6,2);

\draw[purple,->] (0,0)--(0.75,0.95);
\draw[purple,->] (0.75,0.95)--(1,1.55);
\draw[purple,->] (1,1.55)--(2,1.55);
\draw[purple,->] (2,1.55)--(4.5,1.65);
\draw[purple,->] (4.5,1.65)--(5.5,0.5);
\draw[purple,->] (5.5,0.5)--(6,1);

\end{tikzpicture}
    
\end{center}

Then, $T$ is decomposed as the sum of $T_1$ and $T_2$.
\end{example}

\subsection{Cycle-free stair-shaped transport paths}

\ 

To use Theorem \ref{thm: stair shaped induced transport maps}, for a given transport path, one may want to find a stair-shaped good decomposition of it. However, the stair-shaped matrix generated by Algorithm \ref{rem: Stair shaped Algorithm} does not necessarily correspond to a good decomposition, even if we start with a good decomposition, as demonstrated by the following example.

\begin{example}

Let $T$ be the graph given in the following figure, and $\gamma_{i,j}$ be the curve on $T$ from $x_i$ to $y_j$ for each $i,j$.
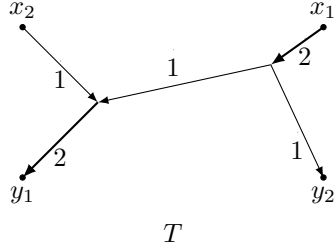
\begin{center}

\begin{tikzpicture} [>=latex]

\filldraw[black] (3,0) circle (1pt) node[anchor=north]{$y_1$};
\filldraw[black] (7,0) circle (1pt) node[anchor=north] {$y_2$};
\filldraw[black] (3,2) circle (1pt) node[anchor=south]{$x_2$};
\filldraw[black] (7,2) circle (1pt) node[anchor=south]{$x_1$};

\filldraw[black] (6.75,1.85) circle (0pt) node[anchor=north]{$2$};
\filldraw[black] (6.65,0.6) circle (0pt) node[anchor=north]{$1$};
\filldraw[black] (5,1.7) circle (0pt) node[anchor=north]{$1$};
\filldraw[black] (3.5,1.55) circle (0pt) node[anchor=north]{$1$};
\filldraw[black] (3.5,0.5) circle (0pt) node[anchor=north]{$2$};

\draw[->] (3,2) -- (4,1);
\draw[thick,->] (7,2) --  (6.3,1.5);
\draw[->] (6.3,1.5) --  (4,1);
\draw[thick,->] (4,1) -- (3,0);
\draw[->] (6.3,1.5) --  (7,0);
\end{tikzpicture}

\begin{tikzpicture}
\filldraw[black] (5,0) circle (0pt) node[anchor=south]{$T$};
\end{tikzpicture}

\end{center}
Then,
$$\eta=\delta_{\gamma_{1,1}} + \delta_{\gamma_{1,2}} + \delta_{\gamma_{2,1}}$$
is a good decomposition of $T$ with the representing matrix
$$
A=[a_{ij}] = 
\begin{bmatrix}
1& 1 \\
1& 0
\end{bmatrix}.
$$
Algorithm \ref{rem: Stair shaped Algorithm} gives the stair-shaped matrix 
$$ B=[b_{ij}] = 
\begin{bmatrix}
2& 0 \\
0& 1
\end{bmatrix}.
$$
However, the corresponding measure, 
$$\eta_B:= 2\delta_{\gamma_{1, 1}} +  \delta_{\gamma_{2, 2} }$$ 
is not a good decomposition of $T$ anymore.

\end{example}

To overcome this issue, we introduce the following concepts:
\begin{definition}
Given $A\in \mathcal{A}_{M,N}$, an elementary matrix $E [ (i_1,j_1),(i_2,j_2 ) ]$ is called admissible to $A$ if $a_{ij}>0$ for all $(i,j)\in \{(i_1,j_1),(i_2,j_2 ), (i_1,j_2),(i_2,j_1 )\}$.
For any two matrices $A, B\in \mathcal{A}_{M,N}$, we say $A\triangleq B$  if there exists a list of real numbers $\{t_k\}_{k=1}^K$ and a list of elementary matrices $\{E_k\}_{k=1}^K$ admissible to $A$ such that
$B = A + \sum_{k=1}^K t_k E_k$ for some $K\in \mathbb{N} \cup \{ \infty \}$.
\end{definition}

\begin{lemma}
\label{lem: eta_B}
Suppose $A$ is the representing matrix of a finite measure $\eta_A$ on $\Gamma$ satisfying $(p_0)_\#\eta_A=\mu^-$ and $(p_\infty)_\#\eta_A=\mu^+$.  For  any matrix $B=[b_{ij}]$ with $A\triangleq B$, define 
\begin{equation}
\label{eqn: eta_B}
\eta_B:=
\sum_{ \substack{i,j \\ \text{ with } a_{ij}>0} }  
\frac{b_{ij}}{a_{ij}}\eta_A\lfloor_{\Gamma_{x_i,y_j}}.
\end{equation}
Then $\eta_B$ is a finite measure
on $\Gamma$ with $(p_0)_\#\eta_B=\mu^-$ and $(p_\infty)_\#\eta_B=\mu^+$. Moreover, $B$ is the representing matrix of $\eta_B$ and $\eta_B \precc \eta_A$.
\end{lemma}

\begin{proof} The condition 
$A \triangleq B$ gives 
$$B = A + \sum_{k} t_k E_k,$$ 
for some real numbers $t_k$ and elementary matrices $E_k = E[(i_k,j_k),(i'_k,j'_k)]$ that are \emph{admissible} to $A$. 

Note that
\begin{eqnarray*}
\eta_B(\Gamma)
&=&
\sum_{ \substack{i,j \\ \text{ with } a_{ij}>0} }  
\frac{b_{ij}}{a_{ij}}\eta_A\lfloor_{\Gamma_{x_i,y_j}}(\Gamma)
=
\sum_{ \substack{i,j \\ \text{ with } a_{ij}>0} }  
\frac{b_{ij}}{a_{ij}}\eta_A(\Gamma_{x_i,y_j}) 
=
\sum_{ \substack{i,j \\ \text{ with } a_{ij}>0} }  
b_{ij} \\
&=&
\sum_{ \substack{i,j \\ \text{ with } a_{ij}>0} }  
\left(a_{ij}+t_k (E_k)_{ij} \right)
=
\sum_{ \substack{i,j \\ \text{ with } a_{ij}>0} }  
a_{ij} 
=
\eta_A(\Gamma)<\infty.
\end{eqnarray*}
Moreover,
\begin{eqnarray*}
(p_0)_{\#}\eta_B
&=&
\sum_{ \substack{i,j \\ \text{ with } a_{ij}>0} }  
\frac{b_{ij}}{a_{ij}}\eta_A(\Gamma_{x_i,y_j})\delta_{x_i}
=
\sum_{ \substack{i,j \\ \text{ with } a_{ij}>0} }  
b_{ij}\delta_{x_i} 
=
\sum_{i}\left(\sum_{ \substack{j \\ \text{ with } a_{ij}>0} }  
\left(a_{ij}+ \sum_{k} t_k (E_k)_{ij} \right)\right)\delta_{x_i} \\
&=&
\sum_{i}\left(\sum_{ \substack{j \\ \text{ with } a_{ij}>0} }  
a_{ij}\right)\delta_{x_i}
=
\sum_{ \substack{i,j \\ \text{ with } a_{ij}>0} }  
a_{ij}\delta_{x_i}=(p_0)_{\#}\eta_A=\mu^-.
\end{eqnarray*}
Similarly, $(p_\infty)_{\#}\eta_B=\mu^+$.

We now show that $B$ is the representing matrix of $\eta_B$, i.e., $\eta_B(\Gamma_{x_{i'},y_{j'}})=b_{i'j'}$ for each pair $(i', j')$.
If $a_{i'j'}=0$, then $\eta_B(\Gamma_{x_{i'},y_{j'}}) =0$ since the sum is over all $a_{ij}>0$. Also, since $E_k$'s are admissible to $A$, this gives $(E_k)_{i' j'} =0$ for all $k$, so that $b_{i'j'}=0=\eta_B(\Gamma_{x_{i'},y_{j'}})$.
If  $a_{i'j'}>0$, then since $\eta_A(\Gamma_{x_{i'},y_{j'}})=a_{i'j'}$, 
$$\eta_B(\Gamma_{x_{i'},y_{j'}}) = 
\sum_{ \substack{i,j \\ \text{ with } a_{ij}>0} }  \frac{b_{ij}}{a_{ij}}\eta_A\lfloor_{\Gamma_{x_i,y_j}}(\Gamma_{x_{i'},y_{j'}}) = b_{i'j'}.$$
Therefore, $B$ is the representing matrix of $\eta_B$.

In the end, we show $\eta_B \precc \eta_A$ by using Lemma \ref{lem: equivalent_definition_precc}. 
Suppose $\eta_B(\Gamma_{x_{i'},y_{j'}}) =  b_{i'j'}>0$, then previous argument gives $a_{i'j'}>0$.
Also, by definition of $\eta_B$,
$$\int_{\Gamma_{x_{i'},y_{j'}} }  I_\gamma d\eta_B = 
\frac{b_{i'j'}}{a_{i'j'}}
\int_{\Gamma_{x_{i'},y_{j'}} }  I_\gamma d\eta_A, \text{ and hence }
\frac{1}{b_{i'j'}}\int_{\Gamma_{x_{i'},y_{j'}} }  I_\gamma d\eta_B = 
\frac{1}{a_{i'j'}}
\int_{\Gamma_{x_{i'},y_{j'}} }  I_\gamma d\eta_A.$$
As a result, $S_{i'j'}(\eta_B) = S_{i'j'}(\eta_A)$ as desired.
\end{proof}

\begin{proposition} 
\label{prop: matrix_good_decomp}
Let  $T$ be a cycle-free transport path from $\mu^-$ to $\mu^+$.
Suppose $\eta_A$ is a good decomposition of $T$, then for any matrix $B=[b_{ij}]$ with $A\triangleq B$, $\eta_B$ given in (\ref{eqn: eta_B}) is also a good decomposition of $T$.
\end{proposition}

\begin{proof}
Let $A=[a_{ij}] \in \mathcal{A}_{M,N}$, $B=[b_{ij}]\in \mathcal{A}_{M,N} $, then $A \triangleq B$ gives 
$$B = A + \sum_{k} t_k E_k,$$ 
for some real numbers $t_k$ and elementary matrices $E_k = E[(i_k,j_k),(i'_k,j'_k)]$ that are \emph{admissible} to $A$.
Using $S_{i,j}(\eta)$ defined in (\ref{eqn: S_ij}), we have
\begin{eqnarray*}
\int_{\Gamma}I_\gamma d(\eta_B-\eta_A)
&=& 
\int_{\Gamma}I_\gamma \,d \left(\sum_{i,j}\frac{b_{ij}}{a_{ij}}\eta_A \lfloor_{\Gamma_{x_i,y_j}} - \sum_{i,j}\eta_A \lfloor_{\Gamma_{x_i,y_j}} \right) \\
&=&
\sum_{i,j} \frac{b_{ij}-a_{ij}}{a_{ij}}\int_{\Gamma_{x_i,y_j}}I_\gamma d\eta_A \\
&=&
\sum_{i,j}(b_{ij}-a_{ij})S_{i,j}(\eta_A)
=
\sum_{k} t_k\sum_{i,j} (E_k)_{ij} S_{i,j}(\eta_A)\\
&=&
\sum_k  t_k \cdot \left(S_{i_k,j_k}(\eta_A) - S_{i_k,j'_k}(\eta_A) -S_{i'_k,j_k}(\eta_A) + S_{i'_k,j'_k}(\eta_A) \right).
\end{eqnarray*}
Since $E_k$'s are admissible to $A$, then $a_{ij} >0$ for $(i,j) \in \{(i_k,j_k)), (i_k,j'_k)), (i'_k,j_k)), (i'_k,j'_k)) \}$.
Since $$S_{i_k,j_k}(\eta_A) - S_{i_k,j'_k}(\eta_A) -S_{i'_k,j_k}(\eta_A) + S_{i'_k,j'_k}(\eta_A)$$  is on $T$ and $a_{ij}>0$, 
direct calculation gives
$$ \partial \left(S_{i_k,j_k} (\eta_A)- S_{i_k,j'_k}(\eta_A) -S_{i'_k,j_k}(\eta_A) + S_{i'_k,j'_k}(\eta_A) \right) =0.$$
By Definition \ref{def: cycle-free current}, $T$ is a cycle-free transport path implies
$$S_{i_k,j_k}(\eta_A) - S_{i_k,j'_k}(\eta_A) -S_{i'_k,j_k}(\eta_A) + S_{i'_k,j'_k}(\eta_A) =0. $$
Hence, 
$$\int_\Gamma  I_\gamma d   \eta_B =\int_\Gamma  I_\gamma d \eta_A.$$

By using an analogous argument as in the proof of Step 1 in Lemma \ref{lem: S_ij(1,1)}, it follows that $\eta_B$ is also a good decomposition of $T$. 
\end{proof}

Given a matrix $A$ with non-negative entries, Theorem \ref{thm: stairshaped-matrix} gives a stair-shaped matrix $B$, such that $A \cong B$, which by definition says $B= A + \sum_{k} t_k E_k$ for some elementary matrices $E_k$.
In general, $A \cong B$ does not imply $A \triangleq B$, since it is possible that some $E_k$'s are not admissible to $A$.
However, when each entries of $A$ is positive (as illustrated in Example \ref{ex: stair shaped decomposition and measures}),  $A \cong B$ implies 
$A \triangleq B$.
In general, when $A$ satisfies certain conditions as stated in the following corollary, we have both $A \cong B$ and $A \triangleq B$, so that the $\eta_B$ in (\ref{eqn: eta_B}) is a stair-shaped good decomposition.

Suppose $A=[a_{ij}]$, let $A[(i_0,j_0),(i'_0,j'_0)]$ be the ``sub-matrix" of $A$ with entries $a_{ij}$'s such that $i_0 \le i \le i'_0$, $j_0 \le j \le j'_0$.

\begin{corollary}
\label{cor: stair shaped blocks}
Let $T$ be a cycle-free transport path from $\mu^-$ to $\mu^+$. Let $A =[a_{ij}]$ be the representing matrix
of a good decomposition $\eta_A$ of $T$. If there exist a list of sub-matrices $A_k = A [(i_k,j_k), (i'_{k},j'_{k})]$ of $A$ such that 
\begin{itemize}
    \item [(a)] $(i_1,j_1) = (1,1)$ and $i'_k \le i_{k+1} \le i'_k+1, \ j'_k \le j_{k+1} \le j'_k+1$ for each $k$,
    \item[(b)] all elements of the sub-matrix $A_k$ are positive for each $k$,
    \item [(c)] all elements of $A$ not in any of the sub-matrices are $0$,
\end{itemize}
then there exists a stair-shaped good decomposition $\eta_B$ of $T$ with $\eta_B \precc \eta_A$. 
Hence, $T$ is stair-shaped.
\end{corollary}

\begin{proof}
We construct the desired stair-shaped matrix by using induction.
We first apply Theorem \ref{thm: stairshaped-matrix} to the sub-matrix $$A_1 = A [(i_1,j_1), (i'_1,j'_1)]$$ and get a stair-shaped $A'_1$.
Then replace entries in $A$ with entries in $A'_1$ in their corresponding original positions in $A$, and denote this new matrix as $B_1$. 
Inductively, for each $k \ge 1$, apply Theorem \ref{thm: stairshaped-matrix} to the sub-matrix
$$B_k [(i_{k+1},j_{k+1}), (i'_{k+1},j'_{k+1})]$$ of $B_k$ and get a stair-shaped $A'_{k+1}$.
Then replace entries in $B_k$ with entries in $A'_{k+1}$ in their corresponding original positions in $B_{k}$, and denote this matrix as $B_{k+1}$. Note that for each $k$, by condition (a), the sub-matrix $B_{k} [(i_{1},j_{1}), (i'_{k},j'_{k})]$ is stair-shaped and 
\begin{equation}
\label{eqn: B_k}
B_{K} [(i_{1},j_{1}), (i'_{k},j'_{k})]= B_{k} [(i_{1},j_{1}), (i'_{k},j'_{k})], \text{ for each } K\ge k+2.
\end{equation}
As a result, for each $(i,j)$, the limit $\lim_{k \to \infty} B_k(i,j)$ exists and equals the value of  $B_k(i,j)$ when $k$ is large enough. 

Let $B$ be the limit matrix of $\{B_k\}$ whose $(i,j)$-entry 
$ B (i,j)   = \lim_{k \to \infty} B_k(i,j)$ for each $(i,j)$. By (\ref{eqn: B_k}), $B[(i_{1},j_{1}), (i'_{k},j'_{k})]= B_{k} [(i_{1},j_{1}), (i'_{k},j'_{k})]$ for each $k$. Since $B_{k} [(i_{1},j_{1}), (i'_{k},j'_{k})]$ is stair-shaped, $B$ is also stair-shaped. Since $B$ is a stair-shaped matrix, its corresponding measure $\eta_B$ as defined in (\ref{eqn: eta_B}) is stair-shaped.
By $(b)$ and definition of \textit{admissible matrices}, we have $A \triangleq B$. 
Therefore, Proposition \ref{prop: matrix_good_decomp} gives $\eta_B$ is a good decomposition with $\eta_B \precc \eta_A$.
\end{proof}

In the end, we provide a typical matrix of finite size satisfying conditions $(a),(b),(c)$ in Corollary \ref{cor: stair shaped blocks}, and see how to decompose the corresponding cycle-free stair-shaped transport path into the difference of two map-compatible paths. 
\begin{example} \label{ex: stair-shaped decomposition}

Let
$$\mu^- = 4 \delta_{x_1} +11 \delta_{x_2} + 14 \delta_{x_3} +11 \delta_{x_4} +17 \delta_{x_5} +10 \delta_{x_6} +3 \delta_{x_7} +6 \delta_{x_8} + 2 \delta_{x_9} + \delta_{x_{10}} +5 \delta_{x_{11}}, $$
$$\mu^+ = 4 \delta_{y_1} + 3 \delta_{y_2} + 14 \delta_{y_3} + 11 \delta_{y_4} + 12 \delta_{y_5} + 7 \delta_{y_6} +7 \delta_{y_7} + 9 \delta_{y_8} + 3 \delta_{y_9} + 3 \delta_{y_{10}} + 11 \delta_{y_{11}},$$
and $T$ be a cycle-free transport path from $\mu^-$ to $\mu^+$ illustrated by the following diagram:

\begin{center}
\begin{tikzpicture} [>=latex,scale=1.22]

\filldraw[black] (1,8) circle (1pt) node[anchor=east] {$x_{1}$};
\filldraw[black] (0,7) circle (1pt) node[anchor=east] {$x_{2}$};
\filldraw[black] (0.5,6) circle (1pt) node[anchor=east] {$x_{3}$};
\filldraw[black] (1.5,4.5) circle (1pt) node[anchor=north] {$x_{4}$};
\filldraw[black] (7.25,8) circle (1pt) node[anchor=west] {$x_{5}$};
\filldraw[black] (8,7) circle (1pt) node[anchor=west] {$x_{6}$};
\filldraw[black] (12,8.5) circle (1pt) node[anchor=west] {$x_{7}$};
\filldraw[black] (11.25,7.25) circle (1pt) node[anchor=west] {$x_{8}$};
\filldraw[black] (12.5,6.5) circle (1pt) node[anchor=west] {$x_{9}$};
\filldraw[black] (11.5,5) circle (1pt) node[anchor=west] {$x_{10}$};
\filldraw[black] (9.25,4.75) circle (1pt) node[anchor=south] {$x_{11}$};

\filldraw[black] (5,3.5) circle (1pt) node[anchor=north] {$y_{1}$};
\filldraw[black] (4,4) circle (1pt) node[anchor=east] {$y_{2}$};
\filldraw[black] (4,5.5) circle (1pt) node[anchor=north] { };
\filldraw[black] (3.8, 5.35) circle (0pt) node[anchor=north] {$y_{3}$};

\filldraw[black] (5,6.5) circle (1pt) node[anchor=south] {$y_{4}$};
\filldraw[black] (5.75,5.75) circle (1pt) node[anchor=north] {$y_{5}$};
\filldraw[black] (7.35,4.75) circle (1pt) node[anchor=north] {$y_{6}$};
\filldraw[black] (7,3.5) circle (1pt) node[anchor=north] {$y_{7}$};
\filldraw[black] (5.75,3.25) circle (1pt) node[anchor=north] {$y_{8}$};
\filldraw[black] (9.25,6.5) circle (1pt) node[anchor=north] {$y_{9}$};
\filldraw[black] (10,5.75) circle (1pt) node[anchor=north] {$y_{10}$};
\filldraw[black] (9,3) circle (1pt) node[anchor=north] {$y_{11}$};

\draw[->] (1,8)--(1.2,7);
\draw[->] (0,7)--(1.2,7);
\draw[thick,->] (1.2,7)--(2.25,5.75);
\draw[very thick, ->]  (2.25,5.75)--(4,5.5);
\draw[thick,->]       (4,5.5)--(4.5,4.25);
\draw[->]             (4.5,4.25)--(5,3.5);
\draw[->]             (4.5,4.25)--(4,4);

\filldraw[black] (1.1,7.5) circle (0pt) node[anchor=east] {$4$};
\filldraw[black] (0.5,6.95) circle (0pt) node[anchor=south] {$11$};
\filldraw[black] (2,6.25) circle (0pt) node[anchor=south] {$15$};
\filldraw[black] (4.25,4.8) circle (0pt) node[anchor=west] {$7$};
\filldraw[black] (4.65,4) circle (0pt) node[anchor=west] {$4$};
\filldraw[black] (4.3,4.15) circle (0pt) node[anchor=north] {$3$};

\draw[->](0.5,6)--(1.6,5.65);
\draw[->](1.5,4.5)--(1.6,5.65);
\draw[thick,->]         (1.6,5.65)--(2.25,5.75);
\draw[thick,->]        (4,5.5)--(4.75,5.85);
\draw[->]              (4.75,5.85)--(5,6.5);
\draw[->]              (4.75,5.85)--(5.75,5.75);

\filldraw[black] (1.1,5.65) circle (0pt) node[anchor=east] {$14$};
\filldraw[black] (1.6,4.95) circle (0pt) node[anchor=east] {$11$};
\filldraw[black] (1.9,5.25) circle (0pt) node[anchor=south] {$25$};
\filldraw[black] (2.95,5.6) circle (0pt) node[anchor=south] {$40$};
\filldraw[black] (4.2,5.6) circle (0pt) node[anchor=south] {$19$};
\filldraw[black] (4.35,6.2) circle (0pt) node[anchor=west] {$11$};
\filldraw[black] (5.15,5.85) circle (0pt) node[anchor=north] {$8$};

\draw[->](7.25,8)--(6.85,6.65);
\draw[->](8,7)--(6.85,6.65);
\draw[thick,->]       (6.85,6.65)--(6.75,5.85);
\draw[->]                    (6.75,5.85)--(5.75,5.75);
\draw[thick,->]                    (6.75,5.85)--(6.5,5);
\draw[->]                             (6.5,5)--(7.35,4.75);
\draw[thick,->]                        (6.5,5)--(6.35,3.9);
\draw[->]                             (6.35,3.9)--(7,3.5);
\draw[->]                           (6.35,3.9)--(5.75,3.25);

\filldraw[black] (7,7.4) circle (0pt) node[anchor=west] {$17$};
\filldraw[black] (7.5,6.9) circle (0pt) node[anchor=north] {$10$};
\filldraw[black] (6.75,6.25) circle (0pt) node[anchor=west] {$27$};
\filldraw[black] (6.25,5.8) circle (0pt) node[anchor=north] {$4$};
\filldraw[black] (6.9,5.65) circle (0pt) node[anchor=north] {$23$};
\filldraw[black] (6.9,4.9) circle (0pt) node[anchor=north] {$7$};
\filldraw[black] (6.6,4.5) circle (0pt) node[anchor=north] {$16$};
\filldraw[black] (6.6,3.75) circle (0pt) node[anchor=north] {$7$};
\filldraw[black] (6.15,3.7) circle (0pt) node[anchor=north] {$9$};

\draw[->](12,8.5)--(11.25,7.25);
\draw[thick,->](11.25,7.25)--(10.025,6.5025);
\draw[->]       (10.025,6.5025)--(9.25,6.5);
\draw[->]       (10.025,6.5025)--(10,5.75);

\filldraw[black] (11.75,8) circle (0pt) node[anchor=north] {$3$};
\filldraw[black] (10.65,6.9) circle (0pt) node[anchor=north] {$6$};
\filldraw[black] (9.6,6.55) circle (0pt) node[anchor=north] {$3$};
\filldraw[black] (10.2,6.35) circle (0pt) node[anchor=north] {$3$};

\draw[->](11.25,7.25)--(11.75,6);
\draw[->](12.5,6.5)--(11.75,6);
\draw[thick,->]         (11.75,6)--(11.5,5);
\draw[thick,->]            (11.5,5)--(9.55,3.75);
\draw[->]            (9.25,4.75)--(9.55,3.75);
\draw[very thick,->]             (9.55,3.75)--(9,3);

\filldraw[black] (11.3,6.8) circle (0pt) node[anchor=north] {$3$};
\filldraw[black] (12.2,6.3) circle (0pt) node[anchor=north] {$2$};
\filldraw[black] (11.8,5.8) circle (0pt) node[anchor=north] {$5$};
\filldraw[black] (10.7,4.5) circle (0pt) node[anchor=north] {$6$};
\filldraw[black] (9.25,4.4) circle (0pt) node[anchor=north] {$5$};
\filldraw[black] (9.5,3.5) circle (0pt) node[anchor=north] {$11$};

\end{tikzpicture}

\begin{tikzpicture}
\filldraw[black] (0,0) circle (0pt) node[anchor=north]{Transport Path $T$};    
\end{tikzpicture}

\end{center}
Let 
\[ A =
\left[
\begin{array}{cccccccccccccccccccc}
1&1&2&0&0&0&0&0&0&0&0  \\
3&2&1&2&3&0&0&0&0&0&0  \\
0&0&6&7&1&0&0&0&0&0&0  \\
0&0&5&2&4&0&0&0&0&0&0  \\
0&0&0&0&1&3&6&7&0&0&0  \\
0&0&0&0&3&4&1&2&0&0&0  \\
0&0&0&0&0&0&0&0&1&2&0  \\
0&0&0&0&0&0&0&0&2&1&3  \\
0&0&0&0&0&0&0&0&0&0&2  \\
0&0&0&0&0&0&0&0&0&0&1  \\
0&0&0&0&0&0&0&0&0&0&5 \\
\end{array} \right].
\]
Then, $A=[a_{ij}]$ is the corresponding matrix of a good decomposition $\eta_A$ of $T$, namely 
$$\eta_A:=\sum_{i,j}a_{ij}\delta_{\gamma_{x_i,y_j}}.$$
Here, $A$ satisfies conditions $(a),(b),(c)$ in Corollary \ref{cor: stair shaped blocks} with
\[ A_1=
\begin{bmatrix}
1&1&2 \\
3&2&1
\end{bmatrix},\ 
A_2=
\begin{bmatrix}
1&2&3 \\
6&7&1 \\
5&2&4
\end{bmatrix},\ 
A_3 = 
\begin{bmatrix}
1&3&6&7   \\
3&4&1&2   \\
\end{bmatrix},\ 
A_4 = 
\begin{bmatrix}
1&2 \\
2&1
\end{bmatrix} \text{, and } 
A_5 = 
\begin{bmatrix}
3 \\
2  \\
1 \\
5 \\
\end{bmatrix}.
\]

Using algorithm \ref{rem: Stair shaped Algorithm}, we have 
\[ A'_1=
\begin{bmatrix}
4&0&0 \\
0&3&3
\end{bmatrix},\ 
A'_2=
\begin{bmatrix}
8&0&0 \\
6&8&0 \\
0&3&8
\end{bmatrix},\ 
A'_3 = 
\begin{bmatrix}
4&7&6&0   \\
0&0&1&9   \\
\end{bmatrix},\ 
A'_4 = 
\begin{bmatrix}
3&0 \\
0&3
\end{bmatrix} \text{, and } 
A'_5 = 
\begin{bmatrix}
3 \\
2  \\
1 \\
5 \\
\end{bmatrix}.
\]
By Corollary \ref{cor: stair shaped blocks},
$$\eta_B:=\sum_{i,j} b_{ij}\delta_{\gamma_{x_i,y_j}}$$
is a stair-shaped good decomposition of $T$ with $\eta_B\precc \eta_A$, where the matrix
\[B =[b_{ij}]=
\left[
\begin{array}{cccccccccccccccccccc}
4&0&0&0&0&0&0&0&0&0&0  \\
0&3&8&0&0&0&0&0&0&0&0  \\
0&0&6&8&0&0&0&0&0&0&0  \\
0&0&0&3&8&0&0&0&0&0&0  \\
0&0&0&0&4&7&6&0&0&0&0  \\
0&0&0&0&0&0&1&9&0&0&0  \\
0&0&0&0&0&0&0&0&3&0&0  \\
0&0&0&0&0&0&0&0&0&3&3  \\
0&0&0&0&0&0&0&0&0&0&2  \\
0&0&0&0&0&0&0&0&0&0&1  \\
0&0&0&0&0&0&0&0&0&0&5 \\
\end{array} 
\right]
\]
is stair-shaped.
 
Now, by the proof of Theorem \ref{thm: stair shaped induced transport maps}, one may decompose the stair-shaped matrix $B$ into $B=B_1+B_2$ where
\[B_1 =
\left[
\begin{array}{cccccccccccccccccccc}
4&0&0&0&0&0&0&0&0&0&0  \\
0&0&8&0&0&0&0&0&0&0&0  \\
0&0&0&8&0&0&0&0&0&0&0  \\
0&0&0&0&8&0&0&0&0&0&0  \\
0&0&0&0&0&0&6&0&0&0&0  \\
0&0&0&0&0&0&0&9&0&0&0  \\
0&0&0&0&0&0&0&0&3&0&0  \\
0&0&0&0&0&0&0&0&0&0&3  \\
0&0&0&0&0&0&0&0&0&0&2  \\
0&0&0&0&0&0&0&0&0&0&1  \\
0&0&0&0&0&0&0&0&0&0&5 \\
\end{array} 
\right] \text{ and }
B_2 =
\left[
\begin{array}{cccccccccccccccccccc}
0&0&0&0&0&0&0&0&0&0&0  \\
0&3&0&0&0&0&0&0&0&0&0  \\
0&0&6&0&0&0&0&0&0&0&0  \\
0&0&0&3&0&0&0&0&0&0&0  \\
0&0&0&0&4&7&0&0&0&0&0  \\
0&0&0&0&0&0&1&0&0&0&0  \\
0&0&0&0&0&0&0&0&0&0&0  \\
0&0&0&0&0&0&0&0&0&3&0  \\
0&0&0&0&0&0&0&0&0&0&0  \\
0&0&0&0&0&0&0&0&0&0&0  \\
0&0&0&0&0&0&0&0&0&0&0 \\
\end{array} 
\right].\]

From matrix $B_1$ and the transport path $T$, we may construct the corresponding transport path  $T_1 \in Path(\mu_1^-,\mu_1^+)$ illustrated below, where
$$\mu_1^- = 4 \delta_{x_1} +8 \delta_{x_2} + 8 \delta_{x_3} +8 \delta_{x_4} +6 \delta_{x_5} +9 \delta_{x_6} +3 \delta_{x_7} +3 \delta_{x_8} + 2 \delta_{x_9} + \delta_{x_{10}} +5 \delta_{x_{11}}, $$
and
$$\mu_1^+ = 4 \delta_{y_1} + 8 \delta_{y_3} + 8 \delta_{y_4} + 8 \delta_{y_5} + 6 \delta_{y_7} + 9 \delta_{y_8} + 3 \delta_{y_9} +  11 \delta_{y_{11}}.$$

\begin{tikzpicture} [>=latex,scale=1.19]

\filldraw[black] (1,8) circle (1pt) node[anchor=east] {$x_{1}$};
\filldraw[black] (0,7) circle (1pt) node[anchor=east] {$x_{2}$};
\filldraw[black] (0.5,6) circle (1pt) node[anchor=east] {$x_{3}$};
\filldraw[black] (1.5,4.5) circle (1pt) node[anchor=north] {$x_{4}$};
\filldraw[black] (7.25,8) circle (1pt) node[anchor=west] {$x_{5}$};
\filldraw[black] (8,7) circle (1pt) node[anchor=west] {$x_{6}$};
\filldraw[black] (12,8.5) circle (1pt) node[anchor=west] {$x_{7}$};
\filldraw[black] (11.25,7.25) circle (1pt) node[anchor=west] {$x_{8}$};
\filldraw[black] (12.5,6.5) circle (1pt) node[anchor=west] {$x_{9}$};
\filldraw[black] (11.5,5) circle (1pt) node[anchor=west] {$ $};

\filldraw[black] (11.6,4.9) circle (0pt) node[anchor=west] {$x_{10}$};

\filldraw[black] (9.25,4.75) circle (1pt) node[anchor=south] {$x_{11}$};

\filldraw[black] (5,3.5) circle (1pt) node[anchor=north] {$y_{1}$};

\filldraw[black] (4,5.5) circle (1pt) node[anchor=north] { };

\filldraw[black] (3.8, 5.35) circle (0pt) node[anchor=north] {$y_{3}$};

\filldraw[black] (5,6.5) circle (1pt) node[anchor=south] {$y_{4}$};
\filldraw[black] (5.75,5.75) circle (1pt) node[anchor=north] {$y_{5}$};

\filldraw[black] (7,3.5) circle (1pt) node[anchor=north] {$y_{7}$};
\filldraw[black] (5.75,3.25) circle (1pt) node[anchor=north] {$y_{8}$};
\filldraw[black] (9.25,6.5) circle (1pt) node[anchor=north] {$y_{9}$};

\filldraw[black] (9,3) circle (1pt) node[anchor=north] {$y_{11}$};

\draw[blue, ->] (1,8)--(1.4,7.2);
\draw[cyan, ->] (0,7)--(1.3,7.1);
\draw[blue, ->] (1.4,7.2)--(2.45,5.95);
\draw[cyan, ->] (1.3,7.1)--(2.35,5.85);
\draw[blue, ->]   (2.45,5.95)--(3.85,5.75);
\draw[cyan, ->]   (2.35,5.85)--(3.75,5.65);

\draw[blue,->]        (4,5.5)--(4.5,4.25);
\draw[blue,->]        (4.5,4.25)--(5,3.5);

\draw[semithick,densely dotted](3.85,5.75)--(4,5.5);
\draw[semithick,densely dotted](3.75,5.65)--(4,5.5);

\filldraw[black] (1.3,7.45) circle (0pt) node[anchor=east] {$4$};
\filldraw[black] (0.5,7) circle (0pt) node[anchor=south] {$8$};
\filldraw[black] (2,6.25) circle (0pt) node[anchor=south] { };
\filldraw[black] (4.25,4.8) circle (0pt) node[anchor=west] { };
\filldraw[black] (4.65,4) circle (0pt) node[anchor=west] {$4$};

\draw[orange,->](0.5,6)--(1.6,5.65);
\draw[violet,->](1.5,4.5)--(1.7,5.55);
\draw[orange,->]   (1.6,5.65)--(2.25,5.75);
\draw[violet,->]   (1.7,5.55)--(2.35,5.65);
\draw[orange,->]   (2.25,5.75)--(3.65,5.55);
\draw[orange]   (3.65,5.55)--(4,5.5);
\draw[violet,->]   (2.35,5.65)--(3.75,5.45);
\draw[semithick,densely dotted](3.75,5.45)--(4,5.5);
\draw[orange,->]   (4,5.5)--(4.75,5.85);
\draw[violet,->]   (4.25,5.47)--(4.85,5.75);
\draw[semithick,densely dotted] (3.75,5.45) .. controls (4,5.25) .. (4.25,5.47);
\draw[orange,->]         (4.75,5.85)--(5,6.5);
\draw[violet,->]         (4.85,5.75)--(5.75,5.75);

\filldraw[black] (1.1,5.65) circle (0pt) node[anchor=east] {$8$};
\filldraw[black] (1.65,4.95) circle (0pt) node[anchor=east] {$8$};
\filldraw[black] (1.9,5.25) circle (0pt) node[anchor=south] { };
\filldraw[black] (2.65,5.6) circle (0pt) node[anchor=south] { };
\filldraw[black] (4.2,5.6) circle (0pt) node[anchor=south] { };
\filldraw[black] (4.5,6.2) circle (0pt) node[anchor=west] {$8$};
\filldraw[black] (5.15,5.75) circle (0pt) node[anchor=north] {$8$};

\draw[magenta,->](7.25,8)--(6.75,6.75);
\draw[brown,->](8,7)--(6.85,6.65);
\draw[magenta,->] (6.75,6.75)--(6.65,5.95);
\draw[brown,->]       (6.85,6.65)--(6.75,5.85);
\draw[magenta,->] (6.65,5.95)--(6.4,5.1);
\draw[brown,->]      (6.75,5.85)--(6.5,5);
\draw[magenta,->] (6.4,5.1) -- (6.25,4);
\draw[brown,->]      (6.5,5)--(6.35,3.9);
\draw[semithick, densely dotted](6.25,4)-- (6.35,3.9);

\draw[magenta,->]       (6.35,3.9)--(7,3.5);
\draw[brown,->]    (6.35,3.9)--(5.75,3.25);

\filldraw[black] (7,7.4) circle (0pt) node[anchor=west] {$6$};
\filldraw[black] (7.5,6.9) circle (0pt) node[anchor=north] {$9$};
\filldraw[black] (6.75,6.25) circle (0pt) node[anchor=west] { };

\filldraw[black] (6.9,5.65) circle (0pt) node[anchor=north] { };

\filldraw[black] (6.6,4.5) circle (0pt) node[anchor=north] { };
\filldraw[black] (6.6,3.75) circle (0pt) node[anchor=north] {$6$};
\filldraw[black] (6.15,3.7) circle (0pt) node[anchor=north] {$9$};

\draw[purple,->](12,8.5)--(11.25,7.25);
\draw[purple,->](11.25,7.25)--(10.025,6.5025);
\draw[purple,->](10.025,6.5025)--(9.25,6.5);

\filldraw[black] (11.75,8) circle (0pt) node[anchor=north] {$3$};
\filldraw[black] (10.65,6.9) circle (0pt) node[anchor=north] {$ $};
\filldraw[black] (9.6,6.55) circle (0pt) node[anchor=north] {$3$};

\draw[teal,->](11.25,7.25)--(11.65,6.1);
\draw[red,->](12.5,6.5)--(11.75,6);
\draw[teal,->]    (11.65,6.1)--(11.4,5.1);

\draw[red,->] (11.75,6)--(11.5,5);

\draw[teal,->]    (11.4,5.1)--(9.45,3.85);

\draw[red,->]    (11.5,5)--(9.55,3.75);
\draw[olive,->]      (9.25,4.75)--(9.35,3.95);
\draw[teal,->]     (9.45,3.85)--(8.9,3.1);
\draw[olive,->]    (9.35,3.95)--(8.8,3.2);
\draw[red,->]     (9.55,3.75)--(9,3);

\draw[semithick,densely dotted](9.35,3.95)--(9.55,3.75)--(9.65,3.65);

\draw[semithick,densely dotted](11.6,4.9)--(11.5,5)--(11.4,5.1);
\draw[->]     (11.6,4.9)--(9.65,3.65);
\draw[->]     (9.65,3.65)--(9.1,2.9);

\draw[semithick,densely dotted](8.8,3.2)--(8.9,3.1)--(9,3)--(9.1,2.9);

\filldraw[black] (11.3,6.8) circle (0pt) node[anchor=north] {$3$};
\filldraw[black] (12.2,6.3) circle (0pt) node[anchor=north] {$2$};
\filldraw[black] (11.8,5.8) circle (0pt) node[anchor=north] {$ $};
\filldraw[black] (10.7,4.35) circle (0pt) node[anchor=north] {$1 $};
\filldraw[black] (9.2,4.7) circle (0pt) node[anchor=north] {$5$};
\filldraw[black] (9.5,3.5) circle (0pt) node[anchor=north] {$ $};

\end{tikzpicture}
\begin{center}
\begin{tikzpicture}
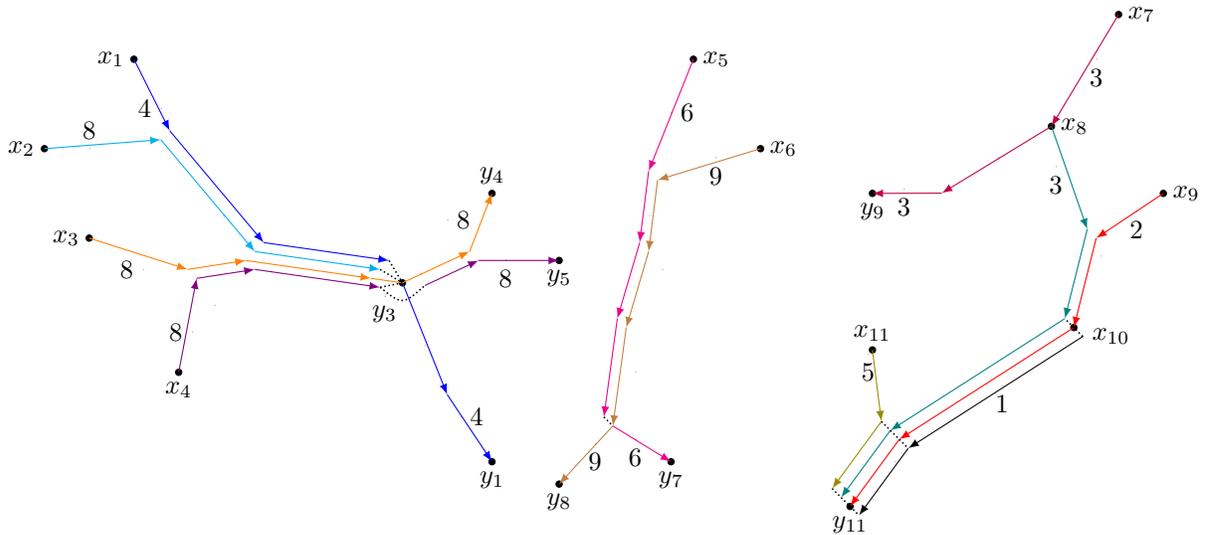
\filldraw[black] (0,0) circle (0pt) node[anchor=north] {Transport Path $T_1$};
    
\end{tikzpicture}   
\end{center}

Note that from the non-zero entries of $B_1$, there exists a transport map
$$\varphi_1: 
\{x_1, x_2,x_3,x_4,x_5,x_6,x_7,x_8,x_9,x_{10},x_{11}\} \longrightarrow 
\{y_1,y_3,y_4,y_5,y_7,y_8,y_9,y_{11}\},$$
where
\begin{eqnarray*}
  &&  \varphi_1(x_1)=y_1,\ 
\varphi_1(x_2)=y_3 ,\ 
\varphi_1(x_3)=y_4 ,\ 
\varphi_1(x_4)=y_5 ,\ 
\varphi_1(x_5)=y_7 ,\ 
\varphi_1(x_6)=y_8,\\
&& \varphi_1(x_7)=y_9 ,\ 
\varphi_1(x_8)=y_{11} ,\ 
\varphi_1(x_9)=y_{11} ,\ 
\varphi_1(x_{10})=y_{11} ,\ 
\varphi_1(x_{11})=y_{11} .
\end{eqnarray*}
Here, $\varphi_{1\#} \mu_1^- = \mu_1^+$, and $(T_1, \varphi_1)$ is compatible.

Similarly, using matrix $B_2$ and transport path $T$, we may construct the corresponding transport path $T_2\in Path(\mu_2^-,\mu_2^+)$ as illustrated below, where 
$$\mu_2^- = 3 \delta_{x_2} +  6 \delta_{x_3} +  3 \delta_{x_4} + 11 \delta_{x_5} +  \delta_{x_6} + 3 \delta_{x_8}, $$
and
$$\mu_2^+ = 3 \delta_{y_2} + 6 \delta_{y_3} + 3 \delta_{y_4} + 4 \delta_{y_5} + 7 \delta_{y_6} +  \delta_{y_7} + 3 \delta_{y_{10}} .$$
\begin{tikzpicture} [>=latex,scale=1.2]

\filldraw[black] (0,7) circle (1pt) node[anchor=east] {$x_{2}$};
\filldraw[black] (0.5,6) circle (1pt) node[anchor=east] {$x_{3}$};
\filldraw[black] (1.5,4.5) circle (1pt) node[anchor=north] {$x_{4}$};
\filldraw[black] (7.25,8) circle (1pt) node[anchor=west] {$x_{5}$};
\filldraw[black] (8,7) circle (1pt) node[anchor=west] {$x_{6}$};

\filldraw[black] (11.25,7.25) circle (1pt) node[anchor=west] {$x_{8}$};

\filldraw[black] (4,4) circle (1pt) node[anchor=east] {$y_{2}$};
\filldraw[black] (4,5.5) circle (1pt) node[anchor=north] { };

\filldraw[black] (3.8, 5.35) circle (0pt) node[anchor=north] {$y_{3}$};

\filldraw[black] (5,6.5) circle (1pt) node[anchor=south] {$y_{4}$};
\filldraw[black] (5.75,5.75) circle (1pt) node[anchor=north] {$y_{5}$};
\filldraw[black] (7.35,4.75) circle (1pt) node[anchor=north] {$y_{6}$};
\filldraw[black] (7,3.5) circle (1pt) node[anchor=north] {$y_{7}$};

\filldraw[black] (10,5.75) circle (1pt) node[anchor=north] {$y_{10}$};

\draw[cyan, ->] (0,7)--(1.3,7.1);
\draw[cyan, ->] (1.3,7.1)--(2.35,5.85);
\draw[cyan, ->]   (2.35,5.85)--(3.75,5.65);

\draw[semithick,densely dotted](3.75,5.65)--(4,5.5);
\draw[cyan,->]       (4,5.5)--(4.5,4.25);
\draw[cyan,->]             (4.5,4.25)--(4,4);

\filldraw[black] (0.5,7) circle (0pt) node[anchor=south] {$3$};
\filldraw[black] (2,6.25) circle (0pt) node[anchor=south] {$ $};
\filldraw[black] (4.25,4.8) circle (0pt) node[anchor=west] {$ $};

\filldraw[black] (4.3,4.15) circle (0pt) node[anchor=north] {$3$};

\draw[orange,->](0.5,6)--(1.6,5.65);
\draw[violet,->](1.5,4.5)--(1.7,5.55);
\draw[orange,->]   (1.6,5.65)--(2.25,5.75);
\draw[violet,->]   (1.7,5.55)--(2.35,5.65);
\draw[orange,->]   (2.25,5.75)--(3.65,5.55);
\draw[orange]   (3.65,5.55)--(4,5.5);
\draw[violet,->]   (2.35,5.65)--(3.75,5.45);
\draw[semithick,densely dotted](3.75,5.45)--(4,5.5);
\draw[violet,->]   (4,5.5)--(4.75,5.85);

\draw[violet,->]         (4.75,5.85)--(5,6.5);

\filldraw[black] (1.1,5.65) circle (0pt) node[anchor=east] {$6$};
\filldraw[black] (1.6,4.95) circle (0pt) node[anchor=east] {$3$};
\filldraw[black] (1.9,5.25) circle (0pt) node[anchor=south] {$ $};
\filldraw[black] (2.95,5.6) circle (0pt) node[anchor=south] {$ $};
\filldraw[black] (4.2,5.6) circle (0pt) node[anchor=south] {$ $};
\filldraw[black] (4.55,6.2) circle (0pt) node[anchor=west] {$3$};

\draw[magenta,->](7.25,8)--(6.75,6.75);

\draw[magenta,->](7.45,7.8)--(7.03,6.75);
\draw[magenta,]  (7.45,7.8)--(6.95,6.55);
\draw[semithick, densely dotted](7.25,8)--(7.45,7.8);

\draw[brown,->](8,7)--(6.85,6.65);
\draw[magenta,->](6.75,6.75)--(6.65,5.95);

\draw[magenta,->]  (6.95,6.55)--(6.85,5.75);

\draw[brown,->]   (6.85,6.65)--(6.75,5.85);
\draw[magenta,->]         (6.65,5.95)--(5.75,5.75);

\draw[magenta,->]         (6.85,5.75)--(6.6,4.9);

\draw[brown,->]     (6.75,5.85)--(6.5,5);
\draw[magenta,->]           (6.6,4.9)--(7.35,4.75);
\draw[brown,->]     (6.5,5)--(6.35,3.9);
\draw[brown,->]            (6.35,3.9)--(7,3.5);

\filldraw[black] (6.6,7.4) circle (0pt) node[anchor=west] {$4$};
\filldraw[black] (7.35,7.4) circle (0pt) node[anchor=west] {$7$};
\filldraw[black] (7.5,6.9) circle (0pt) node[anchor=north] {$1$};
\filldraw[black] (6.75,6.25) circle (0pt) node[anchor=west] {$ $};
\filldraw[black] (6.25,5.8) circle (0pt) node[anchor=north] {$4$};
\filldraw[black] (6.9,5.65) circle (0pt) node[anchor=north] {$ $};
\filldraw[black] (6.9,4.9) circle (0pt) node[anchor=north] {$7$};
\filldraw[black] (6.6,4.5) circle (0pt) node[anchor=north] {$1$};
\filldraw[black] (6.6,3.75) circle (0pt) node[anchor=north] {$1$};

\draw[teal,->](11.25,7.25)--(10.025,6.5025);

\draw[teal,->]       (10.025,6.5025)--(10,5.75);

\filldraw[black] (11.75,8) circle (0pt) node[anchor=north] {$ $};
\filldraw[black] (10.65,6.9) circle (0pt) node[anchor=north] {$3$};
\filldraw[black] (9.6,6.55) circle (0pt) node[anchor=north] {$ $};
\filldraw[black] (10.2,6.35) circle (0pt) node[anchor=north] {$3$};

\end{tikzpicture}

\begin{center}
\begin{tikzpicture}
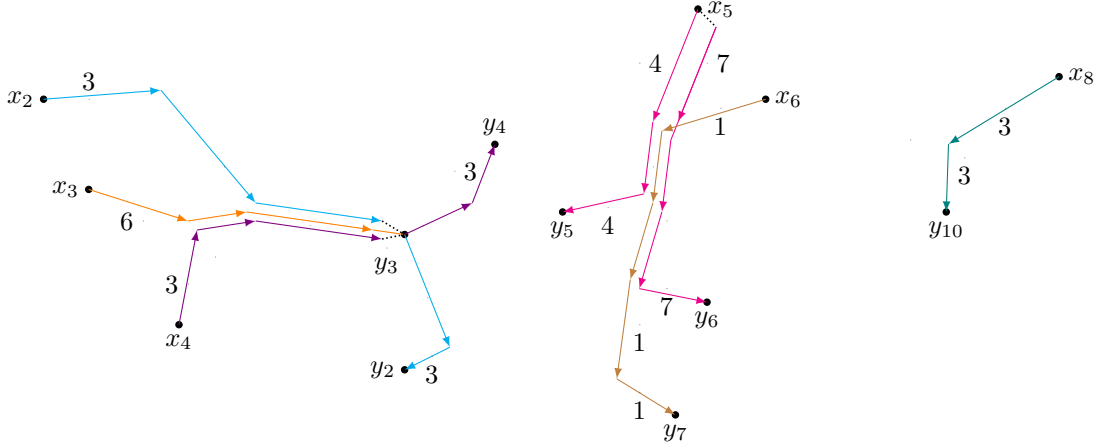
\filldraw[black] (0,0) circle (0pt) node[anchor=north] {Transport Path $T_2$};
    
\end{tikzpicture}   
\end{center}
Again, using the non-zero entries of $B_2$, there exists a transport map
$$\varphi_2: 
\{y_2,y_3,y_4,y_5,y_6,y_7,y_{10}\}
\longrightarrow 
\{x_2,x_3,x_4,x_5,x_6,x_8\}, $$
with 
$$
\varphi_2(y_2)=x_2 ,\ 
\varphi_2(y_3)=x_3 ,\ 
\varphi_2(y_4)=x_4 ,\ 
\varphi_2(y_5)=x_5 ,\ 
\varphi_2(y_6)=x_5 ,\ 
\varphi_2(y_7)=x_6 ,\ 
\varphi_2(y_{10})=x_8 ,\ 
$$
Here, $\mu_2^- = \varphi_{2_\#}\mu_1^+$, and $(-T_2, \varphi_2)$ is compatible. 

As a result, we decompose the cycle-free stair-shaped transport path $T=T_1-T_2$ as the difference of two map-compatible paths $T_1$ and $T_2$.
\end{example}


\begin{thebibliography}{}

\bibitem{Luigi} L. Ambrosio, E. Bru\'e, and D. Semola, \emph{Lectures on Optimal Transport}, Unitext, Volume 130, Springer, 2021.

\bibitem{book}
M. Bernot, V. Caselles, and J.-M. Morel. \emph{Optimal transportation networks. Models and theory}. Lecture Notes in
Mathematics, 1955. Springer, Berlin, 2009.

\bibitem{colombo} M. Colombo, A De Rosa, A, Marchese, Improved stability of optimal traffic paths, \emph{Calc Var Partial Differ Equ}, 57:28, 2018

\bibitem{colombo2020}
M. Colombo, A. De Rosa, and A. Marchese.
On the well-posedness of branched transportation. \emph{Comm. Pure Appl. Math.} 74 (2021), 833-864.




\bibitem{Lin} F. Lin, X. Yang, \textit{Geometric Measure Theory: An Introduction}, Science Press \& International Press, 2002.

\bibitem{Paolini} E. Paolini, E. Stepanov. 
Decomposition of acyclic normal currents in a metric space,
\emph{J Funct Anal},
Vol. 263, Issue 11,
(2012), 3358-3390.

\bibitem{Simon} L. Simon, \emph{Introduction to Geometric Measure Theory}. 2014.

\bibitem{santambrogio} F. Santambrogio, \emph{Optimal Transport for Applied Mathematicians}, Springer, 2015.

\bibitem{smirnov} S. K. Smirnov.  Decomposition of solenoidal vector charges into elementary solenoids, and the structure of normal one-dimensional flows. \emph{Algebra i Analiz}, 5 (1993), 206-238.

\bibitem{villani} C. Villani, \emph{Optimal transport old and new}, Springer, 2009.  

\bibitem{xia1} Q. Xia, Optimal paths related to transport problems. \emph{Commun. Contemp. Math.} Vol.5, No. 2 (2003), 251-279.


\bibitem{boundary_payoff} Q. Xia and S.F. Xu, Ramified optimal transportation with payoff on the boundary. \emph{SIAM J. MATH. ANAL.} Vol 55, No. 1 (2023), 186-209.

\bibitem{xia2015motivations}  Q. Xia, Motivations, ideas and applications of ramified optimal transportation, \emph{ESAIM Math. Model. Numer. Anal.} Vol.
49 No. 6 (2015), 1791--1832.


\end{thebibliography}
\end{document}